\documentclass{amsart}
\title{Notions of amalgamation for AECs and categoricity}
\author{Hanif Joey Cheung}
\date{\today}
\usepackage{amsmath}
\usepackage{amsfonts}
\usepackage{amssymb}
\usepackage{amsthm}
\usepackage{mathrsfs}
\usepackage[mathscr]{euscript}
\usepackage{leftidx}
\usepackage[toc,page]{appendix}
\usepackage{tikz-cd}
\usepackage{mathdots}
\usepackage[sorting=nty,style=alphabetic, maxnames=6]{biblatex}
\theoremstyle{plain}
\newtheorem{thm}{Theorem}[section]
\newtheorem{lem}[thm]{Lemma}
\newtheorem{prop}[thm]{Proposition}
\newtheorem{cor}[thm]{Corollary}
\newtheorem{fac}[thm]{Fact}
\newtheorem{hyp}[thm]{Hypothesis}
\newtheorem*{thm*}{Theorem}
\newtheorem{con}[thm]{Conclusion}

\theoremstyle{definition}
\newtheorem{defn}[thm]{Definition}
\newtheorem{ex}[thm]{Example}
\newtheorem{nota}[thm]{Notation}
\newtheorem{que}[thm]{Question}

\theoremstyle{remark}
\newtheorem*{rem}{Remark}
\newtheorem*{clm}{Claim}

\newbox\noforkbox \newdimen\forklinewidth
\setbox0\hbox{$\textstyle\smile$}
\setbox1\hbox to \wd0{\hfil\vrule width \forklinewidth depth-2pt
	height 10pt \hfil}
\setbox\noforkbox\hbox{\lower 2pt\box1\lower 2pt\box0\relax}
\def\unionstick{\mathop{\copy\noforkbox}\limits}

\def\nonfork_#1{\unionstick_{\textstyle #1}}

\setbox0\hbox{$\textstyle\smile$}
\setbox1\hbox to \wd0{\hfil{\sl /\/}\hfil}
\setbox2\hbox to \wd0{\hfil\vrule height 10pt depth -2pt width
	\forklinewidth\hfil}
\newbox\doesforkbox
\setbox\doesforkbox\hbox{\lower 2pt\box1 \lower 2pt\box2\lower2pt\box0\relax}
\def\nunionstick{\mathop{\copy\doesforkbox}\limits}

\def\fork_#1{\nunionstick_{\textstyle #1}}

\newcommand{\gtp}{\text{gtp}}
\newcommand{\otp}{\text{otp}}
\newcommand{\LS}{\text{LS}}
\newcommand{\dnf}{\unionstick}
\newcommand{\df}{\nunionstick}

\newcommand{\cf}{\text{cf}}
\newcommand{\cl}{\text{cl}}

\tikzset{
	symbol/.style={
		draw=none,
		every to/.append style={
			edge node={node [sloped, allow upside down, auto=false]{$#1$}}}
	}
}

\newcommand{\incarrow}[1]{\arrow[#1, "\iota"]}
\newcommand{\phanArrow}[1]{\arrow[#1, phantom, "\mathscr{A}"]}
\newcommand{\scrA}{\mathscr{A}}

\begin{document}

\maketitle

\begin{abstract}
    Motivated by the free products of groups, the direct sums of modules, and Shelah's $(\lambda,2)$-goodness, we study strong amalgamation properties in Abstract Elementary Classes. Such a notion of amalgamation consists of a selection of certain amalgams for every triple $M_0\leq M_1, M_2$, and we show that if $K$ designates a unique strong amalgam to every triple $M_0\leq M_1, M_2$, then $K$ satisfies categoricity transfer at cardinals $\geq\theta(K)+2^{\LS(K)}$, where $\theta(K)$ is a cardinal associated with the notion of amalgamation. We also show that if such a unique choice does not exist, then there is some model $M\in K$ having $2^{|M|}$ many extensions which cannot be embedded in each other over $M$. Thus, for AECs which admit a notion of amalgamation, the property of having unique amalgams is a dichotomy property in the sense of Shelah's classification theory.
\end{abstract}

    In \cite{GrLe}, Grossberg and Lessman derived a forking-like independence relation on an arbitrary pregeometry $(X, \cl)$, and showed that many of the defining properties of forking in a stable theory are also satisfied in this setting. A natural question to ask then is how to define such a relation on an AEC $(K,\leq)$ where each model has a natural pregeometry which is also coherent with $K$ (see section 7 for a formal definition of these notions), and what properties such a relation would satisfy (in comparison to, for example, a stable independence relation on some monster model of $K$). In particular, the property of types having a unique nonforking extension holds if there is a unique (up to isomorphism) amalgam of models $M_1, M_2$ over $M_0$ where the image of $M_1$ and $M_2$ are independent over $M_0$ (with respect to the pregeometry of the amalgam). This suggests that for such a class, the structure of the class depends not only on assuming the amalgamation property, but how ``well-behaved" the collection of such independent amalgams are. This is reminiscent of the ``stable amalgams" first introduced by Shelah in \cite{Sh83a} and \cite{Sh83b}, where the fact that stable amalgams can be extended and are preserved under continuous chains is used to construct a model in a higher cardinality; in both cases, we are only interested in certain ``nice" amalgams, but the collection of nice amalgams have certain extendibility, continuity, and/or uniqueness properties which is needed for the analysis. This idea of selecting amalgams also underlies the body of literature concerning amalgamation of independent sets/types/diagrams, which arguably began with Shelah's definition of the Dimensional Order Property in \cite{Sh82}, and has developed in multiple directions (including the study of the homology of such diagrams by Goodrick, Kim, and Kolesnikov in \cite{GKK13} and the extension to AECs by Shelah and Vasey in \cite{SV18}).
    
    This selection of certain amalgams as ``nice" is of course a common feature in algebra, such as the construction of free amalgams of groups or direct sums of modules. In a more general setting, Eklof presented in \cite{Ek08} an abstract notion of ``freeness" for a class of modules (building on Shelah's singular compactness theorem from \cite{Sh75} and further work by Hodge in \cite{Ho81}); this class of ``free" models is characterised by each model having an associated ``basis" which are extended under strong embeddings. Whilst this is a powerful abstraction from an algebraic point of view, this is somewhat problematic model-theoretically because it presumes that each model is generated by its basis, and ``generation" translates poorly to non-algebraic contexts. On the other hand, this notion of ``freeness" can also be understood in terms of designating an amalgam $N$ of models $M_0\leq M_1, M_2$ as a ``free amalgam" iff there is a basis of $N$ which is the union of a basis of $M_1$ and a basis of $M_2$ which agree on $M_0$. Doing so allows us to focus on the operation of amalgamating models instead of studying their bases, which bypasses the problem of what ``generation" should mean.
    
    Building on this idea, in this paper we present a framework of a ``notion of amalgamation" for a given AEC. Abstracting from the examples of free amalgamation of groups and direct sum of modules, we isolate the axiomatic properties of absolute minimality, regularity, continuity, and admitting decomposition (Definition \ref{propdef}), which we assume throughout the paper. We also define the uniqueness property of amalgams, which intuitively states that for any triple of models there is a unique amalgam (up to isomorphism) which is ``nice". We refer to a notion satisfying all of the above as a notion of free amalgamation, and establish that when a class $K$ has a notion of free amalgamation and is categorical in a sufficiently large cardinal, then it behaves analogously to the models of a unidimensional first order theory. This allows us to prove a categoricity transfer theorem (Theorem \ref{catmain1}):
    \begin{thm*}
        Suppose $\scrA$ is a notion of free amalgamation in $K$, and $K$ has a prime and minimal model. If $K$ is $\lambda$-categorical in some $\lambda\geq\theta(K)$, then $K$ is $\kappa$-categorical in every cardinal $\kappa\geq\theta(K)+(2^{\LS(K)})^+$.
    \end{thm*}
    (In this formulation, the cardinal $\theta(K)$ is defined from the given notion of free amalgamation, and is analogous to $\kappa(T)$ for $T$ a simple theory.)
    
    Of course, this begs the question of how strong the assumptions above are. In particular, we have mentioned previously that the assumption of unique ``nice" amalgams implies that types have unique nonforking extensions. In fact, like stability, the uniqueness property delineates between structural results on one hand and anti-structural results on the other. This can be seen by combining the above theorem and Theorem \ref{nonunimany}:
    \begin{thm*}
        Suppose $\scrA$ is regular, continuous, absolutely minimal and has weak 3-existence. If $(M_b, M^*, M)$ is a non-uniqueness triple and $p=\gtp(M^*/M_b, M^*)$, then there is $N\geq M$ such that $p$ has $2^{|N|}$-many extensions to $N$.
    \end{thm*}
    
    To tie all of this back to Grossberg and Lessman's investigation of pregeometries, we would like to see how our results apply to a class with pregeometries. In particular, we consider a type $p$ with $U(p)=1$ and $K_p$ the class of realizations of $p$: this class (under some assumptions) is naturally associated with corresponding pregeometries, which allows us to conclude (Theorem \ref{final}):
    \begin{thm*}
        Suppose $K$ admits finite intersections and has a stable independence relation with the $(<\aleph_0)$-witness property. If $U(p)=1$, then $K_p$ is $\lambda$-categorical in all $\lambda>|\text{dom }p|+\+LS(K)$
    \end{thm*}
    Notably, this is analogous to the case of an uncountably categorical countable theory, where the sets $\phi(M)$ for a strongly minimal $\phi(x)$ are also uncountably categorical. This is, of course, a crucial component of the Baldwin-Lachlan proof of Morley's categoricity theorem.
    
    The outline of this paper is as follows: in section 1, we formally define notions of amalgamation for an abstract class, and establish some basic properties which follow from the definition. We introduce some axiomatic properties for notions of amalgamation in section 2, and also explore both examples and counter-examples to these properties.
    
    Section 3 introduces sequential amalgamation, and most of the section is dedicated to proving Theorem \ref{setamal}, which roughly states that when $\scrA$ is well-behaved, then the ordering of the sequence of amalgamation does not affect the $\scrA$-amalgam. We also introduce some notation for amalgams and the cardinal invariant $\mu(K)$. These notions are used in section 4, where an independence relation is defined based on a given notion of amalgamation, and we show that this independence relation behaves similarly to forking in a (super)stable theory.
    
    Section 5 uses the additional assumption that $\scrA$ has uniqueness (as well as some other axiomatic properties introduced in section 2) to show that the class $K$ admits categoricity transfer at cardinals $>\theta(K)+2^{\LS(K)}$, where $\theta(K)$ is a cardinal characteristic derived from the notion of amalgamation. On the other hand, in section 6 we show that failing to have uniqueness implies that there are arbitrarily large models with the maximal number of non-isomorphic (in fact non-biembeddable) extensions. Finally, in section 7 we apply the technology developed to the class $K_p$, which are the realizations of some type $p$ with $U(p)=1$, and show that $K_p$ is necessarily categorical in a tail of cardinals.
    
    This paper was written during the author's Ph.D. program at Carnegie Mellon University. I would like to thank my advisor Rami Grossberg for his feedback and support for this project, and Marcos Mazari Armida for many useful discussions.

\section{Preliminaries}

We first recall some basic definitions regarding abstract elementary classes (AECs) which are found in the literature. A more detailed overview of basic concepts and results can be found in \cite{BaCat}.

\begin{defn}\label{aecdef}
    Let $\tau$ be a language.
    \begin{enumerate}
        \item $(K,\leq_K)$ is an \textbf{abstract class} (in $\tau$) iff:
        \begin{itemize}
            \item $K$ is a class of $\tau$-structures which are closed under $\tau$-isomorphisms
            \item $\leq_K$ is a partial order on $K$, and $M\leq N$ implies that $M$ is a $\tau$-substructure of $N$
            \item The partial order is invariant under isomorphisms: if $M\leq_K N$, $M'\subseteq N'$, $f:M\simeq M'$ and $g:N\simeq N'$ are isomorphisms, and $f\subseteq g$, then $M'\leq_K N'$
        \end{itemize}
        \item $(K,\leq_K)$ is a \textbf{very weak abstract elementary class} if it is an abstract class that satisfies:
        \begin{itemize}
            \item The \textbf{L\"{o}wenheim-Skolem property}: there is a cardinal $\LS(K)$ such that for any model $N\in K$ and any set $A\subseteq N$, there is $M\leq_K N$ such that $A\subseteq M$ and $|M|\leq|A|+\LS(K)$
            \item The \textbf{(weak) Tarski-Vaught chain property}: if $\alpha$ is a limit ordinal and $(M_i)_{i<\alpha}$ is an $\leq_K$-increasing continuous chain of models in $K$, then $N:=\bigcup_{i<\alpha}M_i$ is also a model in $K$, and each $M_i\leq_K N$
        \end{itemize}
        \item $(K,\leq_K)$ is an \textbf{weak abstract elementary class} if it is a very weak AEC which additionally satisfies the \textbf{Coherence property}: if $M_1\leq_K N$, $M_2\leq_K N$, and $M_1\subseteq M_2$, then $M_1\leq_K M_2$
        \item $(K,\leq_K)$ is an \textbf{abstract elementary class} if it is a weak AEC which additionally satisfies the \textbf{Smoothness property}: if $\alpha$ is a limit ordinal, $(M_i)_{i<\alpha}$ is an $\leq_K$-increasing continuous chain, and for each $i<\alpha$ we have that $M_i\leq_K N$, then $M_\alpha:=\bigcup_{i<\alpha} M_i\in K$ and $M_\alpha\leq_K N$ 
    \end{enumerate}
    For $(K,\leq_K)$ an abstract class, we denote by $\tau(K)$ the language of the models in $K$. We drop the subscript in $\leq_K$ when it is clear from context.
\end{defn}

\begin{defn}
    Let $(K,\leq)$ be an abstract class.
    \begin{enumerate}
        \item Given $M,N\in K$, a $\tau$-homomorphism $f:M\longrightarrow N$ is a \textbf{$K$-embedding} iff $f$ is a $\tau$-isomorphism between $M$ and $f[M]$, and $f[M]\leq N$
        \item $(K,\leq)$ has the \textbf{Amalgamation Property} (AP) if for models $M_0, M_1, M_2$ with $K$-embeddings $f_1:M_0\longrightarrow M_1, f_2:M_0\longrightarrow M_2$, there is a model $N\in K$ with $K$-embeddings $g_1:M_1\longrightarrow N,g_2:M_2\longrightarrow N$ such that the following diagram commutes:
        \begin{equation*}
            \begin{tikzcd}
                M_2 \arrow[r, "g_2"] & N\\
                M_0 \arrow[u, "f_2"] \arrow[r, "f_1"] & M_1 \arrow[u, "g_1"]
            \end{tikzcd}
        \end{equation*}
        \item We define the class $K^3:=\{(\bar{a},M,N):M\leq N, \bar{a}\in N\}$
        \item Given $(\bar{a}_1, M, N_1),(\bar{a}_2, M, N_2)\in K^3$, we define the relation $\sim$ such that $(\bar{a}_1,M,N_1)\sim(\bar{a}_2,M,N_2)$ iff there is a model $N'\geq N_2$ and a $K$-embedding $f:N_1\longrightarrow N'$ such that $f\upharpoonright M=\text{id}_M$ and $f(\bar{a}_1)=\bar{a}_2$
    \end{enumerate}
\end{defn}

\begin{fac}
    If $(K,\leq)$ has AP, then $\sim$ is an equivalence relation.
\end{fac}

\begin{defn}\label{galdef}
    Given $(\bar{a}, M, N)\in K^3$, the \textbf{Galois type} $\gtp(\bar{a}/M, N)$ is the equivalence class of $(\bar{a},M,N)$ under $\sim$. We say that $\bar{a}$ realizes the Galois type $p$ if $\gtp(\bar{a}/M, N)=p$. Given an ordered set $I$, we let $S^I(M)$ denote the collection of Galois types of the form $\gtp((a_i)_{i\in I}/M, N)$
\end{defn}

\section{Notions of Amalgamation}

Let $(K,\leq)$ be an abstract class. We would like to capture the idea of selecting certain amalgams of triples $M_0\leq M_1, M_2$ and designating them as the ``nice" amalgams that we will focus on; this is formalized in the following definition.

\begin{defn}
    Let the tuple $(M_0, M_1, M_2, f)$ be given such that $M_0, M_1, M_2\in K$, $M_0\leq M_1$ and $f:M_0\longrightarrow M_2$ is a $K$-embedding. A triple $(N, g_1, g_2)$ is an \textbf{amalgam of $M_1$ and $M_2$ over $M_0$ via $f$} if $N\in K$, $g_1:M_1\longrightarrow N$ and $g_2:M_2\longrightarrow N$ are $K$-embeddings, and the following diagram commutes (where $\iota$ denotes the inclusion embedding):
    \begin{equation*}
        \begin{tikzcd}
            M_2\arrow[r, "g_2"] & N\\
            M_0\arrow[u, "f"] \incarrow{r} & M_1 \arrow[u, "g_1"]
        \end{tikzcd}
    \end{equation*}
    For simplicity, we will also refer to the above diagram as an amalgam (of $M_1$ and $M_2$ over $M_0$ via $f$). We denote the collection of such amalgams by $\text{Amal}(M_0, M_1, M_2, f)$.
    
    A (class) function $\mathscr{A}$ is a \textbf{pre-notion of amalgamation} if:
    \begin{itemize}
        \item Its domain is the class of tuples $(M_0, M_1, M_2,f)$ such that $M_0\leq M_1$ and $f:M_0\longrightarrow M_2$ is a $K$-embedding; and
        \item For each such tuple, $\mathscr{A}(M_0, M_1, M_2, f)\subseteq\text{Amal}(M_0, M_1, M_2, f)$
    \end{itemize}
    For a triplet $(N, g_1, g_2)\in\mathscr{A}(M_0, M_1, M_2, f)$, we say that $(N, g_1, g_2)$ is an \textbf{ $\mathscr{A}$-amalgam of $M_1$ and $M_2$ over $M_0$ via $f$}, which we will also denote by the annotated diagram:
    \begin{equation*}
        \begin{tikzcd}
            M_2\arrow[r, "g_2"] \arrow[dr, phantom, "\mathscr{A}"] & N\\
            M_0\arrow[u, "f"] \incarrow{r} & M_1 \arrow[u, "g_1"]
        \end{tikzcd}
    \end{equation*}
    
    We say that $\mathscr{A}$ is a \textbf{notion of amalgamation} if in addition to being a pre-notion, the following properties hold of $\mathscr{A}$:
    \begin{itemize}
        \item (Completeness) For every tuple $(M_0, M_1, M_2, f)$ as above, $\mathscr{A}(M_0, M_1, M_2, f)$ is nonempty.
        \item $\scrA$ contains trivial amalgams: For any $M_0\leq M_1$, $(M_1, \iota, \text{id})\in\scrA(M_0, M_0, M_1, \iota)$. Diagrammatically,
        \begin{equation*}
            \begin{tikzcd}
                M_1\arrow[r, "\text{id}"] \arrow[dr, phantom, "\mathscr{A}"] & M_1\\
                M_0\incarrow{u} \arrow[r, "\text{id}"] & M_0 \incarrow{u}
            \end{tikzcd}
        \end{equation*}
        \item (Top Invariance) For every $(N, g_1, g_2)\in\mathscr{A}(M_0, M_1, M_2, f)$ and every $K$-isomorphism $h:N\simeq N'$, $(N', h\circ g_1, h\circ g_2)\in\mathscr{A}(M_0, M_1, M_2, f)$. Diagrammatically,
        \begin{equation*}
            \begin{tikzcd}
                M_2\arrow[r, "g_2"] \arrow[dr, phantom, "\mathscr{A}"] & N \arrow[r, "h"] & N'\\
                M_0\arrow[u, "f"] \incarrow{r} & M_1 \arrow[u, "g_1"]
            \end{tikzcd}\Longrightarrow
            \begin{tikzcd}
                M_2\arrow[r, "h\circ g_2"] \arrow[dr, phantom, "\mathscr{A}"] & N'\\
                M_0\arrow[u, "f"] \incarrow{r} & M_1 \arrow[u, "h\circ g_1", swap]
            \end{tikzcd}
        \end{equation*}
        \item (Side Invariance 1) For every $(N, g_1, g_2)\in\mathscr{A}(M_0, M_1, M_2, f)$ and $K$-isomorphism $h:M_1\simeq M'$, $(N, g_1\circ h^{-1}, g_2)\in\mathscr{A}(h[M_0], M', M_2, f\circ (h\upharpoonright M_0)^{-1})$. Diagrammatically,
        \begin{equation*}
            \begin{tikzcd}
                M_2\arrow[r, "g_2"] \arrow[dr, phantom, "\mathscr{A}"] & N\\
                M_0\arrow[u, "f"] \incarrow{r} & M_1 \arrow[u, "g_1"] \arrow[d, "h"]\\
                & M'
            \end{tikzcd}\Longrightarrow
            \begin{tikzcd}
                M_2\arrow[r, "g_2"] \arrow[dr, phantom, "\mathscr{A}"] & N\\
                h[M_0]\arrow[u, "f\circ (h\upharpoonright M_0)^{-1}"] \incarrow{r} & M' \arrow[u, "g_1\circ h^{-1}", swap]
            \end{tikzcd}
        \end{equation*}
        \item (Side Invariance 2) For every $(N, g_1, g_2)\in\mathscr{A}(M_0, M_1, M_2, f)$ and $K$-isomorphism $h:M_2\simeq M'$, $(N, g_1, g_2\circ h^{-1})\in\mathscr{A}(M_0, M_1, M', h\circ f)$. Diagrammatically,
        \begin{equation*}
            \begin{tikzcd}
                M'\\
                M_2\arrow[u, "h"] \arrow[r, "g_2"] \arrow[dr, phantom, "\mathscr{A}"] & N\\
                M_0\arrow[u, "f"] \incarrow{r} & M_1 \arrow[u, "g_1"]
            \end{tikzcd}\Longrightarrow
            \begin{tikzcd}
                M'\arrow[r, "g_2\circ h^{-1}"] \arrow[dr, phantom, "\mathscr{A}"] & N\\
                M_0\arrow[u, "h\circ f"] \incarrow{r} & M_1 \arrow[u, "g_1"]
            \end{tikzcd}
        \end{equation*}
        \item (Symmetry) If $(N, g_1, g_2)\in\mathscr{A}(M_0, M_1, M_2, f)$, then $(N, g_2, g_1)\in\mathscr{A}(f[M_0], M_2, M_1, f^{-1})$. Diagrammatically,
        \begin{equation*}
            \begin{tikzcd}
                M_2\arrow[r, "g_2"] \arrow[dr, phantom, "\mathscr{A}"] & N\\
                M_0\arrow[u, "f"] \incarrow{r} & M_1 \arrow[u, "g_1"]
            \end{tikzcd}\Longrightarrow
            \begin{tikzcd}
                M_1\arrow[r, "g_1"] \arrow[dr, phantom, "\mathscr{A}"] & N'\\
                f[M_0]\arrow[u, "f^{-1}"] \incarrow{r} & M_2 \arrow[u, "g_2"]
            \end{tikzcd}
        \end{equation*}
    \end{itemize}
\end{defn}

\begin{rem}
    Technically, $\mathscr{A}$ fails to even be a class function in the strictest sense, as $\mathscr{A}(M_0, M_1, M_2, f)$ is a proper class because of the invariance properties and also because there is no bound on the cardinality of the amalgams. This can of course be resolved by the assumption of a strongly inaccessible cardinal $\kappa$ such that every model of $K$ has cardinality $<\kappa$; in any case, this is inconsequential to this paper.
\end{rem}

Clearly, if $\scrA$ is a notion of amalgamation for $K$, then $K$ must have the Amalgamation Property as $\scrA$ is complete. Generally, we are interested in notions of amalgamation which specify certain well-behaved amalgams: for example, if $K$ has the Disjoint Amalgamation Property, we may define $\scrA_d$ as only the amalgams
\begin{equation*}
    \begin{tikzcd}
        M_2 \arrow[r, "g_2"] \arrow[dr, phantom, "\mathscr{A}_d"] & N\\
        M_0 \arrow[u, "f"] \incarrow{r} & M_1 \arrow[u, "g_1"]
    \end{tikzcd}
\end{equation*}
where $g_1[M_1]\cap g_2[M_2]=g_1[M_0]$. Since we would like to work in $K$ while ignoring the other amalgams which are not well-behaved, the properties defined above are designed such that some basic results which hold for amalgamation in general also hold for $\scrA$. For example:

\begin{lem}
    Suppose $\mathscr{A}$ is a notion of amalgamation, and:
    \begin{equation*}
        \begin{tikzcd}
            M_2\arrow[r, "g_2"] \arrow[dr, phantom, "\mathscr{A}"] & N\\
            M_0\arrow[u, "f"] \incarrow{r} & M_1 \arrow[u, "g_1"]
        \end{tikzcd}
    \end{equation*}
    \begin{enumerate}
        \item There is some $N'\geq M_1$ and $g_2':M_2\longrightarrow N'$ such that
        \begin{equation*}
            \begin{tikzcd}
                M_2\arrow[r, "g_2'"] \arrow[dr, phantom, "\mathscr{A}"] & N'\\
                M_0\arrow[u, "f"] \incarrow{r} & M_1 \incarrow{u}
            \end{tikzcd}
        \end{equation*}
        \item There is some $N''\geq M_2$ and $g_1':M_1\longrightarrow N''$ such that
        \begin{equation*}
            \begin{tikzcd}
                M_2\incarrow{r} \arrow[dr, phantom, "\mathscr{A}"] & N''\\
                M_0\arrow[u, "f"] \incarrow{r} & M_1 \arrow[u, "g_1'"]
            \end{tikzcd}
        \end{equation*}
    \end{enumerate}
\end{lem}

\begin{proof}
    \begin{enumerate}
        \item Let $N'$ be a copy of $N$ such that $M_1\leq N'$, and $h:N\simeq N'$ be such that $h\circ g_1=\iota:M_1\hookrightarrow N'$. Letting $g_2'=h\circ g_2$, the desired result follows from Top Invariance.
        \item Similar to (1), using $N''$ a copy of $N$ such that $M_2\leq N''$.
    \end{enumerate}
\end{proof}

\begin{lem}
    Suppose $\mathscr{A}$ is a notion of amalgamation, and:
    \begin{equation*}
        \begin{tikzcd}
            M_2\arrow[r, "g_2"] \arrow[dr, phantom, "\mathscr{A}"] & N\\
            M_0\arrow[u, "f"] \incarrow{r} & M_1 \arrow[u, "g_1"]
        \end{tikzcd}
    \end{equation*}
    Then
    \begin{equation*}
        \begin{tikzcd}
            g_2[M_2]\incarrow{r} \arrow[dr, phantom, "\mathscr{A}"] & N\\
            g_1[M_0]\incarrow{u} \incarrow{r} & g_1[M_1] \incarrow{u}
        \end{tikzcd}
    \end{equation*}
\end{lem}

\begin{proof}
    Firstly, note that as the diagram is commutative, indeed $g_1[M_0]=(g_2\circ f)[M_0]\leq g_2[M_2]$. By Side Invariance 1 (via the isomorphism $g_1:M_1\simeq g_1[M_1]$),
    \begin{equation*}
        \begin{tikzcd}
            M_2\arrow[r, "g_2"] \arrow[dr, phantom, "\mathscr{A}"] & N\\
            g_1[M_0]\arrow[u, "f\circ(g_1\upharpoonright M_0)^{-1}"] \incarrow{r} & g_1[M_1] \incarrow{u}
        \end{tikzcd}
    \end{equation*}
    Then, by Side Invariance 2 (via the isomorphism $g_2:M_2\simeq g_2[M_2]$),
    \begin{equation*}
        \begin{tikzcd}
            g_2[M_2]\incarrow{r} \arrow[dr, phantom, "\mathscr{A}"] & N\\
            g_1[M_0]\arrow[u, "g_2\circ f\circ(g_1\upharpoonright M_0)^{-1}"] \incarrow{r} & g_1[M_1] \incarrow{u}
        \end{tikzcd}
    \end{equation*}
    Finally, as $g_1\upharpoonright M_0=g_2\circ f$, hence $g_2\circ f\circ (g_1\upharpoonright M_0)^{-1}=\iota:g_1[M_1]\hookrightarrow g_2[M_2]$ as desired.
\end{proof}

Given the above lemmas, we see that to specify a notion of amalgamation $\scrA$, it suffices to specify when a commutative square of the form
\begin{equation*}
    \begin{tikzcd}
        M_2 \incarrow{r} & N\\
        M_0 \incarrow{u} \incarrow{r} & M_1 \incarrow{u}
    \end{tikzcd}
\end{equation*}
is in fact a $\scrA$-amalgam. Similarly, for most results of $\scrA$, it suffices to prove the statement only for commutative diagrams as above.

\begin{rem}
    Within the model theory literature, it is customary to say that $N$ is an amalagam of $M_1, M_2$ over $M_0$ if there is a  $K$-embedding $f$ such that
    \begin{equation*}
        \begin{tikzcd}
            M_2\arrow[r, "f"] & N\\
            M_0\incarrow{u} \incarrow{r} & M_1 \incarrow{u}
        \end{tikzcd}
    \end{equation*}
    Hence, if $N$ is an amalgam of $M_1, M_2$ over $M_0$, then for any $N'\geq N$, in this customary language it is also true that $N'$ is an amalgam of $M_1, M_2$ over $M_0$. On the other hand, in this paper the phrase ``$N$ (with $g_1, g_2$) is an $\mathscr{A}$-amalgam of $M_1, M_2$ over $M_0$ (via $f$)" refers specifically to the statement ``$(N, g_1, g_2)\in\mathscr{A}(M_0, M_1, M_2, f)$". In particular, since we do not assume that $\mathscr{A}$ has any upward-closure property, it is not necessarily true that for every $N'\geq N$, $(N', g_1, g_2)\in\mathscr{A}(M_0, M_1, M_2, f)$. It is, however, a relevant concept for the current investigation, and so we introduce a slight variant of the phrase to differentiate this interesting case:
\end{rem}

\begin{defn}
    We say that $N$ is an \textbf{$\mathscr{A}$-amalgam by inclusion} of $M_1$ and $M_2$ over $M_0$ if the following diagram is an $\mathscr{A}$-amalgam:
    \begin{equation*}
        \begin{tikzcd}
            M_2\incarrow{r} \arrow[dr, phantom, "\mathscr{A}"] & N\\
            M_0\incarrow{u} \incarrow{r} & M_1 \incarrow{u}
        \end{tikzcd}
    \end{equation*}
    For $M_0\leq M_1, M_2\leq N$, we say that \textbf{$M_1$ and $M_2$ are $\mathscr{A}$-subamalgamated over $M_0$ inside $N$} if there is some $N'\leq N$ such that $N'$ is an $\mathscr{A}$-amalgam by inclusion of $M_1, M_2$ over $M_0$.
\end{defn}

\begin{rem}
    It is important to note that we are \underline{not} asserting that every triple $M_0\leq M_1,M_2$ can be amalgamated by inclusions; nor will we be assuming that such a property holds for any notion of amalgamation we consider. This definition simply allows us to refer specifically to $\scrA$-amalgams of the above form.
\end{rem}

\begin{ex}
    Let $K$ be the class of vector spaces over a fixed field $F$, with $\leq_K$ the subspace relation. We can define $\scrA$ such that for $V\leq W_1, W_2\leq U$, $U$ is an $\scrA$-amalgam of $W_1, W_2$ over $V$ (by inclusion) iff $W_1\cap W_2=V$ and $\text{span}(W_1\cup W_2)=U$. In this example, if $U'\gneq U$, then $U'$ is \underline{not} an $\scrA$-amalgam of $W_1, W_2$ over $V$. However, $W_1, W_2$ are $\scrA$-subamalgamated over $V$ inside $U'$. More generally, if $T$ is (for example) a first order stable theory, we can define $\scrA$ such that for models $M_0\preccurlyeq M_1, M_2$ with $M_1\dnf_{M_0} M_2$, $N$ is an $\scrA$ amalgam iff $N$ is $(a,\kappa_r(T))$-prime over $M_1\cup M_2$.
\end{ex}

Some other examples of notions of amalgamation which we are interested in include:
\begin{enumerate}
    \item Consider the class of groups with the subgroup ordering. Given $G\leq H, K$, the free amalgamated product $H\ast_G K$ is formed by taking the free product of $H, K$ and identifying the two copies of $G$ together. This defines a notion of amalgamation on the class.
    \item Similarly, consider the class of (left-) modules over a fixed ring $R$ with the submodule ordering. As in the case for vector spaces, we can define $\scrA$ such that given $M_0\leq M_1, M_2\leq N$, $N$ is an $\scrA$-amalgam of $M_1, M_2$ over $M_0$ iff $M_1\cap M_2=M_0$ and $\text{span}(M_1\cup M_2)=N$. Note that this is equivalent to defining $\scrA$-amalgams by taking direct sums and quotienting to identify the copies of the amalgamation base.
    \item Consider the class of algebraically closed fields with characteristic $p$: Given $K_0\leq K_1, K_2\leq L$, we define $\scrA$ such that $L$ is an $\scrA$-amalgam of $K_1, K_2$ over $K_0$ iff $K_1\cap K_2=K_0$, $K_1$ and $K_2$ are algebraically independent over $K_0$, and $\text{acl}(K_1\cap K_2)=L$. More generally, this construction holds for any AEC $K$ where each model has a pregeometry which is ``coherent" with $K$; we will develop this idea further in section 7.
    \item In a different vein, let $K$ be a class of algebras which is an expansion of Boolean algebras, for example the class of cylindric algebras or polyadic algebras (of some fixed dimension $\alpha$). $K$ is said to have the \textit{super amalgamation property} if any span $A_0,A_1,A_2$ can be amalgamated by some $(B,f_1,f_2)$ satisfying: for every $x\in A_1$ and $y\in A_2$, if $f_1(x)\leq f_2(y)$ then there is $z\in A_0$ such that $x\leq z$ and $z\leq y$, and vice versa. If $K$ has the super amalgamation property, then the super amalgams define a notion of amalgamation.
    \item In \cite{SV18}, the notion of $\phi$-amalgamation is defined over an $AEC$ for a quantifier-free formula $\phi$ (assuming for simplicity that the language $\tau$ is relational): the diagram
    \begin{equation*}
        \begin{tikzcd}
            M_2 \arrow[r, "f_2"] & N\\
            M_0 \arrow[u, "\iota"] \arrow[r, "\iota"] & M_1 \arrow[u, "f_1"]
        \end{tikzcd}
    \end{equation*}
    is a $\phi$-amalgam iff $\phi(M_1), \phi(M_2)$ are equal as $\tau$-structures and $f_1\upharpoonright\phi(M_1)=f_2\upharpoonright \phi(M_2)$. This is clearly also a notion of amalgamation in the current sense.
    \item On the other hand, in the class of groups with the subgroup ordering, we can define another notion $\scrA$ such that for $G_0\leq G_1, G_2\leq H$, $H$ is an $\scrA$-amalgam of $G_1, G_2$ over $G_0$ iff $H=\langle G_1\cup G_2\rangle$. This is an example where $\scrA$ gives very little structural information about the class.
\end{enumerate}

\section{Some Structural Properties}

The last example above shows that even with $\scrA$ a specifically defined notion of amalgamation, $\scrA$ might not provide any structural information on the underlying class besides having the amalgamation property. As we are interested in stronger results which do not follow simply from the fact that $K$ has AP, we are interested in notions which satisfy some extra properties.

\begin{defn}\label{propdef}
    Let $K$ be an abstract class, and let $\mathscr{A}$ be a notion of amalgamation in $K$.
    \begin{itemize}
        \item $\mathscr{A}$ is \textbf{minimal} if for every $(N, g_1, g_2)\in\mathscr{A}(M_0, M_1, M_2, f)$, $N$ is minimal over $g_1[M_1]\cup g_2[M_2]$ i.e. if $N'\leq N$ and $g_1[M_1]\cup g_2[M_2]\subseteq N'$, then $N'=N$.
        \item $\scrA$ is \textbf{absolutely minimal} if for every $(N, g_1, g_2)\in\mathscr{A}(M_0, M_1, M_2, f)$ and for any $N^*\geq N$, if $N'\leq N^*$ is such that $g_1[M_1]\cup g_2[M_2]\subseteq N'$, then $N\leq N'$.
        \item $\mathscr{A}$ is \textbf{regular} if for every commutative square in $K$, the following conditions are equivalent:
        \begin{enumerate}
            \item The commutative square is an $\mathscr{A}$-amalgam i.e.
            \begin{equation*}
                \begin{tikzcd}
                    M_2\arrow[r, "g_2"] \arrow[dr, phantom, "\mathscr{A}"] & N\\
                    M_0\arrow[u, "f"] \incarrow{r} & M_1 \arrow[u, "g_1"]
                \end{tikzcd}
            \end{equation*}
            \item There is some $M', N', g'$ such that $M_0\leq M'\leq M_1$, $g_2[M_2]\leq N'\leq N$, and $g'= g_1\upharpoonright M'$, with both of the following commutative squares being $\mathscr{A}$-amalgams:
            \begin{equation*}
                \begin{tikzcd}
                    M_2\arrow[r, "g_2"] \arrow[dr, phantom, "\mathscr{A}"] & N' \arrow[dr, phantom, "\mathscr{A}"] \incarrow{r} & N\\
                    M_0\arrow[u, "f"] \incarrow{r} & M' \arrow[u, "g'"] \incarrow{r} & M_1 \arrow[u, "g_1"]
                \end{tikzcd}
            \end{equation*}
            \item For every $M'$ such that $M_0\leq M'\leq M_1$, there exists $N'\leq N$ such that the following commutative square is an $\mathscr{A}$-amalgam:
            \begin{equation*}
                \begin{tikzcd}
                    M_2\arrow[r, "g_2"] \arrow[dr, phantom, "\mathscr{A}"] & N' \incarrow{r} & N\\
                    M_0\arrow[u, "f"] \incarrow{r} & M' \arrow[u, swap, "g_1\upharpoonright M'"]
                \end{tikzcd}
            \end{equation*}
            Moreover, for any such choice of $N'$, the following commutative square is also an $\scrA$-amalgam:
            \begin{equation*}
                \begin{tikzcd}
                    N' \arrow[dr, phantom, "\mathscr{A}"] \incarrow{r} & N\\
                    M' \arrow[u, "g_1\upharpoonright M'"] \incarrow{r} & M_1 \arrow[u, "g_1"]
                \end{tikzcd}
            \end{equation*}
        \end{enumerate}
        \item $\mathscr{A}$ is \textbf{continuous} if for any limit $\delta$ and increasing continuous chains $(M_i)_{i<\delta}, (N_i)_{i<\delta}$ with $K$-embeddings $(f_i:M_i\longrightarrow N_i)_{i<\delta}$ such that:
        \begin{equation*}
            \begin{tikzcd}
                N_0 \incarrow{r} \arrow[dr, phantom, "\mathscr{A}"] & N_1 \incarrow{r} \arrow[dr, phantom, "\mathscr{A}"] & N_2 \incarrow{r} & \cdots \incarrow{r} & N_i \incarrow{r} \arrow[dr, phantom, "\mathscr{A}"] & N_{i+1} \incarrow{r} & \cdots\\
                M_0 \arrow[u, "f_0"] \incarrow{r} & M_1 \arrow[u, "f_1"] \incarrow{r} & M_2 \arrow[u, "f_2"] \incarrow{r} & \cdots \incarrow{r} & M_i \arrow[u, "f_i"] \incarrow{r} & M_{i+1} \arrow[u, swap, "f_{i+1}"] \incarrow{r} & \cdots
            \end{tikzcd}
        \end{equation*}
        The commutative square of the respective unions is also an $\mathscr{A}$-amalgam:
        \begin{equation*}
            \begin{tikzcd}
                N_0 \incarrow{r} \arrow[dr, phantom, "\mathscr{A}"] & \bigcup_{i<\delta}N_i\\
                M_0 \arrow[u, "f_0"] \incarrow{r} & \bigcup_{i<\delta}M_i \arrow[u, swap, "\bigcup_{i<\delta}f_i"]
            \end{tikzcd}
        \end{equation*}
        \item $\mathscr{A}$ \textbf{admits decompositions} if for every $M_0\leq M_1\leq N$, there is a $M_2$ such that $M_0\leq M_2\leq N$ and $N$ is an $\mathscr{A}$-amalgam of $M_1, M_2$ over $M_0$ (via the inclusion maps).
        \item $\mathscr{A}$ has \textbf{uniqueness} if for any two amalgams $(N, g_1, g_2),(N', g_1', g_2')\in\mathscr{A}(M_0, M_1, M_2, f)$, there exists a $K$-\underline{isomorphism} $h:N\cong N'$ such that the following diagram commutes:
        \begin{equation*}
            \begin{tikzcd}
                & & N'\\
                M_2 \arrow[r, "g_2"] \arrow[rru, bend left, "g_2'"] & N \arrow[ru, "h"]\\
                M_0 \arrow[u, "f"] \incarrow{r} & M_1 \arrow[u, "g_1"] \arrow[uur, bend right, "g_1'"]
            \end{tikzcd}
        \end{equation*}
    \end{itemize}
\end{defn}

\begin{rem}\leavevmode
    \begin{itemize}
        \item The ``absolute" in ``absolutely minimal" refers to the fact that the amalgam $N$ of $M_1, M_2$ over $M_0$ is minimal over $M_1\cup M_2$ only relative to models $N'$ which can be jointly embedded with $N$; this is only an issue in the current framework since we do not assume the existence of monster models.
        \item The literature is unfortunately split over the nomenclature for what is defined as uniqueness above: this property is sometimes known as ``strong uniqueness", whereas (using the language of \cite{SV18}) ``uniqueness" would refer to the property that two amalgamation diagrams can be amalgamated as indexed system of models. However, it is our opinion that within the current presentation the unqualified name ``uniqueness" is more natural in terms of the existence of isomorphisms.
        \item Furthermore, the uniqueness property is substantially different from the other properties defined above. This is because the properties such as minimality, continuity, and regularity are necessary for $\scrA$ to resemble taking direct sums enough to motivate any further work (as we will discuss in Section 3). On the other hand, both the uniqueness property and its failure have significant model-theoretical consequences; we will explore the consequences of the positive case in Section 5, and the consequences of the negative case in Section 6.
    \end{itemize}
\end{rem}

With these properties, we can start differentiating between various notions of amalgamation and the implications on the structure of the underlying class. A simple but illustrative example comes from abelian groups, and more specifically the torsion divisible groups:

\begin{ex}
    Fix $S$ a family of abelian groups such that for $G, H\in S$ with $G\neq H$, for any abelian group $K$ and group embeddings $f:G\longrightarrow K$, $g:H\longrightarrow K$, $f[G]\cap g[H]=0$, where $0$ is the trivial group (for example, the Prufer $p$-groups $S:=\{\mathbb{Z}(p^\infty):p\text{ a prime}\}$). Define the class $K$ such that $M\in K$ iff $M$ is a direct sum $\bigoplus_{i<\alpha} G'_i$, where each $G'_i$ is isomorphic to some $G_i\in S$, and let the ordering $\leq_K$ be the subgroup ordering. Note that the condition on $S$ implies that if $G,H\in K$ and $G\leq_K H$, then $H=G\oplus(\bigoplus_{i<\alpha} H'i)$ for some sequence of subgroups $H'_i$ which are isomorphic to groups in $S$.
    
    In this case, $K$ has an obvious notion of amalgamation $\scrA$, where $H$ is a $\scrA$-amalgam of $G^1, G^2$ over $G^0$ (by inclusion) iff $H=\bigoplus_{i<\alpha} H_i$, and there are sets $S_0, S_1, S_2\subseteq\alpha$ such that:
    \begin{itemize}
        \item For $l=0,1,2$, $G^l=\bigoplus_{i\in S_l} H_i$
        \item $S_1\cap S_2=S_0$ and $S_1\cup S_2=\alpha$
    \end{itemize}
    It is straightforward to see that $\scrA$ is minimal, absolutely minimal, regular, continuous, admits decomposition, and has uniqueness.
    
    It is interesting to note that $S$ as defined above cannot contain $\mathbb{Q}$ since $\mathbb{Q}$ can be embedded as a proper subgroup of itself. Of course, in the case where $K$ is the class of divisible groups, since any divisible group admits a unique decomposition into copies of $Z(p^\infty)$ and $\mathbb{Q}$, $\scrA$ can be naturally extended to a notion of amalgamation in the class of divisible groups. In particular, this extension of $\scrA$ formally relies on the natural notion of amalgamation in the class of vector spaces over $\mathbb{Q}$, which obviously satisfies all of the above properties. In this case, the notion $\scrA$ on the class of divisible groups also satisfies all of these properties.
\end{ex}

The free product over groups also gives rise to more complicated examples of amalgamation, for example using small cancellation theory:

\begin{ex}
    Let $S$ be a class function on triples of groups, such that for $G_0\leq G_1, G_2$, $S(G_0, G_1, G_2)\subseteq \mathscr{P}(G_1\ast_{G_0} G_2)$ is a nonempty family of sets such that each $R\in S(G_0, G_1, G_2)$ is symmetrized and satisfies $C'(1/6)$, where $C'(\lambda)$ is the metric small cancellation condition (see \cite{Ly01}, Chapter 5 for discussion related to small cancellation theory, including the relevant definitions).
    
    Now, let $K$ be the class of groups ordered by the subgroup relation, and define $\scrA$ such that given $G_0\leq G_1, G_2\leq H$, $H$ is an $\scrA$-amalgam of $G_1, G_2$ over $G_0$ (by inclusion) iff $H\cong G_1\ast_{G_0} G_2/\langle R\rangle_N$, where $R\in S(G_0, G_1, G_2)$ and $\langle R\rangle_N$ is the normal closure of $R$ in $H$.
    
    In particular, we note that if $S(G_0, G_1, G_2)$ contains (for example) both the empty set and a set not contained inside $G_0$, then there are two $\scrA$-amalgams of $G_1, G_2$ over $G_0$ which are not isomorphic over $G_1\cup G_2$, and hence $\scrA$ does not have uniqueness. Similarly, whether or not $\scrA$ satisfies regularity, continuity, and admission of decomposition depends on the function $S$. On the other hand, $\scrA$ is necessarily absolutely minimal as $H$ is generated by $G_1\cup G_2$.
\end{ex}

For the rest of this paper, we will restrict our attention to weak AECS:

\begin{hyp}
    $(K,\leq)$ is a weak AEC.
\end{hyp}

We note, however, that many of the results presented do not require the Coherence property of weak AECs (Definition \ref{aecdef}), and so we will be explicit when using the Coherence property. On the other hand, the properties defined in Definition \ref{propdef} for a notion of amalgamation puts additional constraints on the class, and the example below shows that even a very ``natural" notion of amalgamation in a very weak AEC can fail to have the above properties:

\begin{ex}
    Consider the class $(K_{ACFp},\leq_K)$, where $K_{ACFp}$ is the class of algebraically closed fields of characteristic $p$ but $L_1\leq_K L_2$ iff $|L_1|<|L_2|$ or $L_2$ is a limit model over $L_1$. It is straightforward to check that $(K_{ACFp},\leq_K)$ is a very weak AEC. Note that $L_2$ is a limit model over $L_1$ iff $\text{td}(L_2/L_1)=|L_2|$, where $\text{td}(K/F)$ is the transcendental degree of $K$ over $F$.
    
    We define a notion of amalgamation $\scrA$ in the following manner: given $L_0\leq_K L_1, L_2\leq_K M$, $M$ is an $\scrA$-amalgam of $L_1, L_2$ over $L_0$ (by inclusion) iff
    \begin{enumerate}
        \item $L_1\cap L_2=L_0$
        \item $L_1$ and $L_2$ are algebrically independent over $L_0$
        \item Assuming WLOG $|L_1|\leq |L_2|$, $\text{td}(M/L_3)=|L_2|-|L_1|$, where $L_3:=\text{acl}(L_1\cup L_2)$
    \end{enumerate}
    
    The third condition is necessary (for example) in the case where $|L_1|<|L_2|$, since in this case
    \begin{equation*}
        \text{td}(L_3/L_2)=\text{td}(L_1/L_0)=|L_1|<|L_2|
    \end{equation*}
    which implies that $L_3$ is not a limit model over $L_2$. On the other hand, $\scrA$ does not satisfy some of the above properties:
    \begin{itemize}
        \item $\scrA$ is not minimal: given models $L_0, L_1, L_2, L_3, M$ as above, there is some model $M'\lneq_K M$ such that $|M'|=|M|$ and $L_3\leq_K M'$, so in particular $\text{td}(M'/L_3)=\text{td}(M/L_3)$. Hence $M'$ is also an $\scrA$-amalgam of $L_1, L_2$ over $L_0$. This also shows that $\scrA$ is not absolutely minimal (see also Lemma \ref{FIminwp}).
        \item $\scrA$ is not continuous: Suppose the set $\{a_i:i<\omega\}\cup\{b_j:j<\omega_1\}\subseteq M$ are algebraically independent, and define:
        \begin{enumerate}
            \item $M_0=\bar{\mathbb{Q}}$
            \item $N_0=M_0(a_i:i<\omega)$
            \item For $\alpha\geq 1$, $M_i=M_0(b_j:j<\omega\cdot\alpha)$
            \item For $\alpha\geq 1$, $N_i=M_0(\{a_i:i<\omega\}\cup\{b_j:j<\omega\cdot\alpha\})$
        \end{enumerate}
        Note then this gives $\scrA$-amalgams:
        \begin{equation*}
            \begin{tikzcd}
                N_0 \incarrow{r} \arrow[dr, phantom, "\mathscr{A}"] & N_1 \incarrow{r} \arrow[dr, phantom, "\mathscr{A}"] & N_2 \incarrow{r} & \cdots \incarrow{r} & N_i \incarrow{r} \arrow[dr, phantom, "\mathscr{A}"] & N_{i+1} \incarrow{r} & \cdots\\
                M_0 \arrow[u, "f_0"] \incarrow{r} & M_1 \arrow[u, "f_1"] \incarrow{r} & M_2 \arrow[u, "f_2"] \incarrow{r} & \cdots \incarrow{r} & M_i \arrow[u, "f_i"] \incarrow{r} & M_{i+1} \arrow[u, swap, "f_{i+1}"] \incarrow{r} & \cdots
            \end{tikzcd}
        \end{equation*}
        On the other hand, $N_{\omega_1}=M_0(\{a_i:i<\omega\}\cup\{b_j:j<\omega_1\})$ is not a limit model over $M_{\omega_1}=M_0(b_j:j<\omega_1)$ as $\text{td}(N_{\omega_1}/M_{\omega_1})=\aleph_0$, and so in particular $N_{\omega_1}$ is not an $\scrA$-amalgam of $N_0, M_{\omega_1}$ over $M_0$.
    \end{itemize}
\end{ex}

By assuming that $\scrA$ satisfies some of the properties from Definition \ref{propdef}, a few basic results can be deduced. In particular, these results are analogous to basic properties of the direct sum on vector spaces.

\begin{lem}\label{3rotation}
    Suppose $\scrA$ is a notion of amalgamation that is regular. If $M_0, M_1, M_2, M_3, M', N$ are models such that:
    \begin{enumerate}
        \item $M'$ is a $\scrA$-amalgam of $M_1, M_2$ over $M_0$ by inclusion, i.e.
        \begin{equation*}
            \begin{tikzcd}
                M_1 \incarrow{r} \phanArrow{dr} & M'\\
                M_0 \incarrow{u} \incarrow{r} & M_2 \incarrow{u}
            \end{tikzcd}
        \end{equation*}
        \item $N$ is a $\scrA$-amalgam of $M_3, M'$ over $M_0$ by inclusion, i.e.
        \begin{equation*}
            \begin{tikzcd}
                M_3 \incarrow{r} \phanArrow{dr} & N\\
                M_0 \incarrow{u} \incarrow{r} & M' \incarrow{u}
            \end{tikzcd}
        \end{equation*}
    \end{enumerate}
    Then there is $N'\leq N$ such that:
    \begin{enumerate}
        \item $N'$ is an $\scrA$-amalgam of $M_2, M_3$ over $M_0$ by inclusion; and
        \item $N$ is an $\scrA$-amalgam of $M_1, N'$ over $M_0$ by inclusion
    \end{enumerate}
\end{lem}

\begin{proof}
    Note that by the regularity, since $M_0\leq M_2\leq M'$, there exists $N'\leq N$ such that:
    \begin{equation*}
        \begin{tikzcd}
            M_3 \incarrow{r} \phanArrow{dr} & N' \incarrow{r} \phanArrow{dr} & N\\
            M_0 \incarrow{u} \incarrow{r} & M_2 \incarrow{u} \incarrow{r} & M' \incarrow{u}
        \end{tikzcd}
    \end{equation*}
    In particular, we have the following diagram:
    \begin{equation*}
        \begin{tikzcd}
            M_3 \incarrow{r} \phanArrow{dr} & N' \incarrow{r} \phanArrow{dr} & N\\
            M_0 \incarrow{u} \incarrow{r} \arrow[rd, "\text{id}"] & M_2 \incarrow{u} \incarrow{r} \phanArrow{dr} & M' \incarrow{u}\\
            & M_0 \incarrow{u} \incarrow{r} & M_1 \incarrow{u}
        \end{tikzcd}
    \end{equation*}
    Applying regularity to the two commutative squares on the right, this shows that $N$ is indeed a $\scrA$-amalgam of $N', M_1$ over $M_0$ by inclusion.
\end{proof}

\begin{lem}\label{wpunique}
    Suppose $\scrA$ is a notion of amalgamation and is absolutely minimal. If $M_1, M_2$ are $\scrA$-subamalgamated over $M_0$ inside $N$, then there is a unique $N'\leq N$ such that $N'$ is the $\scrA$-amalgam of $M_1, M_2$ over $M_0$ by inclusion.
\end{lem}

\begin{proof}
    Let $N'\leq N$ be an $\scrA$-amalgam of $M_1, M_2$ over $M_0$ by inclusion, and suppose $N^*\leq N$ is also an $\scrA$-amalgam of $M_1, M_2$ over $M_0$ by inclusion. In particular, hence $M_1\cup M_2\subseteq N^*$. As $\scrA$ is absolutely minimal and $N', N^*\leq N$, hence $N'\leq N^*$. The symmetric argument also shows that $N^*\leq N'$, and hence $N'=N^*$.
\end{proof}

\begin{nota}
    If $\scrA$ is absolutely minimal, and the models $M_0 \leq M_1, M_2\leq N$ are such that $M_1, M_2$ are $\scrA$-subamalgamated over $M_0$ inside $N$, then we denote the unique $\scrA$-amalgam inside $N$ by $M_1\oplus_{M_0}^N M_2$.
\end{nota}

\begin{lem}\label{comasso}
    Suppose $\scrA$ is absolutely minimal and regular. Then for any $M\leq N$, the operation $\oplus_M^N$ is commutative and associative where defined.
\end{lem}

\begin{proof}
    That $\oplus_M^N$ is commutative is from $\scrA$ being symmetric. Associativity follows from Lemma \ref{3rotation}.
\end{proof}

\begin{defn}\label{defFI}
    We say an AEC $K$ \textbf{admits finite intersections} (abbreviated to has \textbf{FI}) if whenever $M_1, M_2$ are such that there exists $M_0, N$ with $M_0\leq M_1, M_2\leq N$, then the intersection $M_1\cap M_2$ is a model in $K$.
\end{defn}

\begin{lem}\label{FIminwp}
    Let $\scrA$ be a notion of amalgamation.
    \begin{enumerate}
        \item If $\scrA$ is absolutely minimal, then it is minimal.
        \item If $K$ admits finite intersections and $\scrA$ is minimal, then $\scrA$ is absolutely minimal.
    \end{enumerate}
\end{lem}

\begin{proof}
    Note that by the Invariance properties of $\scrA$, it suffices to show that the above statements hold for any $M_0, M_1, M_2, N$ such that:
    \begin{equation*}
        \begin{tikzcd}
            M_1 \incarrow{r} \phanArrow{dr} & N\\
            M_0 \incarrow{u} \incarrow{r} & M_2 \incarrow{u}
        \end{tikzcd}
    \end{equation*}
    \begin{enumerate}
        \item Assume that $\scrA$ is absolutely minimal. If $N'\leq N$ is such that $M_1\cup M_2\subseteq N'$, then $N\leq N'$ by absolute minimality, and hence $N'=N$. This shows that $\scrA$ is minimal.
        \item Assume that $K$ admits finite intersections and $\scrA$ is minimal. Then, if $N^*\geq N$ and $N'\leq N^*$ is such that $M_1\cup M_2\subseteq N'$, since $K$ admits finite intersection, $N''=N\cap N'$ is also a model of $K$, and furthermore $M_1\cup M_2\subseteq N''$. But then by minimality, $N''=N$, and hence $N\leq N'$ as desired.
    \end{enumerate}
\end{proof}

\begin{lem}\label{amalcard}
    Suppose $\scrA$ is minimal. If $N$ is an $\scrA$-amalgam of $M_1, M_2$ over $M_0$ by inclusion and $|N|\geq \LS(K)$, then $|N|=|M_1|+|M_2|+\LS(K)$. 
\end{lem}

\begin{proof}
    Since $M_1, M_2\leq N$, by the L\"{o}wenheim-Skolem axiom there is some $N'\leq N$ such that $|N'|\leq \LS(K)+|M_1\cup M_2|$. Since $\scrA$ is minimal, hence $N'=N$, giving the desired result.
\end{proof}

\begin{lem}\label{restransfer}
    Suppose $\scrA$ is a notion of amalgamation that is regular and continuous. Let $N$ be an $\scrA$-amalgam of $M^*, M$ over $M_b$ by inclusion, $\delta$ be a limit ordinal, and $(M_i)_{i<\delta}$ be a continuous resolution of $M$ such that $M_b\leq M_0$. Then there is a continuous resolution $(N_i)_{i<\delta}$ of $N$ such that for each $i<\delta$, $N_i$ is an $\scrA$-amalgam of $M^*, M_i$ over $M_b$ by inclusion.
\end{lem}

\begin{proof}
    We will construct $N_i$ by induction:
    \begin{enumerate}
        \item Since $N$ is an $\scrA$-amalgam of $M^*, M$ over $M_b$ by inclusion, and $M_0$ is such that $M_b\leq M_0\leq M$, by regularity there is $N_0\leq N$ such that:
        \begin{equation*}
            \begin{tikzcd}
                M^* \incarrow{r} \phanArrow{dr} & N_0 \incarrow{r} \phanArrow{dr} & N\\
                M_b \incarrow{u} \incarrow{r} & M_0 \incarrow{r} \incarrow{u} & M \incarrow{u}
            \end{tikzcd}
        \end{equation*}
        \item If $N_i$ is already defined, by construction
        \begin{equation*}
            \begin{tikzcd}
                M^* \incarrow{r} \phanArrow{dr} & N_i\\
                M_b \incarrow{u} \incarrow{r} & M_i \incarrow{u}
            \end{tikzcd}
        \end{equation*}
        As $\scrA$ is regular, it is also the case that
        \begin{equation*}
            \begin{tikzcd}
                N_i \incarrow{r} \phanArrow{dr} & N\\
                M_i \incarrow{u} \incarrow{r} & M \incarrow{u}
            \end{tikzcd}
        \end{equation*}
        Since $M_i\leq M_{i+1}\leq M$, again by regularity, there is $N_{i+1}\leq N$ such that
        \begin{equation*}
            \begin{tikzcd}
                N_i \incarrow{r} \phanArrow{dr} & N_{i+1} \incarrow{r} \phanArrow{dr} & N\\
                M_i \incarrow{u} \incarrow{r} & M_{i+1} \incarrow{r} \incarrow{u} & M \incarrow{u}
            \end{tikzcd}
        \end{equation*}
        \item At limit stage $\alpha$, we have
        \begin{equation*}
            \begin{tikzcd}
                M^* \incarrow{r} \phanArrow{dr} & N_0 \incarrow{r} \phanArrow{dr} & N_1 \incarrow{r} & \cdots \incarrow{r} & N_i \incarrow{r} \phanArrow{dr} & N_{i+1} \incarrow{r} & \cdots\\
                M_b \incarrow{r} \incarrow{u} & M_0 \incarrow{r} \incarrow{u} & M_1 \incarrow{r} \incarrow{u} & \cdots \incarrow{r} & M_i \incarrow{r} \incarrow{u} & M_{i+1} \incarrow{r} \incarrow{u} & \cdots
            \end{tikzcd}
        \end{equation*}
        As $\scrA$ is continuous and $(M_i)_{i<\alpha}$ is an increasing continuous chain, letting $N_\alpha=\bigcup_{i<\alpha} N_i$, we get that
        \begin{equation*}
            \begin{tikzcd}
                M^* \incarrow{r} \phanArrow{dr} & N_\alpha\\
                M_b \incarrow{u} \incarrow{r} & M_\alpha \incarrow{u}
            \end{tikzcd}
        \end{equation*}
    \end{enumerate}
\end{proof}

\section{Sequential amalgamation}

From a model-theoretic perspective, that the class of vector spaces over a fixed (countable) field is uncountably categorical stems from the exchange property of vectors and the fact that all vector spaces are direct sums of 1-dimensional spaces. In order to mimic this structure (or equivalently, the structure of models with a pregeometry), we must first define the amalgam of not only two models but of a possibly infinite sequence of models. We thus devote this section to showing that under the assumptions of $\scrA$ being absolutely minimal, regular, and continuous, then sequential amalgamation under $\scrA$ behaves as one would expect from the example of direct sums.

\begin{nota}
    For an ordinal $\alpha$, we define the ordinal $s(\alpha)$ by:
    \begin{itemize}
        \item $s(\alpha)=\alpha$ for limit $\alpha$
        \item $s(\alpha)=\alpha+1$ otherwise
    \end{itemize}
\end{nota}

\begin{defn}
    Let $M_b\in K$, and let $(M_i)_{i<\alpha}$ be a sequence of models such for each $i<\alpha$, $M_b\leq M_i$. We say that $N$ is an \textbf{$\mathscr{A}$-amalgam of $(M_i)_{i<\alpha}$ over $M_b$} if there exists a sequence of models $(N_i)_{i<s(\alpha)}$ and $K$-embeddings $(f_i:M_i\longrightarrow N_{i+1})_{i<\alpha}$ such that:
    \begin{enumerate}
        \item $N_0=M_b$ and $N_1=f_0[M_0]$
        \item For each $i<\alpha$, $f_i[M_b]=M_b$ and $f_i[M_i]\leq N_{i+1}$
        \item $(N_i)_{i<s(\alpha)}$ is a continuous resolution of $N$ i.e. it is an increasing continuous chain with $N=\bigcup_{i<s(\alpha)} N_i$.
        \item For every $i\geq 1$, the following diagram is an $\mathscr{A}$-amalgam:
        \begin{equation*}
            \begin{tikzcd}
                N_i \incarrow{r} \arrow[dr, phantom, "\mathscr{A}"] & N_{i+1}\\
                M_b \incarrow{u} \incarrow{r} & M_i \arrow[u, "f_i"]
            \end{tikzcd}
        \end{equation*}
    \end{enumerate}
    Paralleling the two-model case, we say that $N$ is an \textbf{$\mathscr{A}$-amalgam by inclusion} of $(M_i)_{i<\alpha}$ over $M_b$ if $N$ is an $\mathscr{A}$-amalgam as above with each $f_i$ being an inclusion map $\iota_i:M_i\hookrightarrow N_{i+1}$. When each $M_i\leq N$, we say that \textbf{$(M_i)_{i<\alpha}$ is $\mathscr{A}$-subamalgamated over $M_b$ inside $N$} if there is some $N'\leq N$ such that $N'$ is an $\mathscr{A}$-amalgam by inclusion.
\end{defn}

In order to understand what properties of sequential amalgams is desirable for our analysis, recall that any divisible group can be uniquely decomposed as a direct sum of (copies of) the rationals and Prufer $p$-groups. Using this as a guiding example, ideally the amalgamation of a sequence of models should be independent from the order of amalgamation, and moreover it should be possible to take subsets of a ``basis" to construct smaller models. In order to prove this claim (Theorem \ref{setamal}), we proceed by a number of lemmas:

\begin{lem}\label{inittele}
    If $N$ is an $\scrA$-amalgam of $(M_i)_{i<\alpha}$ over $M_b$, then for any $\beta\leq\alpha$, there exists some $L\leq N$ such that:
    \begin{enumerate}
        \item $L$ is an $\scrA$-amalgam of $(M_i)_{i<\beta}$ over $M_b$; and
        \item $N$ is an $\scrA$-amalgam of the sequence $(L)^\frown (M_i)_{\beta\leq i<\alpha}$ over $M_b$
    \end{enumerate}
\end{lem}

\begin{proof}
    Let $(N_i)_{i<s(\alpha)}$ be a continuous resolution of $N$ witnessing that $N$ is an $\scrA$-amalgam of $(M_i)_{i<\alpha}$ over $M_b$ via the maps $(f_i:M_i\longrightarrow N_{i+1})_{i<\alpha}$, and let $L=\bigcup_{i<\beta} N_i$. Then the resolution $(N_i)_{i<s(\beta)}$ witnesses that $L$ is the desired $\scrA$-amalgam, and moreover the sequence $(L)^\frown (N_i)_{\beta\leq i<s(\alpha)}$ witnesses that $N$ is also an $\scrA$-amalgam of $(L)^\frown (M_i)_{\beta\leq i<\alpha}$ over $M_b$ (via the maps $(\iota: L\longrightarrow N)^\frown (f_i)_{\beta\leq i<\alpha}$).
\end{proof}

\begin{lem}\label{tailtele}
    Suppose that $\mathscr{A}$ is a notion of amalgamation which is regular and continuous. If $N$ is an $\mathscr{A}$-amalgam of $(M_i)_{i<\alpha}$ over $M_b$ via the maps $(f_i:M_i\longrightarrow N)_{i<\alpha}$, then there is some $L\leq N$ such that:
    \begin{itemize}
        \item $L$ is an $\mathscr{A}$-amalgam of $(M_i)_{1\leq i<\alpha}$ over $M_b$ via the same maps; and
        \item $N$ is an $\mathscr{A}$-amalgam of $L$ and $M_0$ over $M_b$ in the following diagram:
        \begin{equation*}
            \begin{tikzcd}
                M_0 \arrow[r, "f_0"] \phanArrow{dr} & N\\
                M_b \incarrow{u} \incarrow{r} & L \incarrow{u}
            \end{tikzcd}
        \end{equation*}
    \end{itemize}
\end{lem}

\begin{proof}
    Fix $(N_i)_{i<s(\alpha)}$ a continuous resolution of $N$ witnessing that $N$ is an $\mathscr{A}$-amalgam of $(M_i)_{i<\alpha}$ over $M_b$ via $(f_i)_{i<\alpha}$. Let us first construct the model $L$ as the union of an increasing continuous chain $(L_i)_{1\leq i<s(\alpha)}$, with the following conditions:
    \begin{enumerate}
        \item $L_1=f_1[M_1]$, and each $L_i\leq N_i$
        \item For limit $\delta$, $L_\delta=\bigcup_{1\leq i<\delta} L_i$
        \item For $i\geq 1$, the following diagram is an $\mathscr{A}$-amalgam:
        \begin{equation*}
            \begin{tikzcd}
                M_{i+1} \arrow[r, "f_{i+1}"] \arrow[dr, phantom, "\mathscr{A}"] & L_{i+1}\\
                M_b \incarrow{u} \incarrow{r} & L_i \incarrow{u}
            \end{tikzcd}
        \end{equation*}
        \item For $i\geq 1$, the following diagram is an $\mathscr{A}$-amalgam:
        \begin{equation*}
            \begin{tikzcd}
                L_i \incarrow{r} \arrow[dr, phantom, "\mathscr{A}"] & N_i\\
                M_b \incarrow{u} \incarrow{r} & M_0 \arrow[u, "f_0"]
            \end{tikzcd}
        \end{equation*}
    \end{enumerate}
    For the successor step, recall that as $(N_i)_{i<s(\alpha)}$ witnesses that $N$ is an $\mathscr{A}$-amalgam of $(M_i)_{i<\alpha}$ over $M_b$, in particular for each $i<\alpha$, the following diagram is an $\mathscr{A}$-amalgam:
    \begin{equation*}
        \begin{tikzcd}
            M_i \arrow[r, "f_i"] \arrow[dr, phantom, "\mathscr{A}"] & N_{i+1}\\
            M_b \incarrow{u} \incarrow{r} & N_i \incarrow{u}
        \end{tikzcd}
    \end{equation*}
    Hence, as $L_i\leq N_i$ by assumption, by regularity there exists some $L_{i+1}\leq N_{i+1}$ such that:
    \begin{equation*}
        \begin{tikzcd}
            M_i \arrow[r, "f_i"] \arrow[dr, phantom, "\mathscr{A}"] & L_{i+1} \incarrow{r} \arrow[dr, phantom, "\mathscr{A}"] & N_{i+1}\\
            M_b \incarrow{u} \incarrow{r} & L_i \incarrow{u} \incarrow{r} & N_i \incarrow{u}
        \end{tikzcd}
    \end{equation*}
    It remains to show that (4) is satisfied. We note that combining the above diagram and assuming (4) holds for $L_i$, we get the following diagram:
    \begin{equation*}
        \begin{tikzcd}
            M_{i+1} \arrow[r, "f_i"] \phanArrow{dr} & L_{i+1} \incarrow{r} \phanArrow{dr} & N_{i+1}\\
            M_b \arrow[dr, "\text{id}"] \incarrow{r} \incarrow{u} & L_i \incarrow{u} \incarrow{r} \phanArrow{dr} & N_i \incarrow{u}\\
            & M_b \incarrow{u} \incarrow{r} & M_0 \arrow[u, "f_0"]
        \end{tikzcd}
    \end{equation*}
    Applying regularity to the two commutative squares on the right, we see that (4) is satisfied at the $i+1$ step:
    \begin{equation*}
        \begin{tikzcd}
            L_{i+1} \incarrow{r} \phanArrow{dr} & N_{i+1}\\
            M_b \incarrow{u} \incarrow{r} & M_0 \arrow[u, "f_0"]
        \end{tikzcd}
    \end{equation*}
    For the limit step, it suffices to check again that $L_\delta$ satisfies (4). Since $L_i$ satisfies (4) by assumption for $i<\delta$, we have the diagram:
    \begin{equation*}
        \begin{tikzcd}
            f_0[M_0] \incarrow{r} \phanArrow{dr} & N_1 \incarrow{r} \phanArrow{dr} & N_2 \incarrow{r} & \cdots \incarrow{r} & N_i \incarrow{r}\phanArrow{dr} & N_{i+1} \incarrow{r} & \cdots\\
            M_b \incarrow{u} \incarrow{r} & L_1 \incarrow{r} \incarrow{u} & L_2 \incarrow{r} \incarrow{u} & \cdots \incarrow{r} & L_i \incarrow{r} \incarrow{u} & L_{i+1} \incarrow{r} \incarrow{u} & \cdots
        \end{tikzcd}
    \end{equation*}
    Hence by continuity (and invariance), we get that
    \begin{equation*}
        \begin{tikzcd}
            M_0 \arrow[r, "f_0"] \phanArrow{dr} & \bigcup_{i<\delta} N_i \arrow[r, "\text{id}"] & N_\delta\\
            M_b \incarrow{u} \incarrow{r} & \bigcup_{i<\delta} L_i \incarrow{u} \arrow[r, "\text{id}"] & L_\delta \incarrow{u}
        \end{tikzcd}
    \end{equation*}
    This completes the definition of $(L_i)_{1\leq i<s(\alpha)}$. Note then that this resolution of $L=\bigcup_{i<s(\alpha)} L_i$ is a witness to the fact that $L$ is an $\mathscr{A}$-amalgam of $(M_i)_{1\leq i<\alpha}$ over $M_b$, and moreover the proof for (4) in the limit case also shows that $N=\bigcup_{i<s(\alpha)} N_i$ is an $\scrA$-amalgam of $M_0, L$ over $M_b$, as desired.
\end{proof}

\begin{cor}\label{inittailtele}
    Suppose $\scrA$ is regular and continuous. If $N$ is an $\scrA$-amalgam of $(M_i)_{i<\alpha}$ over $M_b$ by inclusion, then for any $0<j<\alpha$, there are $L_1, L_2\leq N$ such that:
    \begin{itemize}
        \item $L_1$ is an $\scrA$-amalgam of $(M_i)_{i<j}$ over $M_b$ by inclusion
        \item $L_2$ is an $\scrA$-amalgam of $(M_i)_{j<i<\alpha}$ over $M_b$ by inclusion; and
        \item $N$ is an $\scrA$-amalgam of $(L_1, M_j, L_2)$ over $M_b$ by inclusion
    \end{itemize}
\end{cor}

\begin{proof}
    That $L_1$ exists by Lemma \ref{inittele} and $L_2$ exists by Lemma \ref{tailtele}.
\end{proof}

\begin{lem}\label{taildecom}
    Suppose $\scrA$ is a notion of amalgamation that is regular and continuous. Let $N$ be an $\scrA$-amalgam of $M^*, M'$ over $M_b$ by the following diagram:
    \begin{equation*}
        \begin{tikzcd}
            M^*\arrow[r, "g"] \phanArrow{dr} & N\\
            M_b \incarrow{u} \incarrow{r} & M' \incarrow{u}
        \end{tikzcd}
    \end{equation*}
    If $M'$ is an $\scrA$-amalgam of $(M_i)_{i<\alpha}$ over $M_b$ (via the $K$-embeddings $(f_i:M_i\longrightarrow M)_{i<\alpha}$), then the sequence $(M^*)^\frown (M_i)_{i<\alpha}$ is $\scrA$-subamalgamated over $M_b$ inside $N$.
\end{lem}

\begin{proof}
    Since $M'$ is an $\scrA$-amalgam of $(M_i)_{i<\alpha}$ over $M_b$ via the maps $(f_i)_{i<\alpha}$, there is a continuous resolution $(M_i')_{i<s(\alpha)}$ of $M'$ such that $M_0'=M_b, M_1'=f_0[M_0]$ and for each $1\leq i<\alpha$,
    \begin{equation*}
        \begin{tikzcd}
            M_i'\incarrow{r} \phanArrow{dr} & M_{i+1}'\\
            M_b \incarrow{u} \incarrow{r} & M_i \arrow[u, "f_i"]
        \end{tikzcd}
    \end{equation*}
    So let us define an increasing continuous chain $(N_i)_{i<\beta}$ such that
    \begin{enumerate}
        \item $\beta=\alpha+2$ iff $\alpha<\omega$; otherwise $\beta=s(\alpha)$
        \item $N_0=M_b$ and $N_1=g[M^*]$
        \item For limit $\delta$, $N_\delta=\bigcup_{i<\delta} N_i$
        \item For each $i<\omega$, $M_i'\leq N_{i+1}\leq N$ and the commutative squares in the following diagram are $\scrA$-amalgams:
        \begin{equation*}
            \begin{tikzcd}
                N_1 \incarrow{r} \phanArrow{dr} & N_{i+2} \incarrow{r} \phanArrow{dr} & N\\
                M_b \incarrow{u} \incarrow{r} \incarrow{dr} & M_{i+1}' \incarrow{u} \incarrow{r} & M \incarrow{u}\\
                & M_i \arrow[u, "f_i"]
            \end{tikzcd}
        \end{equation*}
        \item For each $i$ such that $\omega\leq i<\alpha$, $M_i'\leq N_i\leq N$ and the commutative squares in the following diagram are $\scrA$-amalgams:
        \begin{equation*}
            \begin{tikzcd}
                N_1 \incarrow{r} \phanArrow{dr} & N_{i+1} \incarrow{r} \phanArrow{dr} & N\\
                M_b \incarrow{u} \incarrow{r} \incarrow{dr} & M_{i+1}' \incarrow{u} \incarrow{r} & M \incarrow{u}\\
                & M_i \arrow[u, "f_i"]
            \end{tikzcd}
        \end{equation*}
    \end{enumerate}
    We will define $N_i$ inductively to satisfy the above conditions:
    \begin{itemize}
        \item For $i=0$ and $i=1$, the construction of $N_i$ is specfied as above.
        \item For $i=2$, note that since $N_1=g[M^*]$ by definition, we have (by Side Invariance) that
        \begin{equation*}
            \begin{tikzcd}
                g[M^*] \arrow[r, "\text{id}"] & N_1 \incarrow{r} \phanArrow{dr} & N\\
                M_0' \arrow[r, "\text{id}"] & M_b \incarrow{u} \incarrow{r} & M \incarrow{u}
            \end{tikzcd}
        \end{equation*}
        As $M_b\leq M_1'\leq M$, by regularity there exists some $N_2\leq N$ such that
        \begin{equation*}
            \begin{tikzcd}
                N_1 \incarrow{r} \phanArrow{dr} & N_2 \incarrow{r} \phanArrow{dr} & N\\
                M_b \incarrow{u} \incarrow{r} & M_1' \incarrow{u} \incarrow{r} & M \incarrow{u}
            \end{tikzcd}
        \end{equation*}
        \item If $1\leq i<\omega$, then by the inductive hypothesis, we have:
        \begin{equation*}
            \begin{tikzcd}
                N_1 \incarrow{r} \phanArrow{dr} & N_{i+1} \incarrow{r} \phanArrow{dr} & N\\
                M_b \incarrow{u} \incarrow{r} & M_i' \incarrow{u} \incarrow{r} & M \incarrow{u}
            \end{tikzcd}
        \end{equation*}
        As $M_i'\leq M_{i+1}'\leq M$, again by regularity there is some $N_{i+1}\leq N$ such that
        \begin{equation*}
            \begin{tikzcd}
                N_1 \incarrow{r} \phanArrow{dr} & N_{i+1} \incarrow{r} \phanArrow{dr} & N_{i+2} \incarrow{r} \phanArrow{dr} & N\\
                M_b \incarrow{u} \incarrow{r} & M_i' \incarrow{u} \incarrow{r} & M_{i+1}' \incarrow{u} \incarrow{r} & M \incarrow{u}
            \end{tikzcd}
        \end{equation*}
        Furthermore, apply regularity to the two commutative squares on the left, we also get:
        \begin{equation*}
            \begin{tikzcd}
                N_1 \incarrow{r} \phanArrow{dr} & N_{i+2} \incarrow{r} \phanArrow{dr} & N\\
                M_b \incarrow{u} \incarrow{r} & M_{i+1}' \incarrow{u} \incarrow{r} & M \incarrow{u}
            \end{tikzcd}
        \end{equation*}
        \item If $i=\omega$, then by the inductive hypothesis we have
        \begin{equation*}
            \begin{tikzcd}
                N_1 \incarrow{r} \arrow[dr, phantom, "\mathscr{A}"] & N_2 \incarrow{r} \arrow[dr, phantom, "\mathscr{A}"] & N_3 \incarrow{r} & \cdots \incarrow{r} & N_{i+1} \incarrow{r} \arrow[dr, phantom, "\mathscr{A}"] & N_{i+2} \incarrow{r} & \cdots\\
                M_b \incarrow{u} \incarrow{r} & M_1' \incarrow{u} \incarrow{r} & M_2' \incarrow{u} \incarrow{r} & \cdots \incarrow{r} & M_i' \incarrow{u} \incarrow{r} & M_{i+1}' \incarrow{u} \incarrow{r} & \cdots
            \end{tikzcd}
        \end{equation*}
        Defining $N_\omega=\bigcup_{i<\omega} N_i$, by continuity we have that
        \begin{equation*}
            \begin{tikzcd}
                N_1 \incarrow{r} \phanArrow{dr} & N_\omega\\
                M_b \incarrow{u} \incarrow{r} & M_\omega' \incarrow{u}
            \end{tikzcd}
        \end{equation*}
        Now, since $N_\omega\leq N$, by regularity (specifically, the ``moreover" part of condition (3), see Definition \ref{propdef}), it is also true that
        \begin{equation*}
            \begin{tikzcd}
                N_\omega \incarrow{r} \phanArrow{dr} & N\\
                M_\omega' \incarrow{u} \incarrow{r} & M \incarrow{u}
            \end{tikzcd}
        \end{equation*}
        \item For successor and limit $i$'s beyond $\omega$, the construction is the same as above except for the shifted indices.
    \end{itemize}
    Letting $N'=\bigcup_{i<\beta} N_i$, it remains to show that $N'$ is an $\scrA$-amalgam of $(M^*)^\frown(M_i)_{i<\alpha}$ over $M_b$ (via the maps $(g)^\frown (f_i)_{i<\alpha}$) i.e. that for each $i<\omega$ and $j$ such that $\omega\leq j<\alpha$,
    \begin{equation*}
        \begin{tikzcd}
            N_{i+1} \incarrow{r} \phanArrow{dr} & N_{i+2} & N_j \incarrow{r} \phanArrow{dr} & N_{j+1}\\
            M_b \incarrow{u} \incarrow{r} & M_i \arrow[u, "f_i"] & M_b \incarrow{u} \incarrow{r} & M_j \arrow[u, "f_j"]
        \end{tikzcd}
    \end{equation*}
    For the $i<\omega$ case, recall that $(M_i')_{i<\alpha}$ witnesses that $M$ is an $\scrA$-amalgam of $(M_i)_{i<\alpha}$ over $M_b$, and hence for each $i$,
    \begin{equation*}
        \begin{tikzcd}
            M_i' \incarrow{r} \phanArrow{dr} & M_{i+1}'\\
            M_b \incarrow{u} \incarrow{r} & M_i \arrow[u, "f_i"]
        \end{tikzcd}
    \end{equation*}
    Combining this with condition (4) above and the construction of $N_i$, we get the diagram
    \begin{equation*}
        \begin{tikzcd}
            N_1 \incarrow{r} \phanArrow{dr} & N_{i+1} \incarrow{r} \phanArrow{dr} & N_{i+2} \incarrow{r} \phanArrow{dr} & N\\
            M_b \arrow[dr, "\text{id}"] \incarrow{r} \incarrow{u} & M_i' \incarrow{u} \incarrow{r} \phanArrow{dr} & M_{i+1}' \incarrow{u} \incarrow{r} & M \incarrow{u}\\
            & M_b \incarrow{u} \incarrow{r} & M_i \arrow[u, "f_i"]
        \end{tikzcd}
    \end{equation*}
    Note that apply regularity to the two commutative squares in the middle column gives us the desired result. As the same argument applies to the case of $\omega\leq j<\alpha$ with shifted indices, this completes the proof.
\end{proof}

\begin{lem}\label{initdecom}
    Suppose $N$ is an $\scrA$-amalgam of $(M_i)_{i_<\alpha}$ over $M_b$ via the maps $(f_i:M_i\longrightarrow N)_{i<\alpha}$. If additionally $M_0$ is an $\scrA$-amalgam of $(L_j)_{j<\beta}$ over $M_b$ via the maps $(g_j:L_j\longrightarrow M_0)_{j<\beta}$, then $N$ is an $\scrA$-amalgam of the concatenated sequence $(L_j:j<\beta)^\frown (M_i:i<\alpha)$ over $M_b$.
\end{lem}

\begin{proof}
    Fix $(M_j')_{j<s(\beta)}$ a continuous resolution of $M_0$ witnessing that it is an $\scrA$-amalgam of $(L_j)_{j<\alpha}$ over $M_b$, and also fix $(N_i)_{i<s(\alpha)}$ a continuous resolution of $N$ witnessing that it is an $\scrA$-amalgam of $(M_i)_{i<s(\alpha)}$ over $M_b$. Consider then the sequence $S=(f_0[M_j'])_{j<s(\beta)}^\frown (N_i')_{1\leq i<s(\alpha)}$: it is a continuous resolution of $N$ since $N_0=f_0[M_0]=\bigcup_{j<s(\beta)}$. Since $f_0\upharpoonright M_j'$ is a $K$-isomorphism between $M_j'$, and $f_0[M_j']$, Invariance of $\scrA$ implies the desired result.
\end{proof}

\begin{lem}\label{wpsequniq}
    Suppose $\scrA$ is absolutely minimal. Let $N$ be an $\scrA$-amalgam of $(M_i)_{i<\alpha}$ over $M_b$ by inclusion, and suppose that $N'\geq N$, $N^*\leq N'$, and each $M_i\subseteq N^*$. Then $N\leq N^*$.
\end{lem}

\begin{proof}
    By induction on $\alpha$:
    \begin{itemize}
        \item If $\alpha=2$, then this is true by definition of $\scrA$ being absolutely minimal.
        \item Assuming the statement is true for $\alpha$. Given $N$ an $\scrA$-amalgam of $(M_i)_{i<\alpha+1}$ over $M_b$ (by inclusion) and $N', N^*$ as above, let $N_\alpha\leq N$ be such that $N_\alpha$ is an $\scrA$-amalgam of $(M_i)_{i<\alpha}$ over $M_b$, and hence by induction $N_\alpha\leq N^*$. But $N$ is an $\scrA$-amalgam of $N_\alpha, M_\alpha$ over $M_b$, and as $M_\alpha\leq N^*$ by assumption, hence $N\leq N^*$ as $\scrA$ is absolutely minimal.
        \item For limit $\delta$, if $N$ is an $\scrA$-amalgam of $(M_i)_{i<\delta}$ over $M_b$, then fix $(N_\alpha)_{\alpha<\delta}$ a continuous resolution of $N$ witnessing that $N$ is an $\scrA$-amalgam. In particular, each $N_\alpha$ is an $\scrA$-amalgam of $(M_i)_{i<\alpha}$ over $M_b$. Now, given $N', N^*$ as above, by induction each $N_\alpha\leq N^*$. Since $N=\bigcup_{\alpha<\delta}N_\alpha$, hence $N\subseteq N^*$. Furthermore, as $N, N^*\leq N'$, by Coherence we have that $N\leq N^*$.
    \end{itemize}
\end{proof}

\begin{cor}
    Suppose $\scrA$ is absolutely minimal. If each $M_i\leq N$ and the sequence $(M_i)_{i<\alpha}$ is $\scrA$-subamalgamated over $M_b$ inside $N$, then there is a unique $N'\leq N$ which is an $\scrA$-amalgam of $(M_i)_{i<\alpha}$ over $M_b$.
\end{cor}

\begin{nota}
    If $(M_i)_{i<\alpha}$ are such that each $M_b\leq M_i\leq N$ and the sequence $(M_i)_{i<\alpha}$ is $\scrA$-subamalgamated inside $N$ via inclusion, then we denote the unique $\scrA$-amalgam inside $N$ by $\bigoplus_{M_b, i<\alpha}^N M_i$.
\end{nota}

\begin{lem}\label{gridamal}
    Suppose $\scrA$ is regular, continuous, and absolutely minimal. Let $\alpha$ be a limit ordinal, and $(M^1_i)_{i\leq\alpha}, (M^2_i)_{i\leq\alpha}, (N_i)_{i\leq\alpha}$ be increasing continuous chains such that for each $i<\alpha$, $N_i$ is an $\scrA$-amalgam of $M^1_i, M^2_i$ over $M_b$ by inclusion. Then for any $i,j<\alpha$, there is $M_{ij}\leq N_{\text{max}(i,j)}$ which is an $\scrA$-amalgam of $M^1_i, M^2_j$ over $M_b$ by inclusion. Moreover, we can choose the system of $M_{ij}$'s such that if $i$ is a limit ordinal, then $M_{ij}=\bigcup_{k<i}M_{kj}$, and similarly if $j$ is a limit ordinal.
\end{lem}

\begin{proof}
    Let $M^1_i, M^2_i, N_i$ be as above, so that we have the diagram
    \begin{equation*}
        \begin{tikzcd}
            \vdots & & & \iddots\\
            M^2_1 \arrow[u] \arrow[rr] & & N_1 \arrow[ru]\\
            M^2_0 \arrow[u] \arrow[r] & N_0 \arrow[ru]\\
            M_b \arrow[u] \arrow[r] & M^1_0 \arrow[u] \arrow[r] & M^1_1 \arrow[uu] \arrow[r] & \cdots
        \end{tikzcd}
    \end{equation*}
    where all the arrows are inclusions and all the commutative squares with a vertex at $M_b$ are $\scrA$-amalgams. Letting $M_{ii}$ be defined as $N_i$, we will define $M_{ij}$ by induction on $\text{max}(i,j)<\alpha$ such that in addition to the requirements above, we have additionally that the condition $(A(ij))$ holds when $i,j$ are not limits:
    \begin{equation*}
        \begin{tikzcd}
            M_{i-1,j} \incarrow{r} \phanArrow{dr} & M_{i, j}\\
            M_{i-1,j-1} \incarrow{u} \incarrow{r} & M_{i,j-1} \incarrow{u}
        \end{tikzcd}\tag{A(i,j)}
    \end{equation*}
    (where $M_{-1,-1}=M_b, M_{-1, j}=M^2_j, M_{i,-1}=M^1_i$)
    \begin{itemize}
        \item For $M_{01}$, note that $N_1$ is an $\scrA$-amalgam of $M^1_1, M^2_1$ over $M_b$ (by inclusion). As $M_b\leq M^1_0\leq M^1_1$, by regularity there is $M_{01}\leq N_1$ such that
        \begin{equation*}
            \begin{tikzcd}
                M^2_1 \incarrow{r} \phanArrow{dr} & M_{01} \incarrow{r} \phanArrow{dr} & N_1\\
                M_b \incarrow{r} \incarrow{u} & M^1_0 \incarrow{r} \incarrow{u} & M^1_1 \incarrow{u}
            \end{tikzcd}
        \end{equation*}
        $M_{10}$ is defined symmetrically.
        \item If $M_{ij}$ is defined for all $i,j\leq\alpha$, then for any $i\leq\alpha$, by regularity there is $M_{i,\alpha+1}$ such that
        \begin{equation*}
            \begin{tikzcd}
                M^2_{\alpha+1} \incarrow{r} \phanArrow{dr} & M_{i,\alpha+1} \incarrow{r} \phanArrow{dr} & N_{\alpha+1}\\
                M_b \incarrow{r} \incarrow{u} & M^1_i \incarrow{r} \incarrow{u} & M^1_{\alpha+1} \incarrow{u}
            \end{tikzcd}
        \end{equation*}
        It is straightforward to see that condition $(A(i,\alpha+1))$ by induction on $i$ (and using regularity for the base case). We define $M_{\alpha+1, j}$ which satisfies $(A(\alpha+1,j))$ by the symmetric argument. Finally, to see that condition $(A(\alpha+1,\alpha+1))$ holds, note that by definition of $M_{\alpha,\alpha+1}$, we have
        \begin{equation*}
            \begin{tikzcd}
                M^2_{\alpha+1} \incarrow{r} \phanArrow{dr} & M_{\alpha,\alpha+1} \incarrow{r} \phanArrow{dr} & N_{\alpha+1}\\
                M_b \incarrow{r} \incarrow{u} & M^1_{\alpha} \incarrow{r} \incarrow{u} & M^1_{\alpha+1} \incarrow{u}
            \end{tikzcd}
        \end{equation*}
        Apply regularity to the commutative square on the right (and symmetry), we get that
        \begin{equation*}
            \begin{tikzcd}
                M^1_{\alpha+1} \incarrow{r} \phanArrow{dr} & M_{\alpha+1,\alpha} \incarrow{r} \phanArrow{dr} & N_{\alpha+1}\\
                M_b \incarrow{r} \incarrow{u} & N_{\alpha} \incarrow{r} \incarrow{u} & M_{\alpha,\alpha+1} \incarrow{u}
            \end{tikzcd}
        \end{equation*}
        The $\scrA$-amalgam on the right shows that $(A(\alpha+1, \alpha+1))$ is indeed satisfied.
        \item If $\delta$ is a limit and $M_{ij}$ are defined for $i,j<\delta$, then by regularity let $M_{i,\delta}\leq N_\delta$ be an $\scrA$-amalgam of $M^1_i, M^2_\delta$ over $M_b$. We need to show that:
        \begin{clm}
            $M_{i,\delta}=\bigcup_{j<\delta} M_{ij}$
        \end{clm}
        To prove the claim, note that since $(M^2_j)_{j<\delta}$ is a continuous resolution of $M^2_\delta$, by Lemma \ref{restransfer} there is a continuous resolution $(M'_{ij})_{j<\delta}$ of $M_{i,\delta}$ such that each $M'_{ij}$ is an $\scrA$-amalgam of $M^1_i, M^2_j$ over $M_b$. But then each $M'_{ij},M_{ij}\leq N_\delta$, and as $\scrA$ is absolutely minimal, by Lemma \ref{wpunique} we have that $M_{ij}=M'_{ij}$. This proves the claim. Additionally, this construction implies that when $\gamma$ is also a limit, then $M_{\gamma,\delta}=\bigcup_{i<\gamma}M_{i,\delta}$.
        
        Symmetrically, we define $M_{\delta,j}$. To finish the construction, we need to show that:
        \begin{equation*}
            \bigcup_{i<\delta}M_{i,\delta}=N_\delta=\bigcup_{j<\delta}M_{\delta,j}
        \end{equation*}
        But this is true since each $N_\alpha\leq M_{\alpha,\delta},M_{\delta, \alpha}\leq N_\delta$, and $\bigcup_{\alpha<\delta}N_\alpha=N_\delta$.
    \end{itemize}
\end{proof}

\begin{cor}\label{resamal}
    Suppose $\scrA$ is absolutely minimal, regular, and continuous. If $\alpha$ is a limit ordinal, and $(M^1_i)_{i\leq\alpha}, (M^2_i)_{i\leq\alpha}, (N_i)_{i\leq\alpha}$ are increasing continuous chains such that for each $i<\alpha$, $N_i$ is an $\scrA$-amalgam of $M^1_i, M^2_i$ over $M_b$ by inclusion, then $N_\alpha$ is an $\scrA$-amalgam of $M^1_\alpha, M^2_\alpha$ over $M_b$.
\end{cor}

\begin{proof}
    For $i,j<\alpha$, let $M_{ij}\leq N_\alpha$ be constructed as in the above Lemma, and for each $i<\alpha$, let $M_{i,\alpha}=\bigcup_{j<\alpha} M_{ij}$.
    \begin{clm}
        \begin{equation*}
            \begin{tikzcd}
                M^2_\alpha \incarrow{r} \phanArrow{dr} & M_{0,\alpha}\\
                M_b \incarrow{u} \incarrow{r} & M^1_0 \incarrow{u}
            \end{tikzcd}
        \end{equation*}
    \end{clm}
    \begin{proof}
        Note that by condition $(A(0,j))$ for each $j<\alpha$, we have that
        \begin{equation*}
            \begin{tikzcd}
                M^1_0 \incarrow{r} \phanArrow{dr} & M_{0,0} \incarrow{r} \phanArrow{dr} & M_{0,1} \incarrow{r} & \cdots \incarrow{r} & M_{0,i} \incarrow{r} \phanArrow{dr} & M_{0,i+1} \incarrow{r} & \cdots\\
                M_b \incarrow{u} \incarrow{r} & M^2_0 \incarrow{u} \incarrow{r} & M^2_1 \incarrow{u} \incarrow{r} & \cdots \incarrow{r} & M^2_i \incarrow{u} \incarrow{r} & M^2_{i+1} \incarrow{u} \incarrow{r} & \cdots
            \end{tikzcd}
        \end{equation*}
        Hence the claim holds as $\scrA$ is continuous.
    \end{proof}
    \begin{clm}
        For each $i<\alpha$,
        \begin{equation*}
            \begin{tikzcd}
                M_{i,\alpha} \incarrow{r} \phanArrow{dr} & M_{i+1,\alpha}\\
                M^1_i \incarrow{u} \incarrow{r} & M^1_{i+1} \incarrow{u}
            \end{tikzcd}
        \end{equation*}
    \end{clm}
    This holds by the same argument.
    \begin{clm}
        For limit $\delta<\alpha$, $M_{\delta,\alpha}=\bigcup_{i<\delta} M_{i,\alpha}$, and moreover
        \begin{equation*}
            \begin{tikzcd}
                M^2_\alpha \incarrow{r} \phanArrow{dr} & M_{\delta,\alpha}\\
                M_b \incarrow{u} \incarrow{r} & M^1_\delta \incarrow{u}
            \end{tikzcd}
        \end{equation*}
    \end{clm}
    \begin{proof}
        Note that
        \begin{equation*}
            \bigcup_{i<\delta} M_{i,\alpha}=\bigcup_{i<\delta}\bigcup_{j<\alpha}M_{ij}=\bigcup_{j<\alpha}M_{\delta,j}=M_{\delta,\alpha}
        \end{equation*}
        For the moreover part, combining the above claims and induction along $\delta$, we get that
        \begin{equation*}
            \begin{tikzcd}
                M^2_\alpha \incarrow{r} \phanArrow{dr} & M_{0,\alpha} \incarrow{r} \phanArrow{dr} & M_{1,\alpha} \incarrow{r} & \cdots \incarrow{r} & M_{i,\alpha} \incarrow{r} \phanArrow{dr} & M_{i+1,\alpha} \incarrow{r} & \cdots\\
                M_b \incarrow{u} \incarrow{r} & M^1_0 \incarrow{u} \incarrow{r} & M^1_1 \incarrow{u} \incarrow{r} & \cdots \incarrow{r} & M^1_i \incarrow{u} \incarrow{r} & M^1_{i+1} \incarrow{u} \incarrow{r} & \cdots
            \end{tikzcd}
        \end{equation*}
        As $\scrA$ is continuous, hence $\bigcup_{i<\delta} M_{i,\alpha}=M_{\delta,\alpha}$ is an $\scrA$-amalgam of $M^2_\alpha,M^1_\delta$ over $M_b$.
    \end{proof}
    Combining the above claims, we get the diagram (for all $i<\alpha$)
    \begin{equation*}
        \begin{tikzcd}
            M^2_\alpha \incarrow{r} \phanArrow{dr} & M_{0,\alpha} \incarrow{r} \phanArrow{dr} & M_{1,\alpha} \incarrow{r} & \cdots \incarrow{r} & M_{i,\alpha} \incarrow{r} \phanArrow{dr} & M_{i+1,\alpha} \incarrow{r} & \cdots\\
            M_b \incarrow{u} \incarrow{r} & M^1_0 \incarrow{u} \incarrow{r} & M^1_1 \incarrow{u} \incarrow{r} & \cdots \incarrow{r} & M^1_i \incarrow{u} \incarrow{r} & M^1_{i+1} \incarrow{u} \incarrow{r} & \cdots
        \end{tikzcd}
    \end{equation*}
    As $\scrA$ is continuous, hence $\bigcup_{i<\alpha}M_{i,\alpha}$ is an $\scrA$-amalgam of $M^1_\alpha, M^2_\alpha$ over $M_b$. But since for any $i<j<\alpha$, $N_i\leq M_{ij},M_{ji}\leq N_j$, we have that $N_\alpha=\bigcup_{i<\alpha}M_{i,\alpha}$. This completes the proof.
\end{proof}

\begin{thm}\label{subseq}
    Suppose $\scrA$ is absolutely minimal, regular, and continuous. Let $N$ be an $\scrA$-amalgam of $(M_i)_{i<\alpha}$ over $M_b$ by inclusion. Then for any subsequence $S\subseteq\alpha$, there is some $M_S\leq N$ which is an $\scrA$-amalgam of $(M_{S(j)})_{j<|S|}$ over $M_b$ by inclusion. Moreover, if $\bar{S}$ is the complement of $S$ in $\alpha$ (and considered as an increasing sequence), then there is $M_{\bar{S}}$ such that additionally, $N$ is an $\scrA$-amalgam of $M_S, M_{\bar{S}}$ over $M_b$ by inclusion.
\end{thm}

\begin{proof}
    We will proceed by induction on the length of $\alpha$:
    \begin{itemize}
        \item When $\alpha=2$, this is trivial.
        \item Assume the claim holds for $\alpha$. Given $N$ an $\scrA$-amalgam of $(M_i)_{i<\alpha+1}$ over $M_b$, suppose that $S$ is a subsequence of $\alpha+1$. This breaks down into three cases:
        \begin{enumerate}
            \item If $S=\{\alpha\}$, then the case is trivial.
            \item If $S\subseteq \alpha$, then consider $N_\alpha\leq N$ which is an $\scrA$-amalgam of $(M_i)_{i<\alpha}$ over $M_b$ (as guaranteed by Lemma \ref{inittele}): by the inductive hypothesis, $M_S\leq N_\alpha$ exists, and so does $M_{\bar{S}_\alpha}$, where $\bar{S}_\alpha$ is the complement of $S$ w.r.t. $\alpha$. Now, since $N$ is an $\scrA$-amalgam of $N_\alpha$ and $M_\alpha$ over $M_b$, we get the $\scrA$-amalgams
            \begin{equation*}
                \begin{tikzcd}
                    M_S \incarrow{r} \phanArrow{dr} & N_\alpha & N_\alpha \incarrow{r} \phanArrow{dr} & N\\
                    M_b \incarrow{r} \incarrow{u} & M_{\bar{S}_\alpha} \incarrow{u} & M_b \incarrow{r} \incarrow{u} & M_\alpha \incarrow{u}
                \end{tikzcd}
            \end{equation*}
            By Lemma \ref{3rotation}, hence there is some $M_{\bar{S}}$ such that:
            \begin{itemize}
                \item $M_{\bar{S}}$ is an $\scrA$-amalgam of $M_{\bar{S}_\alpha}, M_\alpha$ over $M_b$; and
                \item $N$ is an $\scrA$-amalgam of $M_S, M_{\bar{S}}$ over $M_b$
            \end{itemize}
            Furthermore, by Lemma \ref{initdecom}, $M_{\bar{S}}$ is also an $\scrA$-amalgam of $(M_j:j<|\bar{S}_\alpha|)^\frown(M_\alpha)$ over $M_b$.
            \item If $S\ni\alpha$ and $S\cap\alpha\neq\varnothing$, then $\bar{S}$ satisfies the above case (2), so the same construction gives the required submodels.
        \end{enumerate}
        \item Let $\delta$ be a limit, and suppose the claim holds for all $\alpha<\delta$. Given $N$ an $\scrA$-amalgam of $(M_i)_{i<\delta}$ over $M_b$ by inclusion, let $(N_i)_{i<\delta}$ be a continuous resolution of $N$ such that each $N_i$ is an $\scrA$-amalgam of $(M_j)_{j<i}$ over $M_b$. Now, if $S$ is a subsequence of $\delta$, denote $S_\alpha:=S\upharpoonright\alpha$ and $\bar{S}_\alpha:=\bar{S}\upharpoonright\alpha$. Note then that for each $\alpha$, $\bar{S}_\alpha$ is the complement of $S_\alpha$ relative to $\alpha$, and hence the inductive hypothesis implies that there are models $M^\alpha_{S_\alpha},M^\alpha_{\bar{S}_\alpha}\leq N_\alpha$ such that:
        \begin{itemize}
            \item $M^\alpha_{S_\alpha}$ is an $\scrA$-amalgam of $(M_{S(j)})_{j<|S_\alpha|}$ over $M_b$
            \item $M^\alpha_{\bar{S}_\alpha}$ is an $\scrA$-amalgam of $(M_{\bar{S}(j)})_{j<|\bar{S}_\alpha|}$ over $M_b$
            \item $N_\alpha$ is an $\scrA$-amalgam of $M^\alpha_{S_\alpha},M^\alpha_{\bar{S}_\alpha}$ over $M_b$
        \end{itemize}
        Moreover, by Lemma \ref{wpsequniq}, $M^\alpha_{S_\alpha}$ is the unique $\scrA$-amalgam of $(M_{S(j)})_{j<|S|}$ over $M_b$ inside $N_\delta$, and similarly for $M^\alpha_{\bar{S}_\alpha}$. Hence we will drop the superscript, and define $M_S:=\bigcup_{\alpha<\delta} M_{S_\alpha},M_{\bar{S}}:=\bigcup_{\alpha<\delta} M_{\bar{S}_\alpha}$. Note then that the chains $(M_{S_\alpha})_{\alpha\leq\delta},(M_{\bar{S}_\alpha})_{\alpha\leq\delta},(N_{\alpha})_{\alpha\leq\delta}$ satisfies the hypothesis of Corollary \ref{resamal} above, and hence $N_\delta$ is an $\scrA$-amalgam of $M_S,M_{\bar{S}}$ over $M_b$. Moreover, the continuous resolution $(M_{S_\alpha})_{\alpha<\delta}$ witnesses that $M_S$ is an $\scrA$-amalgam of $(M_{S(j)})_{j<|S|}$ over $M_b$ as desired.
    \end{itemize}
\end{proof}

\begin{rem}
    It should be noted that if $K$ is assumed to be an AEC rather than a weak AEC (i.e. $K$ has Smoothness), then the proof of the above theorem can be simplified considerably: if $N$ is an $\scrA$-amalgam of $(M_i)_{i<\alpha}$ over $M_b$ by inclusion and $S$ is a subsequence of $\alpha$, then the $\scrA$-amalgam of $(M_{S(i)})_{i<\otp(S)}$ over $M_b$ can be easily defined by induction. This works even at limit stages when $K$ is assumed to have Smoothness; otherwise, the above argument seems to be necessary.
\end{rem}

\begin{thm}\label{setamal}
    Suppose $\scrA$ is absolutely minimal, regular, and continuous. Let $\alpha\geq 2$ be an ordinal, and $\sigma:|\alpha|\longrightarrow\alpha$ be any enumeration of $\alpha$. Then $N$ is an $\scrA$-amalgam of $(M_i)_{i<\alpha}$ over $M_b$ by inclusion iff $N$ is also an $\scrA$-amalgam of $(M_{\sigma(j)})_{j<|\alpha|}$.
\end{thm}

\begin{proof}
    Let $N$ be an $\scrA$-amalgam of $(M_i)_{i<\alpha}$ over $M_b$ by inclusion. We proceed by induction on $\alpha$:
    \begin{itemize}
        \item When $\alpha=2$, this is just Lemma \ref{comasso}.
        \item Suppose the claim holds for $n$, which is finite. If $\sigma$ is an enumeration of $n+1$, then $\sigma$ is a permutation of $n+1$. There are two cases to consider:
        \begin{itemize}
            \item If $\sigma(n)=n$, then $\sigma\upharpoonright n$ is a permutation of $n$, and the claim follows from the inductive hypothesis.
            \item Otherwise, let $m=\sigma(n)<n$. By Corollary \ref{inittailtele}, there are models $N_1, N_2\leq N$ such that:
            \begin{itemize}
                \item $N_1$ is an $\scrA$-amalgam of $(M_i)_{i<m}$ over $M_b$
                \item $N_2$ is an $\scrA$-amalgam of $(M_i)_{m<i\leq n}$ over $M_b$
                \item $N$ is an $\scrA$-amalgam of $(N_1, M_m, N_2)$ over $M_b$
            \end{itemize}
            But then by Lemma \ref{comasso}, $N$ is also an $\scrA$-amalgam of $(N_1, N_2, M_m)$ over $M_b$. Now, if $N'\leq N$ is an $\scrA$-amalgam of $N_1, N_2$ over $M_b$, then by Lemma \ref{initdecom} and \ref{taildecom}, $N'$ is an $\scrA$-amalgam of $(M_j)_{j\neq m}$ over $M_b$. Since $\sigma(n)=m$ and $\sigma$ is a permutation, (by re-indexing) the inductive hypothesis implies that $N'$ is also an $\scrA$-amalgam of $(M_{\sigma(i)})_{i<n}$ over $M_b$, and hence $N$ is an $\scrA$-amalgam of $(M_{\sigma(i)})_{i<n+1}$ over $M_b$.
        \end{itemize}
        \item Suppose the claim holds for an infinite $\alpha$, and so $|\alpha|=|\alpha+1|$. Given $\sigma:|\alpha|\longrightarrow \alpha+1$ an enumeration, there is some $\beta<|\alpha|$ such that $\sigma(\beta)=\alpha$. Let $S$ be the subsequence of $\alpha$ such that $\text{ran }S=\text{ran }\sigma\upharpoonright\beta$, and let $\bar{S}$ be its complement in $\alpha$, so in particular $\bar{S}=S'^\frown\alpha$ for some subsequence $S'$ of $\alpha$. Now, since $N$ is an $\scrA$-amalgam of $(M_i)_{i<\alpha+1}$ over $M_b$, there is an $N^*\leq N$ such that $N^*$ is an $\scrA$-amalgam of $(M_i)_{i<\alpha}$ over $M_b$, and $N$ is an $\scrA$-amalgam of $N^*, M_\alpha$ over $M_b$. But since $S$ is a subsequence of $\alpha$ and $S'$ is its complement w.r.t. $\alpha$, by Theorem \ref{subseq} there are models $N_S, N_{S'}\leq N^*$ such that
        \begin{itemize}
            \item $N_S$ is an $\scrA$-amalgam of $(M_{S(i)})_{i<|S|}$ over $M_b$
            \item $N_{S'}$ is an $\scrA$-amalgam of $(M_{S'(i)})_{i<|S'|}$ over $M_b$
            \item $N^*$ is an $\scrA$-amalgam of $N_S, N_{S'}$ over $M_b$
        \end{itemize}
        Furthermore, since $S, S'$ are subsequences of $\alpha$, $\otp(S),\otp(S')\leq\alpha$, and so by the inductive hypothesis $N_S$ is also an $\scrA$-amalgam of $(M_{\sigma(i)})_{i<\beta}$ over $M_b$. Similarly, $N_{S'}$ is an $\scrA$-amalgam of $(M_{\sigma(i)})_{\beta<i<|\alpha|}$. Moreover, $N$ is also an $\scrA$-amalgam of $(N_S, M_\alpha, N_S')$ over $M_b$ by Lemma \ref{comasso}, and so by Lemma \ref{initdecom} and \ref{taildecom}, $N$ is indeed an $\scrA$-amalgam of $(M_{\sigma(i)})_{i<|\alpha|}$ over $M_b$.
        We also need to show that if $N$ is an $\scrA$-amalgam of $(M_i)_{i<|\alpha|}$ over $M_b$, then $N$ is also an $\scrA$-amalgam of $(M_{\sigma^{-1}(j)})_{j<\alpha+1}$. Again letting $\beta$ be such that $\sigma(\alpha)=\beta$, by Lemma \ref{inittailtele} there are models $N^1, N^2\leq N$ such that
        \begin{itemize}
            \item $N^1$ is an $\scrA$-amalgam of $(M_i)_{i<\beta}$ over $M_b$
            \item $N^2$ is an $\scrA$-amalgam of $(M_i)_{\beta<i<|\alpha|}$ over $M_b$
            \item $N$ is an $\scrA$-amalgam of $(N^1, M_\beta, N^2)$ over $M_b$
        \end{itemize}
        By Lemma \ref{comasso} again, we see that $N$ is also an $\scrA$-amalgam of $(N^1, N^2, M_\beta)$ over $M_b$. If $N'\leq N$ is such that $N'$ is an $\scrA$-amalgam of $N^1, N^2$ over $M_b$, then by the inductive hypothesis $N'$ is also an $\scrA$-amalgam of $(M_{\sigma^{-1}(j)})_{j<\alpha}$ over $M_b$. By Lemma \ref{initdecom}, hence $N$ is an $\scrA$-amalgam of $(M_{\sigma^{-1}(j)})_{j<\alpha+1}$ over $M_b$.
        \item Suppose $\alpha$ is a limit ordinal, and that the claim holds for all $\beta<\alpha$. As $N$ is an $\scrA$-amalgam of $(M_i)_{i<\alpha}$ over $M_b$, let $(N_\beta)_{\beta<\alpha}$ be a continuous resolution of $N$ such that each $N_\beta$ is an $\scrA$-amalgam of $(M_i)_{i<\beta}$ over $M_b$. Now, given $\sigma:|\alpha|\longrightarrow\alpha$ an enumeration, for $j<|\alpha|$ let $\sigma_j:=\sigma\upharpoonright j$, and let $S_j$ be a subsequence of $\alpha$ such that $\text{ran }S_j=\text{ran }\sigma_j$ i.e. $S_j$ is the set enumerated by $\sigma_j$ but re-indexed by the ordinal ordering. Note that since each $S_j$ is a subsequence of $\alpha$, by Theorem \ref{subseq} there is $N_{S_j}\leq N$ which is an $\scrA$-amalgam of $(M_{S_j(i)})_{i<\otp(S_j)}$ over $M_b$. Furthermore, since each $|S_j|<|\alpha|$, $\otp(S_j)<\alpha$, and hence by the inductive hypothesis $N_{S_j}$ is also an $\scrA$-amalgam of $(M_{\sigma_j(i)})_{i<j}$ over $M_b$. Letting $N'=\bigcup_{j<|\alpha|}N_{S_j}$, this implies that $N'$ is an $\scrA$-amalgam of $(M_{\sigma(i)})_{i<|\alpha|}$ over $M_b$.
        \begin{clm}
            $N'=N$
        \end{clm}
        \begin{proof}
            Since $(N_i)_{i<\alpha}$ is a continuous resolution of $N$, it suffices to show that each $N_i\subseteq N'$. Now, for each $i<\alpha$, let $\zeta_i$ be a subsequence of $\sigma$ such that $\text{ran }\zeta_i=i$, and so by the inductive hypothesis $N_i$ is an $\scrA$-amalgam of $(M_{\zeta_i(j)})_{j<\otp(\zeta_i)}$ over $M_b$. But by Theorem \ref{subseq} there is $N'_i\leq N'$ which is an $\scrA$-amalgam of $(M_{\zeta_i(j)})_{j<\otp(\zeta_i)}$ over $M_b$, and as $\scrA$ is absolutely minimal hence $N'_i=N_i$. This proves the claim. 
        \end{proof}
        It remains to show, that when $\alpha$ is \underline{not} an initial ordinal, that if $N$ is an $\scrA$-amalgam of $(M_i)_{i<|\alpha|}$ over $M_b$, then it is also an $\scrA$-amalgam of $(M_{\sigma^{-1}(j)})_{j<\alpha}$. However, we note that the argument analogous to the one given above also works here, and hence the claim is proven for $\alpha$.
    \end{itemize}
\end{proof}

Given Theorems \ref{subseq} and \ref{setamal}, we see that when $\scrA$ is absolutely minimal, regular, and continuous, then $\scrA$-amalgamation of models indexed by a sequence is independent of the ordering, and hence can be considered as being indexed by a set. Moreover, if $N$ is an $\scrA$-amalgam of $(M_i)_{i\in Y}$ over $M_b$ by inclusion, then for any $X\subseteq Y$, there is $N_X\leq N$ which is an $\scrA$-amalgam of $(M_i)_{i\in X}$ over $M_b$.

Before moving on from sequential amalgamation, let us note that when $\scrA$ is additionally assumed to admit decomposition, this actually allows a model to be decomposed as the $\scrA$-amalgam of a sequence of small models:

\begin{lem}\label{amalcardseq}
    Suppose $\scrA$ is a notion of amalgamation which is regular and admits decomposition. Then for any $M_b\lneq N$, there exists an ordinal $\alpha<|N|^+$ and a sequence of models $(M_i)_{i<\alpha}$ such that:
    \begin{itemize}
        \item For every $i<\alpha$, $M_b\lneq M_i \leq N$ and $|M_i|=\LS(K)+|M_b|$
        \item $N$ is an $\scrA$-amalgam of $(M_i)_{i<\alpha}$ over $M_b$ by inclusion.
    \end{itemize}
\end{lem}

\begin{proof}
    Let $\lambda=|N|^+$ and $\mu=|M_b|+\LS(K)$. We will try to define two sequences of models, $(M_i)_{i<\lambda}$ and $(N_i)_{i<\lambda}$, such that:
    \begin{enumerate}
        \item For each $i$, $M_b\leq M_i\leq N$, $N_i\leq N$, and $|M_i|=\mu$
        \item $(N_i)_{i<\lambda}$ is an increasing continuous chain with $N_0=M_b$ and $N_1=M_0$
        \item For every $i\geq 1$, the following diagram is an $\scrA$-diagram:
        \begin{equation*}
            \begin{tikzcd}
                N_i \incarrow{r} \phanArrow{dr} & N_{i+1}\\
                M_b \incarrow{u} \incarrow{r} & M_i \incarrow{u}
            \end{tikzcd}
        \end{equation*}
    \end{enumerate}
    Proceeding inductively:
    \begin{itemize}
        \item For $i=0$, let $M_0$ be any model such that $M_b\lneq M_0\leq N$ and $|M_0|=\mu$, and let $N_0=M_b, N_1=M_0$.
        \item Suppose inductively that $M_i, N_{i+1}$ has been defined to satisfy (3). Since $N_{i+1}\leq N$, either $N_{i+1}=N$ or $N_{i+1}\lneq N$. In the former case, we terminate the inductive construction; otherwise, since $\scrA$ admits decomposition, there is some $M_{i+1}'$ such that
        \begin{equation*}
            \begin{tikzcd}
                N_{i+1} \incarrow{r} \phanArrow{dr} & N\\
                M_b \incarrow{u} \incarrow{r} & M_{i+1}' \incarrow{u}
            \end{tikzcd}
        \end{equation*}
        Note that as $\scrA$ is minimal, $M_{i+1}'-N_{i+1}$ must be nonempty as otherwise $N_{i+1}=N$. So let $M_{i+1}$ be any model of cardinality $\mu$ such that $M_b\lneq M_{i+1}\leq M_{i+1}'$ and $M_{i+1}-N_{i+1}$ is nonempty. Then, as $M_b\leq M_{i+1}\leq M_{i+1}'$, by regularity there exists some $N_{i+2}$ such that
        \begin{equation*}
            \begin{tikzcd}
                N_{i+1} \incarrow{r} \phanArrow{dr} & N_{i+2} \incarrow{r} \phanArrow{dr} & N\\
                M_b \incarrow{u} \incarrow{r} & M_{i+1} \incarrow{u} \incarrow{r} & M_{i+1}' \incarrow{u}
            \end{tikzcd}
        \end{equation*}
        \item For limit $\delta$, let $N_\delta=\bigcup_{i<\delta} N_i$. If $N=N_\delta$, then the construction terminates; otherwise, $M_\delta$ and $N_{\delta+1}$ can be defined by the same procedure as in the successor case.
    \end{itemize}
    Note that by construction, each $N_i\lneq N_{i+1}\leq N$, and as $\lambda=|N|^+$, hence the above procedure must terminate at some ordinal $\alpha<\lambda$. In that case, $(N_i)_{i<s(\alpha)}$ witnesses the fact that $N$ is an $\scrA$-amalgam of $(M_i)_{i<\alpha}$ over $M_b$ by inclusion.
\end{proof}

One last but important property of the direct sum in vector spaces and divisible groups is that under any ``basis" decomposition, any element is contained within the ``span" (or amalgam) of finitely many basis elements. Whilst this is clearly true in the two examples because such algebraic objects are finitary, in the present context we are also interested in classes which are infinitary but not unboundedly; analogously, there are interesting classes which are $\mathscr{L}_{\kappa,\omega}$ classes rather than just a $\mathscr{L}_{\infty, \omega}$ class. To this end, we will define a cardinal $\mu(K)$ by:

\begin{defn}\label{mudef}
    Suppose that $\scrA$ is a notion of amalgamation which is regular, continuous, absolutely minimal, and admits decomposition.
    \begin{enumerate}
        \item For $M\in K$ and $a\in M$, we define $\mu(a, M)$ to be the least cardinal $\mu$ such that: for any $M_b\leq M$ and any sequence $(M_i)_{i<\alpha}$ such that $M$ is the $\scrA$-amalgam of $(M_i)_{i<\alpha}$ over $M_b$ by inclusion, there is a subsequence $S\subseteq\alpha$ with $|S|<\mu$ such that $a\in\bigoplus^M_{M_b,j\in S} M_j$.
        \item We define $\mu(M):=\text{sup}\{\mu(a, M): a\in M\}$
        \item We define $\mu(K):=\text{sup}\{\mu(M): M\in K\}$ if it exists, or $\mu(K)=\infty$ otherwise.
        \item If $\mu(K)<\infty$, then we define $\mu_r(K)$ to be the least regular cardinal $\geq\mu(K)$.
    \end{enumerate}
\end{defn}

\begin{rem}
    Strictly speaking, $\mu(K)$ should be considered as $\mu^{\scrA}(K)$ since the definition depends on $\scrA$ and different notions of amalgamation might give rise to different values of $\mu(K)$. However, since in this paper we will always be considering a class $K$ with a fixed notion of amalgamation $\scrA$, we have chosen to suppress the extra notation.
\end{rem}

\section{An Independence Relation defined from $\scrA$}

In the elementary class of algebraically closed fields with characteristic 0, the forking relationship can be easily understood in terms of transcendental degrees: $\gtp(\bar{a}/F_1)$ does not fork over $F_0$ iff $\text{td}(\bar{a}/F_1)=\text{td}(\bar{a}/F_0)$. Since this is essentially a characterization of forking using the concept of bases, we would expect that a suitably well-behaved notion of amalgamation would also give rise to a forking-like independence relation. To that end, we define:

\begin{defn}\label{ancdef}
    Suppose $K$ is an AEC, $\mathscr{A}$ is a notion of amalgamation in $K$. We define a notion of $\mathscr{A}$-independence, denoted by $\dnf$, as follows: if $M\leq N$ and $A, B\subseteq N$, then $A\dnf_M^N B$ if there exists models $M_1, M_2$ with $M\leq M_1, M_2\leq N$ such that $A\subseteq M_1, B\subseteq M_2$, and $M_1, M_2$ are $\mathscr{A}$-subamalgamated inside $N$ over $M$ i.e. there is some $N'\leq N$ such that
    \begin{equation*}
        \begin{tikzcd}
            M_2 \incarrow{r} \arrow[dr, phantom, "\mathscr{A}"] & N' \incarrow{r} & N\\
            M \incarrow{u} \incarrow{r} & M_1 \incarrow{u}
        \end{tikzcd}
    \end{equation*}
    In such a case, we say that the pair $(M_1, M_2)$ is a \textbf{witness} to $A\dnf_{M}^N B$.
\end{defn}

Our goal here is to establish the conditions necessary for $\dnf$ to behave as forking for stationary types in a simple first order theories: To that end, we need to establish that $\dnf$ has the defining properties of forking:
\begin{itemize}
    \item Invariance
    \item Top monotonicity (i.e. forking does not depend on the ambient model)
    \item Right monotonicity
    \item Base monotonicity
    \item Symmetry
    \item Transitivity
    \item Existence of nonforking extensions
    \item Continuity
    \item $\kappa$-ary character for some cardinal $\kappa$
    \item Uniqueness of nonforking extensions
\end{itemize}
We will show that these properties hold through a series of propositions.

\begin{prop}[Top Monotonicity]\label{topmono}
    Let $\scrA$ be a notion of amalgamation.
    \begin{enumerate}
        \item If $A\dnf_M^N B$ and $N'\geq N$, then $A\dnf_M^{N'} B$
        \item Suppose that $K$ admits finite intersection and $\scrA$ is regular, minimal. If $A\dnf_M^N B$ and $N'\leq N$ is such that $A, B, M\subseteq N'$, then $A\dnf_M^{N'} B$
    \end{enumerate}
\end{prop}

\begin{proof}
    That (1) is true is straightforward from the definition of $\dnf$. For (2), let $(M_1, M_2)$ witness that $A\dnf_M^N B$; as $K$ has FI and $M\leq M_1, M_2, N'\leq N$, both $M_1\cap N'$ and $M_2\cap N'$ are models of $K$, and by regularity $M_1\cap N', M_2\cap N'$ are $\scrA$-subamalgamated over $M$ inside $N$. Since $K$ admits finite intersection and $\scrA$ is minimal, hence $\scrA$ is absolutely minimal by Lemma \ref{FIminwp}, and so in particular $(M_1\cap N')\oplus_M^N (M_2\cap N')\leq N'$. Hence $(M_1\cap N', M_2\cap N')$ is a witness to $A\dnf_M^{N'} B$.
\end{proof}

Some straightforward observations which follow from the definition of $\dnf$ are:

\begin{prop}\label{ancmulti}
    Let $\scrA$ be a notion of amalgamation
    \begin{enumerate}
        \item\label{exist} (Existence) For any $M\leq N$ and $A\subseteq N$, $A\dnf_M^N M$.
        \item\label{sym} (Symmetry) $A \dnf_M^N B$ implies $B\dnf_M^N A$.
        \item\label{rmono} (Right Monotonicity) If $A\dnf_M^N B$ and $B'\subseteq B$, then $A\dnf_M^N B'$.
        \item\label{normal} (Right Normality) $A\dnf_M^N B$ iff $A\dnf_M^N (B\cup M)$
    \end{enumerate}
\end{prop}

\begin{prop}[Base Monotonicity]\label{basemono}
    Suppose $\scrA$ is regular. If $M'$ is such that $M\leq M'\subseteq B$ and $A\dnf_M^N B$, then $A\dnf_{M'}^N B$.
\end{prop}

\begin{proof}
    Let $(M_1, M_2)$ witness that $A\dnf_M^N B$. In particular, this implies that there is some $N'\leq N$ such that
    \begin{equation*}
        \begin{tikzcd}
            M_1\incarrow{r} \phanArrow{dr} & N'\\
            M \incarrow{u} \incarrow{r} & M_2 \incarrow{u}
        \end{tikzcd}
    \end{equation*}
    Since $M\leq M'\leq M_2$ (as $M'\subseteq B\subseteq M_2$), by regularity there exists some $N''\leq N'$ with
    \begin{equation*}
        \begin{tikzcd}
            M_1\incarrow{r} \phanArrow{dr} & N''\incarrow{r} \phanArrow{dr} & N'\\
            M \incarrow{u} \incarrow{r} & M' \incarrow{u} \incarrow{r} & M_2 \incarrow{u}
        \end{tikzcd}
    \end{equation*}
    Since $A\subseteq M_1\leq N''$, hence $(N'', M_2)$ is a witness to $A\dnf_{M'}^N B$.
\end{proof}

\begin{lem}
    If $\scrA$ is regular and $M_0\leq M_1$, then $A\dnf_{M_0}^N M_1$ iff there is some $M_2\leq N$ such that $M_0\leq M_2$, $A\subseteq M_2$ and $M_1, M_2$ are $\scrA$-subamalgamated over $M_0$ inside $N$.
\end{lem}

\begin{proof}
    For the reverse direction, note that $(M_2, M_1)$ is a witness to $A\dnf_{M_0}^N M_1$. For the forward direction, let $(M_2, M')$ witness that $A\dnf_{M_0}^N M_1$, and so in particular $M_0\leq M_1\leq M'$. Hence by regularity, $M_1, M_2$ are also $\scrA$-subamalgamated over $M_0$ inside $N$.
\end{proof}

\begin{prop}[Transitivity]\label{transi}
    Suppose $K$ admits finite intersection and $\scrA$ is regular, absolutely minimal. If $M_0\leq M_1\leq M_2\leq N$ and $A\subseteq N$ is such that $A\dnf_{M_0}^N M_1$ and $A\dnf_{M_1}^N M_2$, then $A\dnf_{M_0}^N M_2$.
\end{prop}

\begin{proof}
    By the above lemma, there exists $M', M''\leq N$ such that $(M', M_1)$ witnesses $A\dnf_{M_0}^N M_1$ and $(M'', M_2)$ witnesses $A\dnf_{M_1}^N M_2$. Hence, there are also models $M'\oplus_{M_0}^N M_1, M''\oplus_{M_1}^N M_2\leq N$ such that:
    \begin{equation*}
        \begin{tikzcd}
            M' \incarrow{r} \phanArrow{dr} & M'\oplus_{M_0}^N M_1 & M'' \incarrow{r} \phanArrow{dr} & M''\oplus_{M_1}^N M_2\\
            M_0 \incarrow{u} \incarrow{r} & M_1 \incarrow{u} & M_1 \incarrow{u} \incarrow{r} & M_2 \incarrow{u}
        \end{tikzcd}
    \end{equation*}
    Since $K$ has FI and $M_0\leq M', M''\leq N$, there is a model $M^*:=M'\cap M''$, and in particular $M_0\leq M^*, A\subseteq M^*$. So by regularity, $M^*, M_1$ are also $\scrA$-subamalgamated over $M_0$ inside $N$ i.e.
    \begin{equation*}
        \begin{tikzcd}
            M_1 \incarrow{r} \phanArrow{dr} & M_1\oplus_{M_0}^N M^* \incarrow{r} \phanArrow{dr} & M_1\oplus_{M_0}^N M'\\
            M_0 \incarrow{u} \incarrow{r} & M^* \incarrow{u} \incarrow{r} & M' \incarrow{u}
        \end{tikzcd}
    \end{equation*}
    Note that since $M_1, M^*\leq M''$, $M_1\oplus_{M_0}^N M^*\leq M''$ as $\scrA$ is absolutely minimal. Therefore, again by regularity, $(M_1\oplus_{M_0}^N M^*), M_2$ are $\scrA$-subamalgamated over $M_1$ inside $N$, so there is some $M^{**}$ such that:
    \begin{equation*}
        \begin{tikzcd}
            M_2 \incarrow{r} \phanArrow{dr} & M^{**} \incarrow{r} \phanArrow{dr} & M_2\oplus_{M_0}^N M''\\
            M_1 \incarrow{u} \incarrow{r} & M_1\oplus_{M_0}^N M^* \incarrow{u} \incarrow{r} & M'' \incarrow{u}
        \end{tikzcd}
    \end{equation*}
    Combining the commutative squares on the left of the two diagrams, we get that
    \begin{equation*}
        \begin{tikzcd}
            M^* \incarrow{r} \phanArrow{dr} & M_1\oplus_{M_0}^N M^* \incarrow{r} \phanArrow{dr} & M^{**}\\
            M_0 \incarrow{u} \incarrow{r} & M_1 \incarrow{u} \incarrow{r} & M_2 \incarrow{u}
        \end{tikzcd}
    \end{equation*}
    Applying regularity once more, hence $(M^*, M_2)$ witness that $A\dnf_{M_0}^N M_2$.
\end{proof}

\begin{prop}[Invariance]\label{ancinvar}
    If $\scrA$ is a notion of amalgamation, then $\dnf$ is invariant under $K$-embeddings: if $A\dnf_M^N B$ and $f:N\longrightarrow N'$ is a $K$-embedding, then $f(A)\dnf_{f[M]}^{N'} f(B)$.
\end{prop}

\begin{proof}
    First, for the case where $f:N\cong N'$ is a $K$-isomorphism, the statement above holds due to the Invariance properties of $\scrA$. Then Proposition \ref{topmono} shows that this is true for the general case where $f$ is a $K$-embedding.
\end{proof}

\begin{cor}
    If $\bar{a}\dnf_{M_0}^N M_1$ and $\gtp(\bar{a}/M_1, N)=\gtp(\bar{b}/M_1, N')$, then $\bar{b}\dnf_{M_0}^{N'} M_1$.
\end{cor}

The above corollary shows that when $\scrA$ is a notion of amalgamation and $\dnf$ is derived from $\scrA$, then in fact $\dnf$ can be extended to a form of nonforking notion for Galois types.

\begin{nota}
    Let $p\in S(M_1)$. We say that $p$ \textbf{does not fork} over $M_0$ if $M_0\leq M_1$ and there is some $\bar{a}$ and a model $N\geq M_1$ such that $(\bar{a}, M_1, N)$ realize $p$, and $\bar{a}\dnf_{M_0}^N M_1$.
    
    We say that $q\geq p$ is a \textbf{nonforking extension} if $q$ does not fork over $\text{dom } p$
\end{nota}

\begin{cor}
    Suppose $K$ admits finite intersection, and $\scrA$ is regular, absolutely minimal. If $p$ does not fork over $M$ and $q$ is a nonforking extension of $p$, then $q$ does not fork over $M$.
\end{cor}

\begin{prop}[Extension]\label{exten}
    Let $p\in S(M_0)$. If $M_1\geq M_0$, then there is $q\in S(M_1)$ such that $q\geq p$ and $q$ does not fork over $M_0$.
\end{prop}

\begin{proof}
    Let $(\bar{a}, M_0, M_2)$ realize $p$, and let $N$ be an $\scrA$-amalgam of $M_2, M_1$ over $M_0$ via
    \begin{equation*}
        \begin{tikzcd}
            M_2 \arrow[r, "f"] \phanArrow{dr} & N\\
            M_0 \incarrow{u} \incarrow{r} & M_1\incarrow{u}
        \end{tikzcd}
    \end{equation*}
    Then $f(\bar{a})\dnf_{M_0}^N M_1$ (as witnessed by $(f[M_2], M_1)$), and $\gtp(f(\bar{a})/M_0, N)=p$. Hence $\gtp(f(\bar{a})/M_1, N)$ is the desired nonforking extension.
\end{proof}

\begin{prop}[Locality, version 1]\label{chainloc}
    Suppose $\scrA$ is regular, continuous, absolutely minimal and admits decomposition. Assume further that $\mu(K)<\infty$. If $(M_i)_{i<\alpha}$ is a strictly increasing continuous chain of models and $\cf(\alpha)\geq\mu_r(K)$, then for every $p\in S(\bigcup_{i<\alpha} M_i)$ a Galois type of length $<\mu(K)$, there is some $i<\alpha$ such that $p$ does not fork over $M_i$.
\end{prop}

\begin{proof}
    Let $M_b:=M_0, M:=\bigcup_{i<\alpha} M_i$, and let us define a sequence of models $(M_i')_{i<\alpha}$ such that:
    \begin{enumerate}
        \item For each $i<\alpha$, $M_b\leq M_i'\leq M_{i+1}$.
        \item $M_0'=M_1$
        \item For each $i<\alpha$, $M_i'$ is such that $M_i\oplus_{M_b}^M M_i'=M_{i+1}$.
    \end{enumerate}
    Note that $\scrA$ admitting decomposition implies that such a sequence exists. Furthermore, by construction we have that $\bigoplus_{M_b, i<\alpha}^M M_i'=M$ (as witnessed by the resolution $(M_i)_{i<\alpha}$).
    
    Given $p\in S(M)$ a Galois type of length $<\mu(K)$, let $(\bar{a}, M, N)$ realize $p$, and let $N^*\leq N$ be such that $N=M\oplus_{M_b}^N N^*$ (again, $N^*$ exists as $\scrA$ admits decomposition). Hence we also have that $N$ is the $\scrA$-amalgam of $\{N^*\}\cup\{M_i'\}_{i<\alpha}$ over $M_b$ (by inclusion). Since $|\bar{a}|<\mu_r(K)\leq\cf(\alpha)$, there is some $i_0<\alpha$ such that
    \begin{equation*}
        \bar{a}\in N^*\oplus_{M_b}^N\bigg(\bigoplus_{M_b, i<i_0}^N M_i'\bigg)=N^*\oplus_{M_b}^N M_{i_0+1}
    \end{equation*}
    Let $N':=N^*\oplus_{M_b}^N M_{i_0+1}$. Since $N$ is the $\scrA$-amalgam of $N^*, M$ over $M_b$ by inclusion, by regularity we also have that $N$ is the $\scrA$-amalgam of $N', M$ over $M_{i_0+1}$. Diagrammatically,
    \begin{equation*}
        \begin{tikzcd}
            N^*\incarrow{r} \phanArrow{dr} & N\\
            M_b \incarrow{u} \incarrow{r} & M \incarrow{u}
        \end{tikzcd}
        \Longrightarrow
        \begin{tikzcd}
            N^*\incarrow{r} \phanArrow{dr} & N'\incarrow{r} \phanArrow{dr} & N\\
            M_b \incarrow{u} \incarrow{r} & M_{i_0+1} \incarrow{u} \incarrow{r} & M \incarrow{u}
        \end{tikzcd}
    \end{equation*}
    Note then that $(N', M)$ is a witness to $\bar{a}\dnf_{M_{i_0+1}}^N M$, and therefore $p$ does not fork over $M_{i_0+1}$.
\end{proof}

In fact, a related formulation of the locality property can be shown to be true using the same proof:

\begin{prop}[Locality, version 2]\label{ancchar}
    Suppose $\scrA$ is regular, continuous, absolutely minimal, admits decomposition and is such that $\mu(K)<\infty$. If $|M|>\mu_r(K)+\LS(K)$ and $p$ is a Galois type over $M$ of length $<\mu(K)$, then there is some $M^*\leq M$ such that $|M^*|<\mu_r(K)+\LS(K)^+$ and $p$ does not fork over $M^*$.
\end{prop}

\begin{proof}
    Let $\lambda=|M|$, and take some $M_b\leq M$ with $|M_b|=\LS(K)$. As $\scrA$ admits decomposition and is absolutely minimal, by Lemma \ref{amalcardseq} there is a sequence $(M_i)_{i<\alpha}$ such that:
    \begin{enumerate}
        \item $\alpha<\lambda^+$
        \item For each $i<\alpha$, $M_b\leq M_i\leq M$ and $|M_i|=\LS(K)$
        \item $M=\bigoplus_{M_b, i<\alpha}^M M_i$
    \end{enumerate}
    Further, as $\scrA$ is regular and continuous, by Theorem \ref{setamal} we may assume that $\alpha=\lambda$. Letting $(\bar{a}, M, N)$ be a realization of $p$, as in the proof for the above proposition there exists some $N^*$ such that $N=N^*\oplus_{M_b}^N M$. Now, as $|A|<\mu(K)<\infty$, there is some subset $S\subseteq \lambda$ such that $|S|<\mu_r(K)$ and $A\subseteq N^*\oplus_{M_b}^N\big(\bigoplus_{M_b, i\in S}^N M_i\big)$. Letting $N'=N^*\oplus_{M_b}^N\big(\bigoplus_{M_b, i\in S}^N M_i\big)$, hence (as in the above proof) $N$ is the $\scrA$-amalgam of $N', M$ over $\bigoplus_{M_b, i\in S}^N M_i$ by regularity. Therefore, letting $M^*=\bigoplus_{M_b, i\in S}^N M_i$, we have $A\dnf_{M^*}^N M$. Furthermore, as $|S|<\mu_r(K)$ and each $|M_i|=\LS(K)$, by Lemma \ref{amalcard}, $|M^*|<\mu_r(K)+\LS(K)^+$ as desired.
\end{proof}

\begin{cor}\label{anccont}
    Suppose $K$ admits finite intersection, and $\scrA$ is regular, continuous, absolutely minimal, admits decomposition, and is such that $\mu(K)<\infty$. If $M\in K$, $(M_i)_{i<\alpha}$ is continuous resolution of $M$ with $\cf(\alpha)\geq\mu_r(K)$, and $p\in S(M)$ is a type of length $<\mu(K)$ such that each $p\upharpoonright M_i$ does not fork over $M_0$, then $p$ does not fork over $M_0$.
\end{cor}

\begin{proof}
    By the previous proposition, there is some $i<\alpha$ such that $p$ does not fork over $M_i$. But $p\upharpoonright M_i$ does not fork over $M_0$ by assumption, and so by Proposition \ref{transi} $p$ does not fork over $M_0$.
\end{proof}

\begin{prop}[Uniqueness]\label{ancuni}
    Suppose $K$ admits finite intersection, $\scrA$ has uniqueness and is regular. Then for any Galois type $p\in S(M)$ and any $N\geq M$, there is a unique $q\in S(N)$ such that $q$ is a nonforking extension of $p$.
\end{prop}

\begin{proof}
    Let $q_1, q_2\in S(N)$ be nonforking extensions of $p$, and let $(\bar{a}, N, N_1), (\bar{b}, N, N_2)$ be realizations of the types respectively. Since $q_1\upharpoonright M=q_2\upharpoonright M=p$ and $K$ has AP (since $\scrA$ is a notion of amalgamation), we may assume that there is a $K$-isomorphism $f:N_1\longrightarrow_M N_2$ such that $f(\bar{a})=\bar{b}$. Now, as each $q_i$ is a nonforking extension of $p$, there exists $M_1\leq N_1$ such that $(M_1, N)$ is a witness to $\bar{a}\dnf_M^{N_1} N$, and similarly $M_2\leq N_2$. Letting $M''=f[M_1]\cap M_2$, note then that $M\leq M'', \bar{b}\in M''$. Hence, by regularity, $(M'',N)$ is also a witness for $\bar{b}\dnf_M^{N_2} N$. Further, let $M'\leq N_1$ be such that $f[M']=M''$, and similarly $(M', N)$ is a witness for $\bar{a}\dnf_M^{N_1} N$. But as $f$ is also an isomorphism between $M'$ and $M''$ over $M$, by uniqueness of $\scrA$ there is an isomorphism $g$ satisfying the following commutative diagram (where all the unlabelled maps are inclusions):
    \begin{equation*}
        \begin{tikzcd}
            & N\oplus_M^{N_1} M' \arrow[dr, "g"] \arrow[from=dd]\\
            N \arrow[ur] \arrow[rr, crossing over] & & N\oplus_M^{N_2} M''\\
            & M' \arrow[dr, "f"]\\
            M \arrow[uu] \arrow[rr] \arrow[ur] & & M'' \arrow[uu]
        \end{tikzcd}
    \end{equation*}
    In particular, $g(\bar{a})=f(\bar{a})=f(\bar{b})=g(\bar{b})$ and $g[N]=N$, and hence $\gtp(\bar{a}, N, N_1)=\gtp(\bar{b}, N, N_2)$. This completes the proof.
\end{proof}

\begin{cor}
    Suppose $K$ admits finite intersection, $\scrA$ has uniqueness and is regular. If $(M_i)_{i<\alpha}$ is an increasing continuous chain, and $(p_i)_{i<\alpha}$ is an increasing chain of types (with each $p_i\in S(M_i)$) such that each $p_{i+1}$ is a nonforking extension of $p_i$, then there is $p\in S\big(\bigcup_{i<\alpha} M_i\big)$ such that $p\upharpoonright M_i= M_i$ and $p$ does not fork over $M_0$.
\end{cor}

\begin{proof}
    Denote $M_\alpha:=\bigcup_{i<\alpha} M_i$. By Proposition \ref{exten}, let $p\in S(M_\alpha)$ be a nonforking extension of $p_0$. Note then that for each $i<\alpha$, $p\upharpoonright M_i$ also does not fork over $M_0$, and hence is a nonforking extension of $p_0$. By the above proposition, hence $p\upharpoonright M_i=p_i$.
\end{proof}

\begin{defn}
    We say that $\scrA$ is a notion of \textbf{geometric amalgamation} if $\scrA$ is regular, continuous, absolutely minimal, and admits decomposition with $\mu(K)<\infty$. We say that $\scrA$ is a notion of \textbf{free amalgamation} if additionally $\scrA$ has uniqueness.
\end{defn}

\begin{con}
    Suppose $K$ is an AEC admitting finite intersections, $\scrA$ is a notion of free amalgamation on $K$. Then the relation $\dnf$ satisfies:
    \begin{itemize}
        \item Invariance (Proposition \ref{ancinvar})
        \item Existence (Proposition \ref{ancmulti}.\ref{exist})
        \item Symmetry (Proposition \ref{ancmulti}.\ref{sym})
        \item Top monotonicity (Proposition \ref{topmono})
        \item Right (and left) monotonicity (Proposition \ref{ancmulti}.\ref{rmono} and symmetry)
        \item Base monotonicity (Proposition \ref{basemono})
        \item Normality (Proposition \ref{ancmulti}.\ref{normal})
        \item Transitivity (Proposition \ref{transi})
        \item Extension (Proposition \ref{exten})
        \item $\mu(K)$-local character for types of length $<\mu(K)$ (Proposition \ref{chainloc} and \ref{ancchar})
        \item Continuity (Proposition \ref{anccont})
        \item Uniqueness of nonforking extensions (Proposition \ref{ancuni})
    \end{itemize}
\end{con}

This completes the propositions needed to prove that $\dnf$ has the desired properties (under certain assumptions on $K$ and $\scrA$). A nontrivial example of such an independence relation comes from the class of free groups:

\begin{ex}\label{fgex}
    Let $K$ be the class of free groups, with an ordering $\leq_f$ such that $G\leq_f H$ iff $G$ is a free factor in $H$ i.e. there is some set $Y$ such that we can consider $H=F(Y)$ (the free group with $Y$ as the set of generators), and moreover there is some $X\subseteq Y$ such that $G=F(X)$.
    
    Note then that $K$ is a weak AEC which admits finite intersections (see the Appendix for details), and taking $\scrA$ to be the notion of free amalgamation gives us that $\scrA$ is minimal (hence absolutely minimal), regular, continuous, admits decomposition, and has uniqueness. It is also clear that $\mu(K)=\aleph_0$. In this case, $\bar{a}\dnf_F^G \bar{b}$ iff there is a free basis $X$ of $G$ (so $F(X)=G$) along with subsets $X_0, X_1, X_2\subseteq X$ such that:
    \begin{itemize}
        \item $F_0=F(X_0)$
        \item $X_1\cap X_2=X_0$
        \item $\bar{a}\in F(X_1)$ and $\bar{b}\in F(X_2)$
    \end{itemize}
    Moreover, the above lemmas show that $\dnf$ for the class of free groups behaves as if for a superstable first order theory; this is not surprising since by defining superstability in terms of uniqueness of limit models, the uncountable categoricity of the class implies that it is indeed superstable as an AEC.
    
    On the other hand, this example is notable for two reasons:
    \begin{enumerate}
        \item The free factors of a free group are \underline{not} closed under infinite intersections (for an example, see \cite{Bu77}), and in particular the class of free groups do not admit arbitrary intersection. This is in contrast to classes such as vector spaces and algebraically closed fields, where the pregeometry is used to define independence but implies that the class admits intersections.
        \item The first order theory of free groups is known to be not superstable (see, for example, \cite{Po83}), whereas $(K,\leq_f)$ is indeed superstable as an AEC. Furthermore, since $G\leq_f H$ imlies that $G$ is an elementary substructure of $H$ (see the Appendix), this implies that the free groups lies within the superstable part of the theory of free groups. This fact is, of course, trivial given that the free groups are uncountably categorical, but does show how different the class of free groups is from the class of elementarily free groups.
    \end{enumerate}
\end{ex}

Before ending this section, let us demonstrate the known fact that the existence of a superstable-like independence notion implies that the class is tame:

\begin{defn}
    Let $I$ be a linear order. We say that $K$ is \textbf{$(<\lambda)$-tame for $I$-types} if for any model $M$ and $p, q\in S^I(M)$, $p\neq q$ iff there exists some $N\leq M$ such that $|N|<\lambda$ and $p\upharpoonright N\neq q\upharpoonright N$. We say that $K$ is $\lambda$-tame if it is $(<\lambda^+)$-tame.
\end{defn}

\begin{cor}\label{anctame}
    If $K$ admits finite intersection and $\scrA$ is a notion of free amalgamation, then $K$ is $(\mu_r(K)+\LS(K))^+$-tame for types of length $<\mu(K)$.
\end{cor}

\begin{proof}
    Let $M\in K$ with $|M|>\mu_r(K)+\LS(K)$, $p, q\in S(M)$ be types of length $l<\mu(K)$, and let $(\bar{a}, M, N_1), (\bar{b}, M, N_2)$ realize $p, q$ respectively. By Proposition \ref{ancchar}, there is $M_a\leq M$ such that $\bar{a}\dnf_{M_a}^{N_1} M$ and $|M_a|=\mu_r(K)+\LS(K)$. Define $M_b\leq M$ similarly, and (by the L\"{o}wenheim-Skolem property) let $M_0\leq M$ be such that $M_a, M_b\leq M_0$ and $|M_0|=\mu_r(K)+\LS(K)$. Then by Proposition \ref{basemono}, $\bar{a}\dnf_{M_0}^{N_1} M$ and $\bar{b}\dnf_{M_0}^{N_2} M$.
    Now, if $p, q$ are such that $p\upharpoonright M'=q\upharpoonright M'$ for every $M'\leq M$ with $|M'|\leq \mu_r(K)+\LS(K)$, then in particular $p\upharpoonright M_0=q\upharpoonright M_0$. But $p$ is a nonforking extension of $p\upharpoonright M_0$ and similarly $q$, so by Proposition \ref{ancuni}, $p=q$.
\end{proof}

\begin{rem}
    The statement of the the above Corollary begs comparison to a result of Boney that appears as Theorem 3.7 of \cite{Va17}, stating that a pseudouniversal AEC is $(\aleph_0)$-tame.\footnote{The actual result is slightly stronger, but difficult to state here precisely due to small conflicts of notation.} Since pseudouniversality is a strengthening\footnote{To quote \cite{Va17}, the extra requirement is that ``the isomorphism characterizing equality of Galois types is unique".} of admitting intersections with $\mu(K)=\aleph_0$ (when a suitable notion of amalgamation is defined), at first glance our result appears to be comparable. Besides the slightly different cardinal arithmetic, the main difference is that our result here relaxes the intersection requirement to only finite intersections, but at the expense of requiring $\scrA$ to have uniqueness (which, as Section 6 explores, has strong implications regarding the structure of $K$ and is not a trivial assumption).
\end{rem}

\section{Categoricity Transfer Using Unique Amalgams}

Up until this point, we have three primary examples of classes with a notion of free amalgamation which have guided our exploration:
\begin{itemize}
    \item The class of vector spaces over a fixed field with the subspace (equivalently, elementary submodel) ordering
    \item The class of (torsion) divisible groups with the subgroup ordering
    \item The class of free groups with the ``free factor" ordering (see Example \ref{fgex})
\end{itemize}
The key characteristic shared, and indeed the driving intuition for this study, is that such classes have some notion of ``basis" which generates each model. Now, in the case of the class of vector spaces, the eventual categoricity of the class can be derived from the fact that any bijection between two bases extends to an isomorphism between the spanned spaces. An analogous principle clearly holds also for the free groups, and the same argument can be applied more generally to the cases of strongly minimal first order theories and quasiminimal excellent classes with the countable closure property. On the other hand, this does \underline{not} apply to the class of divisible groups, and the torsion divisible groups are not categorical in any cardinal whereas the class of free groups are uncountably categorical. In this sense, we will formalize the intuitive argument above to establish sufficient conditions for a categoricity transfer theorem.

One aspect of the argument above for vector spaces is that two superspaces $V, W$ of $U$ are isomorphic over $U$ if $V/U, W/U$ have the same dimension. Although there is no notion of dimensionality within the current context, we note that the dimension of a vector space only differs from its cardinality for spaces of small dimension. This allows us to formalize the notion of two extensions being ``isomorphic" when they are of sufficiently large cardinality:

\begin{defn}
    Suppose $\scrA$ is a notion of free amalgamation in $K$.
    \begin{enumerate}
        \item We define $\theta(K)=\mu(K)+\LS(K)$
        \item Given $M_b\leq M_t$ and $\alpha$ an ordinal, for a model $N\geq M_b$ we write ``$N\cong M_t^\lambda/M_b$" to indicate that $N=\bigoplus_{M_b, i<\alpha}^N M_i$, where each $M_i\cong_{M_b} M_t$.
        \item We define an equivalence relation $\backsim$ on pairs of models of $K$ by: given $M_1\leq N_1$ and $M_2\leq N_2$, $(N_1, M_1)\backsim (N_2, M_2)$ iff there is a $K$-isomorphism $f:N_1^{\theta(K)}/M_1\cong N_2^{\theta(K)}/M_2$ with $f[M_1]=M_2$.
    \end{enumerate}
\end{defn}

\begin{rem}
    Note that the above definition does not construct $M_t^\alpha/M_b$ as a particular model, but if $N_1, N_2$ are such that both $N_1\cong M_t^\alpha/M_b$ and $N_2\cong M_t^\alpha/M_b$, then in fact $N_1\cong_{M_b} N_2$ by uniqueness of $\scrA$, and hence we may consider $M_t^\alpha/M_b$ as a particular choice of representative inside $K$. In this sense, for any ordinal $0<\beta<\alpha$ we may consider $M_b\leq M_t\leq M_t^\beta/M_b\leq M_t^\alpha/M_b$. In this sense, we extend the notation by defining $M_t^0/M_b=M_b$
\end{rem}

\begin{lem}\label{simlarge}
    Suppose $\scrA$ is a notion of free amalgamation. If $(N_1, M)\backsim (N_2, M)$, then for any $\lambda\geq\theta(K)$, there is a $K$-isomorphism $f:N_1^\lambda/M\cong_M N_2^\lambda/M$.
\end{lem}

\begin{proof}
    Decompose $\lambda=\bigsqcup_{j<\lambda} S_j$ such that each $|S_j|=\theta(K)$. Defining models $N_1^*, N_1^i, N_2^*, N_2^i$ such that $N_l^*=\bigoplus_{M, i<\lambda}^{N_l^*} N_l^i$ for $l=1, 2$, note then that for each $j, j'<\lambda$,
    \begin{equation*}
        \bigoplus_{M, i\in S_j}^{N_1^*} N_1^i\cong N_1^{\theta(K)}/M \cong N_2^{\theta(K)}/M \cong \bigoplus_{M, i\in S_{j'}}^{N_2^*} N_2^i
    \end{equation*}
    So let us define $N_l^{S_j}=\bigoplus_{M, i\in S_j}^{N_l^*} N_1^i$. Then, by applying Theorem \ref{subseq}, we get that $N_l^*=\bigoplus_{M, j<\lambda}^{N_l^*} N_l^{S_j}$. Hence, as $\scrA$ has uniqueness, we get that $N_1^*, N_2^*$ are isomorphic over $M$.
\end{proof}

\begin{defn}
    Given $K$ an AEC, we say that $M\in K$ is a \textbf{prime and minimal model} of $K$ if:
    \begin{enumerate}
        \item For every $N\in K$, there is a $K$-embedding $\iota_N:M\longrightarrow N$; and
        \item For every $K$-embedding $f:N_1\longrightarrow N_2$, $f\circ \iota_{N_1}=\iota_{N_2}$
    \end{enumerate}
    If $K$ has a prime and minimal model, we fix such a model and denote it by $0_K$.
\end{defn}

\begin{thm}\label{simuni}
    Suppose $\scrA$ is a notion of free amalgamation in $K$, and $0_K$ is a prime and minimal model. If $K$ is $\lambda$-categorical in some $\lambda\geq\theta(K)$, then for any $M_1, M_2$ in $K$ with $|M_1|=|M_2|=\LS(K)$, $(M_1, 0_K)\backsim (M_2, 0_K)$.
\end{thm}

\begin{proof}
    Given $M_1, M_2$ and $\lambda$ as above, note that by Lemma \ref{amalcardseq}, $|M_1^\lambda/0_K|=|M_2^\lambda/0_K|=\lambda$. Hence, by $\lambda$-categoricity, there is some $K$-isomorphism $f:M_1^\lambda/0_K \cong M_2^\lambda/0_K$, and moreover $f[0_K]=0_K$ as $0_K$ is prime and minimal. Thus, WLOG we may assume that $N=M_1^\lambda/0_K=M_2^\lambda/0_K$, and in fact that there exists sequence $(M_1^i)_{i<\lambda},(M_2^i)_{i<\lambda}$ such that:
    \begin{enumerate}
        \item For each $i<\lambda$, $M_1^i$ is isomorphic to $M_1$ and $M_2^i$ is isomorphic to $M_2$ (over $0_K$).
        \item Each $0_K\leq M_1^i, M_2^i\leq N$; and
        \item $N=\bigoplus_{0_K, i<\lambda}^N M_1^i=\bigoplus_{0_K, i<\lambda}^N M_2^i$
    \end{enumerate}
    We will construct two sequences of sets $(S_j)_{j<\omega}, (T_j)_{j<\omega}$, satisfying:
    \begin{enumerate}
        \item Each $S_j\subseteq \lambda$ with $|S_j|=\theta(K)$, and similarly for $T_j$
        \item $S_0=T_0=\theta(K)$
        \item $S_j\subseteq S_{j+1}$ and $T_j\subseteq T_{j+1}$
        \item For each $j<\omega$, $\bigoplus_{0_K, i\in T_j}^N M_2^i\leq\bigoplus_{0_K, i\in S_{j+1}}^N M_1^i$; and
        \item For each $j<\omega$, $\bigoplus_{0_K, i\in S_j}^N M_1^i\leq\bigoplus_{0_K, i\in T_{j+1}}^N M_2^i$
    \end{enumerate}
    Let us first show that such a construction is sufficient: defining $S:=\bigcup_{j<\omega} S_j$ and $T:=\bigcup_{j<\omega} T_j$, note that as $\scrA$ is absolutely minimal,
    \begin{equation*}
        \bigoplus_{0_K, i\in S}^N M_1^i=\bigcup_{j<\omega}\Bigg(\bigoplus_{0_K, i\in S_j}^N M_1^i\Bigg)
    \end{equation*}
    The same statement holds for $T$ and $M_2^i$. Hence, by (3) and (4) of the construction above, we have that $\bigoplus_{0_K, i\in S}^N M_1^i=\bigoplus_{0_K, i\in T}^N M_2^i$. But since $|S|=|T|=\theta(K)$, hence we can take $\bigoplus_{0_K, i\in S} M_1^i\cong M_1^{\theta(K)}/0_K$, and therefore $(M_1, 0_K)\backsim (M_2, 0_K)$.
    
    Let us complete the proof by constructing $S_j, T_j$. Given $S_j, T_j$ already defined, consider $M'=\bigoplus_{0_K, i\in T_j}^N M_2^i$: by Lemma \ref{amalcard}, $|M'|=\LS(K)+|T_j|=\theta(K)$, and hence there is $S_{j+1}\subseteq\lambda$ such that $|S_{j+1}|=\theta(K)+\mu_r(K)=\theta(K)$ and $M'\subseteq \bigoplus_{0_K, i\in S_{j+1}}^N M_1^i$. Similarly we can define $T_{j+1}$, and this completes the proof.
\end{proof}

Note that the conclusion of the above theorem holds for the classes of vector spaces and free groups, but not for divisible groups: for example, letting $0_G$ denote the trivial group, it is clear that if $p\neq q$ are primes, then $(\mathbb{Z}(p^\infty),0_G),(\mathbb{Z}(q^\infty),0_G)$ are not $\backsim$ equivalent.

\begin{lem}
    Suppose $\scrA$ is a notion of free amalgamation in $K$. Given models $M_0\leq M_1, M_2$, if $(M_1, M_0)\backsim (M_2, M_0)$, then for any ordinal $\beta$,
    \begin{equation*}
        (M_1^\beta/M_0)\oplus_{M_0}(M_2^{\theta(K)}/M_0)\cong_{M_0} M_1^{|\beta|+\theta(K)}/M_0\cong_{M_0} M_2^{|\beta|+\theta(K)}/M_0
    \end{equation*}
\end{lem}

\begin{proof}
    As $(M_1,M_0)\backsim (M_2, M_0)$, $M_1^{\theta(K)}/M_0\cong_{M_0} M_2^{\theta(K)}/M_0$, and hence
    \begin{equation*}
        (M_1^\beta/M_0)\oplus_{M_0}(M_2^{\theta(K)}/M_0)\cong_{M_0} M_1^{\beta+\theta(K)}/M_0
    \end{equation*}
    Furthermore, by Theorem \ref{setamal} and Lemma \ref{simlarge}, we have that
    \begin{equation*}
        M_1^{\beta+\theta(K)}/M_0\cong_{M_0} M_1^{|\beta|+\theta(K)}/M_0\cong_{M_0} M_2^{|\beta|+\theta(K)}/M_0
    \end{equation*}
\end{proof}

\begin{thm}\label{catmain1}
    Suppose $\scrA$ is a notion of free amalgamation in $K$, and $K$ has a prime and minimal model. If $K$ is $\lambda$-categorical in some $\lambda\geq\theta(K)$, then $K$ is $\kappa$-categorical in every cardinal $\kappa\geq\theta(K)+(2^{\LS(K)})^+$.
\end{thm}

\begin{proof}
    By the previous theorem, for any $M_1, M_2\in K_{\LS(K)}$, $(M_1, 0_K)\backsim (M_2, 0_K)$. Hence by Lemma \ref{simlarge}, it suffices to show that if $M\in K$ and $|M|=\kappa\geq\theta(K)+(2^{\LS(K)})^+$, then $M\cong M'^\kappa/0_K$ for some $M'\in K_{\LS(K)}$.
    
    So given $M\in K$ and $|M|=\kappa$, by Lemma \ref{amalcardseq} we can decompose $M=\bigoplus_{0_K, i<\kappa}^M M_i$ such that each $|M_i|=\LS(K)$. Letting $\Gamma:=\{M_i/\cong:i<\kappa\}$ (where $\cong$ is the equivalence relation of $K$-isomorphism), note that since $|\Gamma|\leq 2^{\LS(K)}$, there is some $P\in\Gamma$ which is realized $\geq\theta(K)+(2^{\LS(K)})^+$ times in the sequence $(M_i)_{i<\kappa}$. For each $Q\in\Gamma$, let us also fix some $M_Q\in\{M_i:i<\kappa\}$ such that $M_Q\vDash Q$. Note that by the previous theorem, for any $Q_1, Q_2\in\Gamma$, $(M_{Q_1}, 0_K)\backsim (M_{Q_2}, 0_K)$.
    
    Defining $S:=\{i\in\kappa: M_i\vDash P\}$, we can decompose $S$ as a disjoint union $S=\bigsqcup_{Q\in\Gamma} S_Q$ which is indexed by $\Gamma$ and such that each $|S_Q|\geq\theta(K)+(2^{\LS(K)})^+$ and is a regular cardinal (possibly except for $S_P$). Now, for each $Q\in\Gamma$, we have that $\bigoplus_{0_K, i\in S_Q}^M M_i\cong M_P^{|S_Q|}/0_K$ as each $i\in S_Q\subseteq S$. Defining $N_{S_Q}=\bigoplus_{0_K, i\in S_Q}^M M_i$, note that as $|S_Q|\geq\theta(K)$ and $(M_P, 0_K)\backsim (M_Q,0_K)$, by Theorem \ref{subseq} there is a sequence $(N_Q^i)_{i<|S_Q|}$ such that $N_{S_Q}=\bigoplus_{0_K, i<|S_Q|}^M N_Q^i$ and such that each $N_Q^i\vDash Q$.
    
    Now, for each $Q\in\Gamma$ such that $Q\neq P$, let $T_Q:=\{i\in\kappa: M_i\vDash Q\}$, and define $N^*_Q:=\bigoplus_{0_K, i\in T_Q}^M M_i$. By Theorems \ref{subseq} and \ref{setamal}, each $N_{S_Q}, N^*_Q$ are $\scrA$-subamalgamated inside $M$ over $0_K$, and so we have that
    \begin{equation*}
        N_{S_Q}\oplus_{0_K}^M N^*_Q=\Bigg(\bigoplus_{0_K, i<|S_Q|}^M N_Q^i\Bigg)\oplus_{0_K}^M\Bigg(\bigoplus_{0_K, i\in T_Q}^M M_i\Bigg)
    \end{equation*}
    In other words, $N_{S_Q}\oplus_{0_K}^M N^*_Q\cong M_Q^{|S_Q|+|T_Q|}/0_K$ by Lemma \ref{simlarge}. In particular, as $(M_Q,0_K)\backsim(M_P,0_K)$, we also have that $N_{S_Q}\oplus_{0_K}^M N^*_Q\cong M_P^{|S_Q|+|T_Q|}/0_K$.
    
    This implies that
    \begin{align*}
        M=\bigoplus_{0_K, i<\kappa}^M M_i&=\Bigg(\bigoplus_{0_K, i\in S}^M M_i\Bigg)\oplus_{0_K}^M\Bigg(\bigoplus_{0_K, Q\neq P}^M\Bigg(\bigoplus_{0_K, i\in T_Q}^M M_i\Bigg)\Bigg)\\
        &=N_{S_P}\oplus_{0_K}^M\Bigg(\bigoplus_{0_K, Q\neq P}^MN_{S_Q}\oplus_{0_K}^M N^*_Q\Bigg)
    \end{align*}
    Since $N_{S_P}\cong M_P^{|S_P|}/0_K$ and $N_{S_Q}\oplus_{0_K}^M N^*_Q\cong M_P^{|S_Q|+|T_Q|}/0_K$, thus we get that $M\cong M_P^\kappa/0_K$. This completes the proof.
\end{proof}

Note that in the above argument, the fact that $\lambda>2^{\LS(K)}$ was used to ensure that $|\Gamma|<\lambda$, and hence some $P\in\Gamma$ is realized by many $M_i$'s. In particular, since each $|M_i|=\LS(K)$, in fact we can bound $|\Gamma|\leq I(K, \LS(K))$, where $I(K, \theta)$ is the number of non-isomorphic models in $K_\theta$. This gives the following strengthening:

\begin{thm}
    Suppose $\scrA$ is a notion of free amalgamation in $K$, and $K$ has a prime and minimal model. If $K$ is $\lambda$-categorical in some $\lambda\geq\theta(K)$, then $K$ is $\kappa$-categorical in every cardinal $\kappa\geq\theta(K)+I(K,\LS(K))^+$.
\end{thm}

This concludes our study of categoricity transfer in the case where there is a prime and minimal model, which for most algebraic examples is the trivial object inside the class. On the other hand, this is a very strong assumption from a model-theoretic point of view; for example, intuitively the class of saturated algebraically closed fields (equivalently, the algebreically closed fields of infinite transcendental degree) should also allow the same argument for categoricity transfer, but the class lacks a prime and minimal model. In order to modify the above argument to work in this case, we need to strengthen the notion of amalgmation with an additional property:

\begin{defn}
    Let $\scrA$ be a notion of amalgamation that is regular and absolutely minimal. We say that $\scrA$ is \textbf{3-monotonic} if the following condition is satisfied:
    Given models $M_0\leq M_1, M_2, M_3\leq N$ such that
    \begin{enumerate}
        \item $M_1, M_2$ are $\scrA$-subamalgamated inside $N$ over $M_0$; and
        \item $N$ is the $\scrA$-amalgam of $M_3, M_1\oplus_{M_0}^N M_2$ over $M_0$ via inclusion
    \end{enumerate}
    Then $N$ is the $\scrA$-amalgam of $M_1\oplus_{M_0}^N M_3, M_2\oplus_{M_0}^N M_3$ over $M_3$.
    
    Diagrammatically, if the following commutative squares are $\scrA$-amalgams:
    \begin{equation*}
        \begin{tikzcd}
            M_1 \incarrow{r} \phanArrow{dr} & M_1\oplus_{M_0}^N M_2 & M_3 \incarrow{r} \phanArrow{dr} & N\\
            M_0 \incarrow{u} \incarrow{r} & M_2 \incarrow{u} & M_0 \incarrow{u} \incarrow{r} & M_1\oplus_{M_0}^N M_2 \incarrow{u}
        \end{tikzcd}
    \end{equation*}
    Then we also have the $\scrA$-amalgam
    \begin{equation*}
        \begin{tikzcd}
            M_1 \oplus_{M_0}^N M_3 \incarrow{r} \phanArrow{dr} & N\\
            M_3 \incarrow{u} \incarrow{r} & M_2 \oplus_{M_0}^N M_3 \incarrow{u}
        \end{tikzcd}
    \end{equation*}
    In particular, these models also form the following commutative diagram (simplifying $M_{ij}:= M_i\oplus_{M_0}^N M_j$ and where all the arrows are inclusion maps), where each face of the cube is an $\scrA$-amalgam:
    \begin{equation*}
        \begin{tikzcd}
            & M_{13} \arrow[rr] \arrow[from=dd] & & N\\
            M_3 \arrow[ru] \arrow[rr, crossing over] & & M_{23} \arrow[ru]\\
            & M_1 \arrow[rr] & & M_{12} \arrow[uu]\\
            M_0 \arrow[uu] \arrow[ru] \arrow[rr] & & M_2 \arrow[uu, crossing over] \arrow[ru]
        \end{tikzcd}
    \end{equation*}
\end{defn}

\begin{lem}\label{3trans}
    Suppose $\scrA$ is regular, continuous, absolutely minimal and 3-monotonic. If $M=\bigoplus_{M_b, i<\alpha}^M M_i$ and $N=N^*\oplus_{M_b}^N M$, then $N=\bigoplus_{N^*, i<\alpha}^N (N^*\oplus_{M_b}^N M_i)$.
\end{lem}

\begin{proof}
    Let $(M'_i)_{i<s(\alpha)}$ be a continuous resolution of $M$ witnessing that $M$ is the $\scrA$-amalgam of $(M_i)_{i<\alpha}$ over $M_b$ by inclusion. As $\scrA$ is absolutely minimal, we have that each $M'_\alpha=\bigoplus_{M_b, i<\alpha}^N M_i$. We will prove the statement by induction on $\alpha$:
    \begin{enumerate}
        \item When $\alpha=1$, the statement is trivially true.
        \item If the $\delta$ is a limit ordinal and the statement is true for all $\alpha<\delta$, then for each $\alpha<\delta$, we have
        \begin{equation*}
            N'_\alpha:=N^*\oplus_{M_b}^N M'_\alpha=N^*\oplus_{M_b}^{N'_\alpha} \Bigg(\bigoplus_{M_b,i<\alpha}^{N'_\alpha} M_i\Bigg)=\bigoplus_{N^*, i<\alpha}^{N'_\alpha} (N^*\oplus_{M_b}^{N'_\alpha} M_i)
        \end{equation*}
        Note that as $\scrA$ is absolutely minimal, we can replace all the superscript $N'_\alpha$ by $N$. As a result, we thus have:
        \begin{enumerate}
            \item $N'_0=N^*\oplus_{M_b}^N M'_0=N^*\oplus_{M_b}^N M_b=N^*$
            \item $N'_1=N^*\oplus_{M_b}^N M_1$
            \item For $1<\alpha<\delta$, $N'_\alpha=\bigoplus_{N^*, i<\alpha}^N (N^*\oplus_{M_b}^N M_i)$
        \end{enumerate}
        Hence, letting $N'_\delta:=\bigcup_{i<\delta} N'_\alpha$, the sequence $(N'_\alpha)_{\alpha<\delta}$ is a witness to
        \begin{equation*}
            N'_\delta=\bigoplus_{N^*, \alpha<\delta}^N (N^*\oplus_{M_b}^N M_i)
        \end{equation*}
        But as $N=N^*\oplus_{M_b}^N \big(\bigcup_{\alpha<\delta} M'_\alpha)$, that $\scrA$ is continuous and absolutely minimal implies that
        \begin{equation*}
            N=\bigcup_{\alpha<\delta} N^*\oplus_{M_b}^N M'_\alpha=\bigcup_{\alpha<\delta} N'_\alpha=N'_\delta
        \end{equation*}
        This completes the proof for the limit step.
        \item If the inductive hypothesis is true for $\alpha$, then we have
        \begin{equation*}
            N^*\oplus_{M_b}^N \Bigg(\bigoplus_{M_b,i<\alpha}^N M_i\Bigg)=\bigoplus_{N^*, i<\alpha}^N (N^*\oplus_{M_b}^N M_i)
        \end{equation*}
        Since $\scrA$ is 3-monotonic, we therefore get the following diagram where each face of the cube is an $\scrA$-amalgam:
        \begin{equation*}
            \begin{tikzcd}
                & N^*\oplus_{M_b}^N M_\alpha \arrow[rr] \arrow[from=dd] & & N\\
                N^* \arrow[ru] \arrow[rr, crossing over] & & \bigoplus_{N^*, i<\alpha}^N (N^*\oplus_{M_b}^N M_i) \arrow[ru]\\
                & M_\alpha \arrow[rr] & & M \arrow[uu]\\
                M_b \arrow[uu] \arrow[ru] \arrow[rr] & & \bigoplus_{M_b,i<\alpha}^N M_i \arrow[uu, crossing over] \arrow[ru]
            \end{tikzcd}
        \end{equation*}
        In particular, the top face thus guarantees that
        \begin{equation*}
            N=(N^*\oplus_{M_b}^N M_\alpha)\oplus_{N^*}^N \Bigg(\bigoplus_{N^*, i<\alpha}^N (N^*\oplus_{M_b}^N M_i)\Bigg)=\bigoplus_{N^*, i<\alpha+1}^N (N^*\oplus_{M_b}^N M_i)
        \end{equation*}
        This completes the successor step of the proof.
    \end{enumerate}
\end{proof}

\begin{cor}\label{orddiff}
    Suppose $\scrA$ is a notion of free amalgamation and is 3-monotonic. For any $M_t\geq M_b$ and ordinals $0<\beta<\alpha$, let $N_1=M_t^\beta/M_b$ and $N_2=M_t^{\beta+1}/M_b$. Then $M_t^\alpha/M_b\cong N_2^{\alpha-\beta}/N_1$.
\end{cor}

\begin{proof}
    Let $M=\bigoplus_{M_b,i<\alpha}^M M'_i$ where each $M'_i\cong_{M_b}M_t$, and hence $M\cong M_t^\alpha/M_b$. Defining $M^*=\bigoplus_{M_b,i<\beta}^M M'_i$, note then that $M^*\cong M_t^\beta/M_b\cong_{M_b} N_1$. Moreover, therefore we have that for each $i$ such that $\beta\leq i<\alpha$, $M^*\oplus_{M_b}^M M'_i\cong M_t^{\beta+1}/M_b\cong_{M_b} N_2$. Hence by the above lemma, we also have that
    \begin{equation*}
        M=M^*\oplus_{M_b}^M \Bigg(\bigoplus_{M_b, i<\alpha-\beta}^M M'_{\beta+i}\Bigg)=\bigoplus_{M^*, i<\alpha-\beta}^M (M^*\oplus_{M_b}^M M'_{\beta+i})\cong N_2^{\alpha-\beta}/N_1
    \end{equation*}
\end{proof}

\begin{thm}\label{simuni2}
    Suppose $\scrA$ is a notion of free amalgamation and is 3-monotonic. Given $M_1\lneq M_2, N_1\lneq N_2$ models of cardinality $\LS(K)$, define $M_b=M_2^{\theta(K)}/M_1, M_t=M_2^{\theta(K)+1}/M_1$ and $N_b, N_t$ likewise. If $K$ is $\lambda$-categorical for some $\lambda>\theta(K)$, then $(M_t,M_b)\backsim(N_t, N_b)$. In particular, $M_b\cong N_b$.
\end{thm}

\begin{proof}
    As before, note that $|N_2^\lambda/N_1|=|M_2^\lambda/M_1|=\lambda$, and hence we can consider $M_2^\lambda/M_1\cong N_2^\lambda/N_1$ by $\lambda$-categoricity. In other words, there is a model $N\in K_\lambda$ and models $(M'_i)_{i<\lambda},(N'_i)_{i<\lambda}$ such that:
    \begin{enumerate}
        \item For each $i<\lambda$, $M_1\leq M'_i\leq N$ and $M'_i\cong_{M_1} M_2$
        \item For each $i<\lambda$, $N_1\leq N'_i\leq N$ and $N'_i\cong_{N_1} N_2$
        \item $N=\bigoplus_{M_1,i<\lambda}^N M'_i=\bigoplus_{N_1,i<\lambda}^N N'_i$
    \end{enumerate}
    First, we will construct sequences of sets $(S_j)_{j<\omega},(T_j)_{j<\omega}$ satisfying:
    \begin{enumerate}
        \item Each $S_j,T_j\subseteq\lambda$, and each $|S_j|,|T_j|=\theta(K)$
        \item $(S_j)_{j<\omega},(T_j)_{j<\omega}$ are increasing sequences of sets
        \item For each $j<\omega$, $\bigoplus_{M_1,i\in S_j}^N M'_i\leq \bigoplus_{N_1,i\in T_{j+1}}^N N'_i$
        \item For each $j<\omega$, $\bigoplus_{N_1,i\in T_j}^N N'_i\leq \bigoplus_{M_1,i\in S_{j+1}}^N M'_i$
    \end{enumerate}
    We will construct these sets by induction:
    \begin{itemize}
        \item Since $|N_1|=\LS(K)$, there is $S_0\subseteq\lambda$ such that $|S_0|=\theta(K)$ and $N_1\leq\bigoplus_{M_1,i\in S_0}^N M'_i$. Similarly we can define $T_0$ such that $M_1\leq\bigoplus_{N_1,i\in T_0}^N N'_i$.
        \item If $T_j$ is defined and $|T_j|=\theta(K)$, then $\bigoplus_{N_1.i\in T_j}^N N'_i$ is of cardinality $\mu(K)+\LS(K)$, and hence there is $S_{j+1}\subseteq\lambda$ such that $|S_{j+1}|=\theta(K)$ and satisfying (3). Similarly we can define $T_{j+1}$ such that (4) is satisfied.
    \end{itemize}
    Letting $S=\bigcup_{j<\omega} S_j$ and $T=\bigcup_{j<\omega}$, note then that $|S|=|T|=\theta(K)$, and therefore we have
    \begin{equation*}
        M_b=M_2^{\theta(K)}/M_1\cong\bigoplus_{M_1,i\in S}^N M'_i=\bigoplus_{N_1,i\in T}^N N'_i\cong N_2^{\theta(K)}/N_1=N_b
    \end{equation*}
    Note that by Theorem \ref{subseq}, we also have that
    \begin{equation*}
        N=\bigoplus_{M_1,i<\lambda}^N M'_i=\Bigg(\bigoplus_{M_1,i\in S}^N M'_i\Bigg)\oplus_{M_1}^N \Bigg(\bigoplus_{M_1,i\notin S}^N M'_i\Bigg)
    \end{equation*}
    So, letting $M^*=\bigoplus_{M_1,i\in S}^N M'_i$, we have by Lemma \ref{3trans} that
    \begin{equation*}
        N=\bigoplus_{M^*,i\notin S}^N (M^*\oplus_{M_1}^N M'_i)
    \end{equation*}
    Furthermore, since $\bigoplus_{M_1,i\in S}^N M'_i=\bigoplus_{N_1,i\in T}^N N'_i$ by construction of $S,T$, we also have that
    \begin{equation*}
        N=\bigoplus_{M^*,i\notin T}^N (M^*\oplus_{M_2}^N N'_i)
    \end{equation*}
    Let us define $M''_i=M^*\oplus_{M_1}^N M'_i$ for $i\notin S$, and note that (by Lemma \ref{amalcard}) we have $|M''_i|=|M^*|+|M'_i|=\theta(K)$. Also, by definition we have that $M''_i\cong M_2^{\theta(K)+1}/M_1=M_t$. Similarly defining $N''_i$ for $i\notin T$, we thus have
    \begin{equation*}
        N=\bigoplus_{M^*,i\notin S}^N M''_i=\bigoplus_{M^*,i\notin T}^N N''_i
    \end{equation*}
    Since $\lambda>\theta(K)=|S|=|T|$, by re-indexing the sequences we may consider
    \begin{equation*}
        N=\bigoplus_{M^*,i<\lambda}^N M''_i=\bigoplus_{M^*,i<\lambda}^N N''_i
    \end{equation*}
    Now, let us define new sequences of sets $(U_k)_{k<\omega},(V_k)_{k<\omega}$ such that
    \begin{enumerate}
        \item For each $k<\omega$, $U_k,V_k\subseteq\lambda$ and $|U_k|=|V_k|=\theta(K)$
        \item $(U_k)_{k<\omega},(V_k)_{k<\omega}$ are increasing sequences of sets
        \item $S_0=T_0=\theta(K)$
        \item For each $k<\omega$, $\bigoplus_{M^*,i\in T_k}^N N''_i\leq\bigoplus_{M^*,i\in S_{k+1}}^N M''_i$
        \item For each $k<\omega$, $\bigoplus_{M^*,i\in S_k}^N M''_i\leq\bigoplus_{M^*,i\in T_{k+1}}^N N''_i$
    \end{enumerate}
    The construction is the same as in Theorem \ref{simuni} and above, using the fact that since each $|U_k|=|V_k|=\theta(K)$, $\bigoplus_{M^*,i\in S_k}^N M''_i,\bigoplus_{M^*,i\in T_k}^N N''_i$ are also of cardinality $\theta(K)$. In particular, if $U=\bigcup_{k<\omega} U_k$ and $V=\bigcup_{k<\omega} V_k$, then we again have that
    \begin{equation*}
        M_t^\theta(K)/M_b\cong\bigoplus_{M^*,i\in U}^N M''_i=\bigoplus_{M^*,i\in V}^N N''_i\cong N_t^\theta(K)/N_b
    \end{equation*}
    This completes the proof.
\end{proof}

\begin{defn}
    Let $K$ be an AEC. We say that $K$ has \textbf{common small models} if for any models $N_1, N_2\in K_{>\LS(K)}$, there is $M_1,M_2\in K_{\LS(K)}$ such that $M_1\leq N_1, M_2\leq N_2$ and $M_1\cong M_2$.
\end{defn}

\begin{rem}
    \begin{enumerate}
        \item If $K$ is $\LS(K)$-categorical, then $K$ has common small models.
        \item If $K$ is $\lambda$-categorical, then $K_{\geq\lambda}$ has common small models
    \end{enumerate}
\end{rem}

\begin{thm}\label{catmain2}
    Suppose $K$ has common small models, and $\scrA$ is a notion of free amalgamation and is 3-monotonic. If $K$ is $\lambda$-categorical for some $\lambda>\theta(K)$, then $K$ is $\kappa$-categorical for any $\kappa>\theta(K)+2^{\LS(K)}$.
\end{thm}

\begin{proof}
    We prove the theorem using a variation of the proof of Theorem \ref{catmain1}.
    \begin{clm}
        Let $N\in K$ with $\kappa:=|N|>\theta(K)+2^{\LS(K)}$. Then for any $M_b\leq N$ with $|M_b|=\LS(K)$, there is $M_t$ such that $M_b\leq M_t\leq N$, $|M_t|=\LS(K)$, and $N\cong N_t^\kappa/N_b$, where $N_b\cong M_t^{\theta(K)}/M_b$ and $N_t\cong M_t^{\theta(K)+1}/M_b$.
        
        By Lemma \ref{amalcardseq}, we can decompose $N=\bigoplus_{M_b,i<\kappa}^N M_i$ where each $|M_i|=\LS(K)$. Letting $\Gamma:=\{M_i/\cong_{M_b}:i<\kappa\}$, for each $P\in\Gamma$ let $S_P:=\{i\in\kappa: M_i\vDash P\}$, and hence in particular $\kappa=\bigsqcup_{P\in\Gamma} S_P$. Note that $|\Gamma|\leq 2^{\LS(K)}+\theta(K)<\kappa$ as $|M_b|=\LS(K)$, and hence there is some $Q\in\Gamma$ such that $|S_Q|>2^{\LS(K)}+\theta(K)$. Additionally, for each $P\in\Gamma$, fix a $M_P\vDash P$.
        
        Let us further decompose $S_Q=T^*\sqcup\bigsqcup_{P\in\Gamma} T_P$ such that $|T^*|=\theta(K)$, and whenever $P\neq Q$, $|T_P|>\theta(K)$ and is regular. Thus by Theorem \ref{subseq} we have that
        \begin{align*}
            N&=\bigoplus_{M_b,i<\kappa}^N M_i=\bigoplus_{M_b, P\in\Gamma}^N \Bigg(\bigoplus_{M_b,i\in S_P}^N M_i\Bigg)\\
            &=\Bigg(\bigoplus_{M_b,i\in T^*}^N M_i\Bigg)\oplus_{M_b}^N\Bigg(\bigoplus_{M_b,i\in T_Q}^N M_i\Bigg)\oplus_{M_b}^N \bigoplus_{M_b,P\neq Q}^N\Bigg(\bigoplus_{M_b, i\in S_P\sqcup T_P}^N M_i\Bigg)
        \end{align*}
        Letting $M^*=\bigoplus_{M_b,i\in T^*}^N M_i$, note that as $T^*\subseteq S_Q$, $M_i\vDash Q$ for each $i\in T^*$, and so $M^*\cong M_Q^{|T^*|}/M_b=M_Q^{\theta(K)}/M_b$. Now, as $\scrA$ is 3-monotonic, by Lemma \ref{3trans}, we have that
        \begin{align*}
            N&=\Bigg(\bigoplus_{M_b,i\in T^*}^N M_i\Bigg)\oplus_{M_b}^N\Bigg(\bigoplus_{M_b,i\in T_Q}^N M_i\Bigg)\oplus_{M_b}^N \bigoplus_{M_b,P\neq Q}^N\Bigg(\bigoplus_{M_b, i\in S_P\sqcup T_P}^N M_i\Bigg)\\
            &=\Bigg(\bigoplus_{M^*,i\in T_Q}^N(M_i\oplus_{M_b}^N M^*)\Bigg)\oplus_{M^*}^N\bigoplus_{M^*,P\neq Q}^N\Bigg(\bigoplus_{M^*,i\in S_P\sqcup T_p}^N (M_i\oplus_{M_b}^N M^*)\Bigg)
        \end{align*}
        So for $i\notin T^*$, let $M'_i:=M_i\oplus_{M_b}^N M^*$. In particular, for any $P\in\Gamma$ and $i\in T_P\subseteq S_Q$, $M'_i\cong M_Q^{\theta(K)+1}/M_b$. Furthermore, by Theorem \ref{simuni2}, for any $P\in\Gamma$, $M^*\cong M_Q^{\theta(K)}/M_b\cong M_P^{\theta(K)}/M_b$, and so in fact for any $P\neq Q$ and $i\in S_P$, $M'_i=M_i\oplus_{M_b}^N M^*\cong M_P^{\theta(K)+1}/M_b$. Letting $N_P:=M_P^{\theta(K)+1}/M_b$, hence by Theorem \ref{simuni2}, for any $P\in\Gamma$, $(N_P,M^*)\backsim(N_Q,M^*)$.  So by Lemma \ref{simlarge}, since for any $P\neq Q$, as $|T_P|>\theta(K)$, we have
        \begin{align*}
            \bigoplus_{M^*,i\in S_P\sqcup T_p}^N M'_i&=\Bigg(\bigoplus_{M^*,i\in S_P}^N M'_i\Bigg)\oplus_{M^*}^N \Bigg(\bigoplus_{M^*,i\in T_P}^N M'_i\Bigg)\\
            &\cong (N_P^{|S_P|}/M^*)\oplus_{M^*} (N_Q^{|T_P|}/M^*)\\
            &\cong N_Q^{|S_P|+|T_P|}/M^*
        \end{align*}
        Substituting this back, we get that
        \begin{equation*}
            N\cong N_Q^{|\kappa-T^*|}/M^*=N_Q^\kappa/M^*
        \end{equation*}
        This proves the claim.
    \end{clm}
    
    So given $M, N\in K_\kappa$ with $\kappa>\theta(K)+2^{\LS(K)}$, since $K$ has common small models, let $M_0\leq M, N_0\leq N$ such that $M_0\cong N_0$. By the above claim, there are models $M_1, M_b, M_t, N_1, N_b, N_t$ such that:
    \begin{enumerate}
        \item $M_0\leq M_1\leq M$ and $N_0\leq N_1\leq N$
        \item $M_b\cong M_1^{\theta(K)}/M_0$ and $N_b\cong N_1^{\theta(K)}/N_0$
        \item $M_t\cong M_1^{\theta(K)+1}/M_0$ and $N_t\cong N_1^{\theta(K)+1}/N_0$
        \item $M\cong M_t^\kappa/M_b$ and $N\cong N_t^\kappa/N_b$
    \end{enumerate}
    Since $K$ is $\lambda$-categorical for some $\lambda>\theta(K)$, by Theorem \ref{simuni2} $(M_t,M_b)\backsim(N_t, N_b)$. Hence by Lemma \ref{simlarge}, $M\cong N$.
\end{proof}

Before ending this section, let us compare our result with other results of categoricity transfer which are relevant to our case:

\begin{fac}[\cite{GrVa}, Theorem 6.3)]\label{tamecat}
    Suppose $K$ is $\LS(K)$-tame with the amalgamation property, joint embedding property, and arbitrary large models. If $K$ is categorical in $\LS(K)$ and $\LS(K)^+$, then $K$ is categorical in all $\lambda\geq\LS(K)$
\end{fac}

\begin{fac}[\cite{Va18}, Corollary 10.9]\label{vaseycat}
    Suppose $K$ is $\LS(K)$-tame, has arbitrary large models, and has primes. If $K$ is categorical in some $\lambda>\LS(K)$, then $K$ is categorical in all $\lambda'>\min(\lambda, \beth_{(2^{\LS(K)})^+})$
\end{fac}

\begin{fac}[\cite{SV18}, Theorem 14.2]\label{svexcellent}
    Let $K$ be an excellent AEC that is categorical in some $\mu>\LS(K)$.
    \begin{enumerate}
        \item There is some $\chi<h(\LS(K))$ such that $K$ is categorical in all $\mu'\geq\min(\mu,\chi)$.\footnote{Recall that $h(\kappa):=\beth_{(2^\kappa)^+}$.}
        \item If $K$ is also categorical in $\LS(K)$, then $K$ is categorical in all $\mu'>\LS(K)$.
    \end{enumerate}
\end{fac}

We note that classes with a notion of free amalgamation are $\mu_r(K)+\LS(K)$-tame (see Lemma \ref{anctame}), and hence Fact \ref{tamecat} is relevant here. On the other hand, many of the algebraic examples we have seen above are not $\LS(K)$-categorical, but we manage to prove categoricity transfer using the additional assumption of a notion of free amalgamation.

With regards to Fact \ref{vaseycat}, we recall from \cite{Va18} that a class which admits (arbitrary) intersections over sets of the form $M\cup\{a\}$ does have primes, and so in particular the result applies to AECs which admit intersection. Now, if the closure operator additionally satisfies the exchange principle (or if a suitable notion of ``independent sets" can be otherwise defined), then it admits a 3-monotonic notion of geometric amalgamation (see also section 7 below). However, this still does not guarantee that the notion of amalgamation has uniqueness, and in this sense the extra assumptions of the exchange principle and uniqueness significantly brings down the cardinality threshold in proving categoricity transfer. On the other hand, the present result is applicable even to classes which do \underline{not} have primes: for example, the class of free groups with free factor ordering.

Finally, regarding Fact \ref{svexcellent}, there are two main points of comparison:
\begin{enumerate}
    \item The relationship between $K$ being excellent and $K$ admitting a notion of free amalgamation is far from clear. Unlike the previous case, the greatest difference here is not regarding uniqueness but rather a sense of dimensionality:
    \begin{itemize}
        \item For $\mathbb{I}$ to be an excellent multidimensional independence relation, it must have $n$-existence and $n$-uniqueness for amalgamation diagrams of all finite dimensions.
        \item For $\scrA$ to be a notion of free amalgamation, it must admit decomposition and have bounded locality i.e. $\mu(K)<\infty$.
    \end{itemize}
    Using first order model theory as an analogy, the proof of Theorem \ref{catmain1} and \ref{catmain2} shows that free amalgamation along with categoricity in a sufficiently large cardinal implies that the class is essentially ``unidimensional", which implies that the class trivially has the NDOP (negation of the Dimensional Order Property). In contrast, the analysis of multidimensional amalgamation in excellence is a natural extension of analysing theories which have the NDOP but are not necessarily as simple as begin unidimensional. On the other hand, our formulation in terms of free amalgamation has also allowed us to prove the anti-structural theorems in the negative case (see Section 6 below), whereas a full main gap theorem from a multidimensional approach has yet to be reached.
    \item The other point of comparison is of course the cardinal bounds present; we believe that this is due much more to the machinery used, and is a reflection of the different level of generality given in the first point.
\end{enumerate}

\section{Amalgamation without Uniqueness, and having many Extensions}

In the previous section, we proved arguably the strongest structural theorem which we could expect for classes with very ``nice" notions of amalgamation. In particular, uniqueness of the notion of amalgamation was necessary to define the model $M_t^\lambda/M_b$, which was central to the argument above. On the other hand, having unique amalgams appears a priori to be a very strong assumption, and hence merits an investigation into when uniqueness can be derived.

The driving intuition here is that if a triple $(M_0, M_1, M_2)$ has two $\scrA$-amalgams which cannot be embedded into each other (w.r.t. to the triple), then by taking $\lambda$-many copies of $M_1$ over $M_0$, we can construct $2^\lambda$-many models which cannot be embedded into each other. However, before we can formalize this argument, we need an additional property to hold for $\scrA$:

\begin{defn}
    Suppose $\scrA$ is a notion of amalgamation and is regular. We say that $\scrA$ has \textbf{weak 3-existence} if: given $M_0\leq M_1, M_2, M_3$, if $M_{ij}$ is a $\scrA$-amalgam of $M_i, M_j$ over $M_0$, then there is a model $N$ which is a $\scrA$-amalgam of $M_3, M_{12}$ over $M_0$ and such that there are $K$-embeddings $f_1, f_2$ making the following diagram commute:
    \begin{equation*}
        \begin{tikzcd}
            & M_{13} \arrow[rr, "f_1"] \arrow[from=dd] & & N\\
            M_3 \arrow[ru] \arrow[rr, crossing over] & & M_{23} \arrow[ru, "f_2"]\\
            & M_1 \arrow[rr] & & M_{12} \arrow[uu]\\
            M_0 \arrow[uu] \arrow[ru] \arrow[rr] & & M_2 \arrow[uu, crossing over] \arrow[ru]
        \end{tikzcd}
    \end{equation*}
\end{defn}

\begin{rem}
    The ``weak" in ``weak 3-existence" indicates that in the above diagram, the commutative square
    \begin{equation*}
        \begin{tikzcd}
            M_{12}\arrow[r, "f_1"] & N\\
            M_3\incarrow{u} \incarrow{r} & M_{23} \arrow[u, "f_2"]
        \end{tikzcd}
    \end{equation*}
    is not necessarily an $\scrA$-amalgam. Note that every other face of the cube is an $\scrA$-amalgam either by assumption or because $\scrA$ is regular. Furthermore, if $\scrA$ is 3-monotonic, then the above commutative square is also necessarily an $\scrA$-amalgam. 
\end{rem}

\begin{lem}
    If $\scrA$ is regular and has uniqueness, then $\scrA$ has weak 3-existence.
\end{lem}

\begin{proof}
    Given $M_0\leq M_1, M_2, M_3$ and $M_{ij}$ an $\scrA$-amalgam of $M_i, M_j$ over $M_0$ by inclusion, let $N$ be an $\scrA$-amalgam of $M_3, M_{12}$ over $M_0$ by inclusion. Note that as $M_0\leq M_1\leq M_{12}$, by regularity there is $N_1\leq N$ such that
    \begin{equation*}
        \begin{tikzcd}
            M_3\incarrow{r} \phanArrow{dr} & N_1 \incarrow{r} \phanArrow{dr} & N\\
            M_0 \incarrow{u} \incarrow{r} & M_1 \incarrow{u} \incarrow{r} & M_{12} \incarrow{u}
        \end{tikzcd}
    \end{equation*}
    But then by uniqueness, there is a $K$-isomorphism $f_1:M_{13}\longrightarrow N_1$ such that $f_1$ is the identity on $M_1\cup M_3$. Defining $f_2$ analogously via $M_2$ and $M_{23}$, this proves the statement.
\end{proof}

\begin{defn}
    Given a triple $(M_0, M_1, M_2)$, we say that it is a \textbf{non-uniqueness triple} if there are models $N_1, N_2$ and $K$-embeddings $f_1, f_2$ such that
    \begin{equation*}
        \begin{tikzcd}
            M_2\arrow[r, "f_1"] \phanArrow{dr} & N_1 & M_2 \arrow[r, "f_2"] \phanArrow{dr} & N_2\\
            M_0 \incarrow{u} \incarrow{r} & M_1 \incarrow{u} & M_0 \incarrow{r} \incarrow{u} & M_1 \incarrow{u}
        \end{tikzcd}
    \end{equation*}
    But there is no $K$-isomorphism $g$ such that the following diagram commutes:
    \begin{equation*}
        \begin{tikzcd}
            & & N_2\\
            M_2 \arrow[r, "f_1"] \arrow[rru, bend left, "f_2"] & N_1 \arrow[ru, "g", symbol=\cong]\\
            M_0 \incarrow{u} \incarrow{r} & M_1 \arrow[u, "\iota"] \arrow[uur, bend right, "\iota"]
        \end{tikzcd}
    \end{equation*}
    We say that the tuple $(N_1, f_1, N_2, f_2)$ \textbf{witnesses} that $(M_0, M_1, M_2)$ is a non-uniqueness triple.
    
    We say that $(M_0, M_1, M_2)$ is a \textbf{uniqueness triple} if it is not a non-uniqueness triple.
\end{defn}

\begin{lem}
    Suppose $\scrA$ is absolutely minimal. If $(M_0, M_1, M_2)$ is a non-uniqueness triple as witnessed by $(N_1, f_1, N_2, f_2)$, then for any $N'\geq N_2$, there is no $K$-embedding $g:N_1\longrightarrow N'$ such that the following diagram commutes:
    \begin{equation*}
        \begin{tikzcd}
            & & N'\\
            M_2 \arrow[r, "f_1"] \arrow[rru, bend left, "f_2"] & N_1 \arrow[ru, "g"]\\
            M_0 \incarrow{u} \incarrow{r} & M_1 \arrow[u, "\iota"] \arrow[uur, bend right, "\iota"]
        \end{tikzcd}
    \end{equation*}
\end{lem}

\begin{proof}
    Let $M_0, M_1, M_2, N_1, N_2, f_1, f_2, N'$ be as above, and assume for a contradiction that there does exist a $K$-embedding $g$ making the above diagram commute. Note then by Invariance of $\scrA$, $g[N_1]\leq N'$ is also an $\scrA$-amalgam of $M_1$ and $(g\circ f_1)[M_2]=f_2[M_2]$ over $M_0$. Since $\scrA$ is absolutely minimal, this implies that $g[N_1]=N_2$, contradicting that $(N_1, f_1, N_2, f_2)$ is a witness to $(M_0, M_1, M_2)$ being a non-uniqueness triple.
\end{proof}

\begin{lem}\label{nonunichain}
    Suppose $\scrA$ is absolutely minimal, regular, and continuous. Let $\delta$ be a limit and $(M_i)_{i\leq\delta},(N_i)_{i<\delta}$ be increasing continuous chains such that 
    \begin{equation*}
        \begin{tikzcd}
            N_0 \incarrow{r} \phanArrow{dr} & N_1 \incarrow{r} \phanArrow{dr} & N_2 \incarrow{r} & \cdots \incarrow{r} & N_i \incarrow{r} \phanArrow{dr} & N_{i+1} \incarrow{r} & \cdots\\
            M_0 \incarrow{u} \incarrow{r} & M_1 \incarrow{u} \incarrow{r} & M_2 \incarrow{u} \incarrow{r} & \cdots \incarrow{r} & M_i \incarrow{u} \incarrow{r} & M_{i+1} \incarrow{u} \incarrow{r} & \cdots
        \end{tikzcd}
    \end{equation*}
    If each $(M_i, N_i, M_{i+1})$ is a uniqueness triple, then $(M_0, N_0, M_\delta)$ is also a uniqueness triple.
\end{lem}

\begin{proof}
    Let $N^1, f_1, N^2, f_2$ be such that
    \begin{equation*}
        \begin{tikzcd}
            N_0 \arrow[r,"f_1"] \phanArrow{dr} & N^1 & N_0 \arrow[r,"f_2"] \phanArrow{dr} & N^2\\
            M_0 \incarrow{u} \incarrow{r} & M_\delta \incarrow{u} & M_b \incarrow{u} \incarrow{r} & M_\delta \incarrow{u}
        \end{tikzcd}
    \end{equation*}
    Inductively, define $N_i^1,g_i^1,N_i^2,g_i^2$ for $i<\delta$ such that:
    \begin{enumerate}
        \item $N_0^1 = f_1[N_0]\leq N^1, N_0^2 = f_2[N_0]\leq N^2$
        \item $g_0^1 = f_1, g_0^2=f_2$
        \item For $k=1,2$, $(N_i^k)_{i<\delta}$ is a continuous resolution of $N^k$
        \item For $k=1,2$ and $i<\delta$, the following is a $\scrA$-amalgam:
        \begin{equation*}
            \begin{tikzcd}
                N_0^k \incarrow{r} \phanArrow{dr} & N_i^k\\
                M_0 \incarrow{u} \incarrow{r} & M_i \incarrow{u}
            \end{tikzcd}
        \end{equation*}
        \item Each $g_i^k:N_i\cong N_i^k$ is an isomorphism such that
        \begin{equation*}
            \begin{tikzcd}
                N_i \arrow[rr] \arrow[dr, "g_i^k"] & & N_{i+1} \arrow[dr, "g_{i+1}^k"]\\
                & N_i^k \arrow[rr] & & N_{i+1}^k\\
                M_i \arrow[uu] \arrow[ru] \arrow[rr] & & M_{i+1} \arrow[uu, crossing over] \arrow[ru]
            \end{tikzcd}
        \end{equation*}
    \end{enumerate}
    Note that taking $g^k = \bigcup_{i<\delta}g_i^k$ and $h=g^2\circ(g^1)^{-1}$ shows that $(N^1,f_1,N^2,f_2)$ is not a witness to non-uniqueness. Proceeding with the induction, note that only the successor step requires verification.
    
    So given $N_i^k, g_i^k$ for $k=1,2$, let $N_{i+1}^k:= N_0^k\oplus_{M_0}^{N^k}M_{i+1}$. Note then that $N_{i+1}^k$ is a $\scrA$-amalgam of $N_i^k, M_{i+1}$ over $M_i$ by inclusion. Furthermore, as $g_i^k:N_i\cong_{M_i} N_i^k$ is an isomorphism and $(M_i, N_i, M_{i+1})$ is a uniqueness triple, hence $g_{i+1}^k$ satisfying (5) exists.
\end{proof}

\begin{lem}
    Suppose $\scrA$ is regular and absolutely minimal. If $(M_0, M_1, M_2)$ is a non-uniqueness triple and $N\geq M_2$, then $(M_0, M_1, N)$ is also a non-uniqueness triple.
\end{lem}

\begin{proof}
    We will show the contrapositive and assume that $(M_0,M_1, N)$ is a uniqueness triple. Let $M_1^*, f_1, M_2^*, f_2$ be two $\scrA$-amalgams such that
    \begin{equation*}
        \begin{tikzcd}
            M_1 \incarrow{r} \phanArrow{dr} & M_*^1 & M_1 \incarrow{r} \phanArrow{dr} & M_*^2\\
            M_0 \incarrow{r} \incarrow{u} & M_2 \arrow[u, "f_1"] & M_0 \incarrow{r} \incarrow{u} & M_2 \arrow[u, "f_2"]
        \end{tikzcd}
    \end{equation*}
    Further, let $N^1, g_1, N^2, g_2$ be $\scrA$-amalgams such that
    \begin{equation*}
        \begin{tikzcd}
            M_*^1 \incarrow{r} \phanArrow{dr} & N^1 & M_*^2 \incarrow{r} \phanArrow{dr} & N^2\\
            M_2 \arrow[u, "f_1"] \incarrow{r} & N \arrow[u, "g_1"] & M_2 \arrow[u, "f_2"] \incarrow{r} & N \arrow[u, "g_2"]
        \end{tikzcd}
    \end{equation*}
    By regularity, each $N^k, g_k$ is an $\scrA$-amalgam of $(M_0, M_1, N)$, and hence by the assumption that this is a uniqueness triple there is an isomorphism $h:N^1\cong N^2$ such that
    \begin{equation*}
        \begin{tikzcd}
            & & N^2\\
            M_1 \arrow[r, "\iota"] \arrow[rru, bend left, "\iota"] & N^1 \arrow[ru, "h", symbol=\cong]\\
            M_0 \incarrow{u} \incarrow{r} & N \arrow[u, "g_1"] \arrow[uur, bend right, "g_2"]
        \end{tikzcd}
    \end{equation*}
    In particular, since $g_k$ extends $f_k$ for both $k=1,2$,
    \[
        h[f_1[M_2]] = (h\circ g_1)[M_2]=g_2[M_2]=f_2[M_2]
    \]
    Now, as $M_*^1$ is an $\scrA$-amalgam of $M_1, f_1[M_2]$ over $M_0$ by inclusion, by invariance $h[M_*^1]$ is an $\scrA$-amalgam of $M_1, f_2[M_2]$ over $M_0$ by inclusion. Hence, by absolute minimality, $h[M_*^1]=M_*^2$, and in particular $h\upharpoonright M_*^1$ is an isomorphism such that
    \begin{equation*}
        \begin{tikzcd}
            & & M_*^2\\
            M_1 \arrow[r, "\iota"] \arrow[rru, bend left, "\iota"] & M_*^1 \arrow[ru, "h\upharpoonright M_*^1", symbol=\cong]\\
            M_0 \incarrow{u} \incarrow{r} & N \arrow[u, "f_1"] \arrow[uur, bend right, "f_2"]
        \end{tikzcd}
    \end{equation*}
    In particular, since $M_*^1, M_*^2$ are two arbitrary $\scrA$-amalgam of $M_1, M_2$ over $M_0$, hence $(M_0, M_1, M_2)$ is also a uniqueness triple.
\end{proof}

\begin{cor}
    Suppose $\scrA$ is absolutely minimal, regular, and continuous. If there is a non-uniqueness triple $(M_0,M_1,M_2)$, then for any $\lambda\geq|M_1|+|M_2|$ there is a non-uniqueness triple $(M_0',M_1',M_2')$ such that
    \[
        |M_0|\leq|M_0'|=|M_1'|\leq|M_1|\leq|M_2'|=\lambda
    \]
\end{cor}

\begin{proof}
    Fix $N,f$ be a $\scrA$-amalgam such that
    \begin{equation*}
        \begin{tikzcd}
            M_2 \incarrow{r} \phanArrow{dr} & N \\
            M_0 \incarrow{u} \incarrow{r} & M_1 \arrow[u, "f"]
        \end{tikzcd}
    \end{equation*}
    Fix $(M^{(i)})_{i<|M_1|}$ be a continuous resolution of $M_1$ with $M^{(0)}=M_0$, and by Lemma \ref{restransfer} let $(N^{(i)})_{i<|M_1|}$ be a continuous resolution of $N$ such that each $N^{(i)}$ is a $\scrA$-amalgam by
    \begin{equation*}
        \begin{tikzcd}
            M_2 \incarrow{r} \phanArrow{dr} & N^{(i)} \\
            M_0 \incarrow{u} \incarrow{r} & M^{(i)} \arrow[u, "f\upharpoonright M^{(i)}", swap]
        \end{tikzcd}
    \end{equation*}
    By Lemma \ref{nonunichain}, there is some $i<|M_1|$ such that $(M^{(i)},M^{(i+1)},N^{(i)})$ is a non-uniqueness triple. Letting $M_0'=M^{(i)}, M_1'=M^{(i+1)}$, and taking $M_2'\geq N^{(i)}$ with $|M_2'|=\lambda$, then the above lemma shows that $(M_0',M_1',M_2')$ is a non-uniqueness triple.
\end{proof}

\begin{thm}\label{nonunimany}
    Suppose $\scrA$ is regular, continuous, absolutely minimal and has weak 3-existence. If $(M_b, M^*, M)$ is a non-uniqueness triple and $p=\gtp(M^*/M_b, M^*)$, then there is $N\geq M$ such that $p$ has $2^{|N|}$-many extensions to $N$.
\end{thm}

\begin{proof}
    Since $(M_b, M^*, M)$ is a non-uniqueness triple, fix $M^0, M^1$ two $\scrA$-amalgams of $M^*, M$ over $M_b$ by inclusion such that there is no $K$-isomorphism from $M^0$ to $M^1$ fixing $M^*\cup M$ pointwise. Define $\lambda:=|M|+\LS(K)$, and let $N$ be such that $N$ is a $\scrA$-amalgam of $(M_i)_{i<\lambda}$ over $M_b$ by inclusion, with isomorphisms $g_i:M_i\cong_{M_b} M$. In particular, this means that there is a continuous resolution $(N_i)_{i<\lambda}$ such that:
    \begin{enumerate}
        \item $N_0=M_b$ and $N_1=M_0$
        \item For each $i<\lambda$
        \begin{equation*}
            \begin{tikzcd}
                    N_i\incarrow{r} \phanArrow{dr} & N_{i+1}\\
                    M_b\incarrow{u} \incarrow{r} & M_i \incarrow{u}
            \end{tikzcd}
        \end{equation*}
    \end{enumerate}
    
    To prove the theorem, for every $\eta\in\leftidx{^\lambda}{2}{}$ we will construct $M_\eta$ a $\scrA$-amalgam of $M^*, N$ over $M_b$, and such that for $\xi\neq\eta$, $\gtp(M^*/N, M_\xi)\neq\gtp(M^*/N, M_\eta)$. So given $\eta\in\leftidx{^\lambda}{2}{}$, let us construct an increasing continuous chain $(M_{\eta,i})_{i<\lambda}$ and embeddings $(h_{\eta,i})_{i<\lambda}$ such that:
    \begin{enumerate}
        \item $M_{\eta, 0}=M^*$ and $h_{\eta,0}=\iota:M_b\hookrightarrow M^*$
        \item $(h_{\eta,i}:N_i\longrightarrow M_{\eta,i})_{i<\lambda}$ is an increasing sequence
        \item For each $i<\lambda$
        \begin{equation*}
            \begin{tikzcd}
                    M^* \incarrow{r} \phanArrow{dr} & M_{\eta,i}\\
                    M_b\incarrow{r} \incarrow{u} & N_i \arrow[u, "h_{\eta,i}"]
            \end{tikzcd}
        \end{equation*}
        \item For each $i<\lambda$
        \begin{equation*}
            \begin{tikzcd}
                    M_{\eta, i} \incarrow{r} \phanArrow{dr} & M_{\eta,i+1}\\
                    N_i\incarrow{r} \arrow[u, "h_{\eta,i}"] & N_{i+1} \arrow[u, "h_{\eta,i+1}"]
            \end{tikzcd}
        \end{equation*}
        \item For each $i<\lambda$, there is a $K$-embedding $f_{\eta, i}$ such that the following diagram commutes:
        \begin{equation*}
            \begin{tikzcd}
                    & M^{\eta(i)} \arrow[from=dd] \arrow[rd, "f_{\eta,i}"] \\
                    M^* \arrow[ru] \arrow[rr, crossing over] & & M_{\eta, i+1}\\
                    & M \arrow[rd, "g_i"] & N_{i+1} \arrow[u, "h_{\eta,i+1}"]\\
                    M_b \arrow[uu] \arrow[rr] \arrow[ru] & & M_i \arrow[u]
            \end{tikzcd}
        \end{equation*}
    \end{enumerate}
    We proceed to construct by induction:
    \begin{itemize}
        \item For $i=0$, define $M_{\eta,0}=M^*$ and $h_{\eta,0}=\iota$ as specified.
        \item For limit $\alpha<\lambda$, let $M_{\eta,\alpha}=\bigcup_{i<\alpha} M_{\eta,i}$, and similarly $h_{\eta, \alpha}=\bigcup_{i<\alpha} h_{\eta, \alpha}$. Note that by (4) and continuity, this implies that
        \begin{equation*}
            \begin{tikzcd}
                    M^* \incarrow{r} \phanArrow{dr} & M_{\eta,\alpha}\\
                    M_b\incarrow{r} \incarrow{u} & N_i \arrow[u, "h_{\eta,\alpha}"]
            \end{tikzcd}
        \end{equation*}
        \item Given $M_{\eta,i}$ and $h_{\eta, i}$ defined, note that we have $\scrA$-amalgams:
        \begin{equation*}
            \begin{tikzcd}
                    N_i \incarrow{r} \phanArrow{dr} & N_{i+1} & M^* \incarrow{r} \phanArrow{dr} & M^{\eta(i)} & M^* \incarrow{r} \phanArrow{dr} & M_{\eta, i} \\
                    M_b \incarrow{u} \incarrow{r} & M \arrow[u, "g_i"] & M_b \incarrow{u} \incarrow {r} & M \incarrow{u} & M_b \incarrow{u} \incarrow{r} & N_i \arrow[u, "h_{\eta, i}"]
            \end{tikzcd}
        \end{equation*}
        Hence, as $\scrA$ has weak 3-existence, there exists a model $M_{\eta, i+1}$ and maps $f_{\eta, i}, h_{\eta, i+1}$ such that $M_{\eta, i+1}$ is an $\scrA$-amalgam of $M^*, N_{i+1}$ over $M_b$ and the following diagram commutes:
        \begin{equation*}
            \begin{tikzcd}
                    & M_{\eta, i} \arrow[from=dd, near start, "h_{\eta, i}"] \arrow[rr] & & M_{\eta, i+1}\\
                    M^* \arrow[ru] \arrow[rr, crossing over] & & M^{\eta(i)} \arrow[ru, "f_{\eta, i}"]\\
                    & N_i\arrow[rr] & & N_{i+1} \arrow[uu, "h_{\eta, i+1}"]\\
                    M_b \arrow[uu] \arrow[ru] \arrow[rr] & & M \arrow[uu, crossing over] \arrow[ru, "g_i"]
            \end{tikzcd}
        \end{equation*}
        In particular, by regularity the following commutative squares are also $\scrA$-amalgams:
        \begin{equation*}
            \begin{tikzcd}
                M^* \incarrow{r} \phanArrow{dr} & M_{\eta, i} \incarrow{r} \phanArrow{dr} & M_{\eta, i+1}\\
                M_b \incarrow{u} \incarrow{r} & N_i \arrow[u, "h_{\eta, i}"] \incarrow{r} & N_{i+1} \arrow[u, "h_{\eta, i+1}"]
            \end{tikzcd}
        \end{equation*}
    \end{itemize}
    Letting $M_\eta:=\bigcup_{i<\lambda} M_{\eta, i}$ and $h_\eta=\bigcup_{i<\lambda} h_{\eta, i}$, note then that $h_\eta$ is a $K$-embedding from $N$ to $M_\eta$ which fixes $M_b$ pointwise.
        
    To complete the proof, it remains to show that for any $\xi\neq\eta$, there is no $M'\geq M_\eta$ and a $K$-embedding $F$ such that the following diagram commutes:
    \begin{equation*}
        \begin{tikzcd}
                & M_\xi \arrow[rr, "F"] \arrow[from=dd] & & M'\\
                M^* \arrow[ru] \arrow[rr, crossing over] & & M_\eta \arrow[ru]\\
                & h_{\xi}[N] \arrow[rd, "h_\eta\circ h_\xi^{-1}"]\\
                M_b \arrow[uu] \arrow[ru] \arrow[rr] & & h_\eta[N] \arrow[uu]
        \end{tikzcd}
    \end{equation*}
    So suppose for a contradiction that such a $M', F$ exists. Since $\xi\neq\eta$, fix $i_0<\lambda$ such that $\xi(i_0)\neq\eta(i_0)$. Assuming WLOG that $\eta(i_0)=0$, by construction of $M_\eta$ we have that
    \begin{equation*}
        \begin{tikzcd}
            & M^0 \arrow[rd, "f_{\eta, i_0}"] \arrow[from=dd]\\
            M^* \arrow[rr, crossing over] \arrow[ru] & & M^*\oplus_{M_b}^{M_\eta} h_{\eta}[M_i]\\
            & M \arrow[uu] \arrow[rd, "g_i"]\\
            M_b \arrow[uu] \arrow[ru] \arrow[rr] & & M_i \arrow[uu, "h_\eta"]
        \end{tikzcd}
    \end{equation*}
    Similarly, since $\xi(i_0)=1$, we have that
    \begin{equation*}
        \begin{tikzcd}
            & M^1 \arrow[rd, "f_{\xi, i_0}"] \arrow[from=dd]\\
            M^* \arrow[rr, crossing over] \arrow[ru] & & M^*\oplus_{M_b}^{M_\xi} h_{\xi}[M_i]\\
            & M \arrow[uu] \arrow[rd, "g_i"]\\
            M_b \arrow[uu] \arrow[ru] \arrow[rr] & & M_i \arrow[uu, "h_\xi"]
        \end{tikzcd}
    \end{equation*}
    But note that as $\scrA$ is absolutely minimal and $F$ is a $K$-embedding,
    \begin{equation*}
        F[M^*\oplus_{M_b}^{M_\xi} h_\xi[M_{i_0}]]=M^*\oplus_{M_b}^{M'} h_\eta[M_{i_0}]=M^*\oplus_{M_b}^{M\eta} h_\eta[M_{i_0}]
    \end{equation*}
    This contradicts that $(M^0, M^1)$ is a witness to $(M_b, M^*, M)$ being a non-uniqueness triple.
\end{proof}

\begin{cor}
    Suppose $\scrA$ is regular, continuous, absolutely minimal, and has weak 3-existence. If $\scrA$ does not have uniqueness, then there is some $M\in K$ and a $p\in S^{\LS(K)+|M|}(M)$ such that for every $\lambda\geq\LS(K)+|M|$, there is $N\in K_\lambda$ such that $N\geq M$ and $p$ has $2^\lambda$ (nonforking) extensions to $N$.
\end{cor}

In particular, if we assume that $K$ is sufficiently type-short and $\scrA$ is a notion of geometric amalgamation with weak 3-existence, then $\scrA$ has uniqueness iff $K$ is $\lambda$-stable on a tail of cardinals.

\section{Classes with Pregeometries and Regular Types}

One last example which we would like to consider is the following: let $T$ be the first order theory in a 2-sorted language, such that models of $T$ are of the form $(V, F)$, where $F$ is a field and $V$ is a vector space over $F$. Whilst $T$ is clearly not categorical in any cardinal, the uncountable categoricity of vector spaces implies that categoricity transfer holds in the subclass where $F$ is fixed. More generally, if we consider the vectors in a model of $T$ to (essentially) realize a regular type, and define the class $K$ where each model consists of the realization of the fixed regular type within a model in $T$, then $K$ also has satisfies some categoricity transfer. In this sense, we wish to prove an analogous result for an AEC with some given notion of independence. This can be seen (essentially) as a case of Zilber's categoricity result for quasiminimal excellent classes from \cite{Zi05} (see also \cite{Ki10} and chapter 2 of \cite{BaCat}), and more specifically as a quasiminimal AEC introduced by Vasey in \cite{Va18q}. However, we will be using the categoricity results of section 5 to provide an alternative proof.

Recall that if $T$ is a stable first order theory, then the realizations of a regular type within a model form a pregeometry (where independence is forking independence). It is hence helpful for us to first investigate how an AEC where each model is a pregeometry admits a notion of amalgamation:

\begin{defn}
    Let $K$ be an AEC. A \textbf{system of pregeometries} for $K$ consists of functions $(\cl_M)_{M\in K}$ such that:
    \begin{enumerate}
        \item For each $M\in K$, $(M, \cl_M)$ is a pregeometry i.e. $\cl_M:\mathscr{P}(M)\longrightarrow\mathscr{P}(M)$ satisfies:
        \begin{enumerate}
            \item For each $X\subseteq M$, $X\subseteq \cl_M(X)=\cl_M(\cl_M(X))$
            \item If $X\subseteq Y$, then $\cl_M(X)\subseteq\cl_M(Y)$
            \item If $a\in\cl_M(X)$, then there exists $X_0\subseteq X$ such that $|X|<\aleph_0$ and $a\in\cl_M(X_0)$
            \item If $b\in\cl_M(A\cup\{a\})-\cl_M(A)$, then $a\in\cl_M(A\cup\{b\})$
        \end{enumerate}
        \item If $M\leq N$, then $\cl_M\subseteq \cl_N$. In particular, $M=\cl_M(M)=\cl_N(M)$
        \item If $B\subseteq N$, $B=\cl_N(B)$, and there exists some $M_0\leq N$ such that $M_0\subseteq\cl_N(B)$, then there is some $M\leq N$ such that $B$ is the universe of $M$.
    \end{enumerate}
    Given $M\in K$ and $B\subseteq M$, we say that $B$ is \textbf{closed} if $B$ is a closed set relative to $\cl_M$. We will similarly use terminology for pregeometries (independent sets, etc.) without explicit references to the ambient model.
\end{defn}

\begin{rem}
    The assumption that each closure operator is finitary is necessary for $K$ to be an AEC: if $\cl_N$ is not finitary, the union of a chain of closed sets might not be closed, and thus $K$ violates the Tarski-Vaught chain axioms. More generally, if each $\cl_N$ has $<\lambda$-character, then $K$ is a $\lambda$-AEC.
\end{rem}

\begin{defn}
    Given $(\cl_M)_{M\in K}$ a system of pregeometries for $K$ and AEC, we define $\scrA$ to be a notion of amalgamation on $K$ by asserting that
    \begin{equation*}
        \begin{tikzcd}
            M_1 \incarrow{r} \phanArrow{dr} & N\\
            M_0 \incarrow{u} \incarrow{r} & M_2 \incarrow{u}
        \end{tikzcd}
    \end{equation*}
    if and only if there is $B_1, B_2\subseteq N$ such that
    \begin{enumerate}
        \item $B_1\cup B_2$ is an independent set and $\cl_N(B_1\cup B_2)=N$
        \item $\cl_N(B_1)=M_1$ and $\cl_N(B_2)=M_2$
        \item $\cl_N(B_1\cap B_2)=M_0$
    \end{enumerate}
\end{defn}

\begin{lem}
     $\scrA$ as defined above is 3-monotonic, absolutely minimal, continuous, and admits decomposition.
\end{lem}

\begin{proof}
    \begin{enumerate}
        \item 3-monotonicity follows straightforwardly from the definition of $\scrA$
        \item For absolute minimality, suppose $N$ is a $\scrA$-amalgam of $M_1, M_2$ over $M_0$ by inclusion. Hence there are $B_1, B_2\subseteq N$ such that $\cl_N(B_1\cup B_2)=N$ and $\cl_N(B_i)=M_i$, and therefore $\cl_N(M_1\cup M_2)=N$. Now, if $N'\geq N$ and $M'\leq N'$ is such that $M_1\cup M_2\subseteq M'$, then
        \begin{equation*}
            N=\cl_N(M_1\cup M_2)=\cl_{N'}(M_1\cup M_2)=\cl_{M'}(M_1\cup M_2)\subseteq M'
        \end{equation*}
        This shows that $\scrA$ is absolutely minimal.
        \item For continuity, suppose $\delta$ is a limit ordinal and there are models $(M_i, N_i)_{i<\delta}$ such that
        \begin{equation*}
            \begin{tikzcd}
                    N_0 \incarrow{r} \phanArrow{dr} & N_1 \incarrow{r} \phanArrow{dr} & N_2 \incarrow{r} & \cdots \incarrow{r} & N_i \incarrow{r} \phanArrow{dr} & N_{i+1} \incarrow{r} & \cdots\\
                    M_0 \incarrow{u} \incarrow{r} & M_1 \incarrow{u} \incarrow{r} & M_2 \incarrow{u} \incarrow{r} & \cdots \incarrow{r} & M_i \incarrow{u} \incarrow{r} & M_{i+1} \incarrow{u} \incarrow{r} & \cdots
            \end{tikzcd}
        \end{equation*}
        Inductively, we will define sets $B, (A_i)_{i<\delta}$ such that:
        \begin{enumerate}
            \item $B\subseteq N_0$ and $A_i\subseteq M_i$
            \item $(A_i)_{i<\delta}$ is an increasing continuous sequence of sets
            \item For each $i<\delta$, $B\cup A_i$ is independent, and $B\cap A_i = A_0$
            \item $\cl_{N_0}(B)=N_0$
            \item For each $i<\delta$, $\cl_{M_i}(A_i)=M_i$
            \item For each $i<\delta$, $\cl_{N_i}(B\cup A_i)=N_i$
        \end{enumerate}
        Note that this is sufficient: letting $A_\delta=\bigcup_{i<\delta}$, $A_\delta$ is a basis for $\bigcup_{i<\delta} M_i$, $B\cup A_\delta$ is independent, and $B\cap A_\delta=A_0$. Moreover, since each $N_i=\cl_{N_i}(B\cup A_i)$, hence $B\cup A_\delta$ is a basis for $\bigcup_{i<\delta} N_i$. Thus the basis $B\cup A_\delta$ witnesses that
        \begin{equation*}
            \begin{tikzcd}
                N_0 \incarrow{r} \phanArrow{dr} & \bigcup_{i<\delta} N_i\\
                M_0 \incarrow{u} \incarrow{r} & \bigcup_{i<\delta} M_i \incarrow{u}
            \end{tikzcd}
        \end{equation*}
        So let us construct the sets $B, (A_i)_{i<\delta}$:
        \begin{itemize}
            \item Since $N_1$ is an $\scrA$-amalgam of $M_1, N_0$ over $M_0$ by inclusion, fix $B, A_1$ a basis of $N_0, M_1$ respectively that witnesses the $\scrA$-amalgam, and let $A_0=B\cap A_1$.
            \item For limit $\alpha$, let $A_\alpha=\bigcup_{i<\alpha} A_i$ as required.
            \item Given $A_i$, by induction $A_i$ is a basis for $M_i$, $B\cup A_i$ is a basis for $N_i$, and $N_{i+1}$ is an $\scrA$-amalgam of $M_{i+1}, N_i$ over $M_i$. By the exchange property, thus there is $A_{i+1}$ a basis of $M_{i+1}$ which extends $A_i$ and such that $B\cup A_{i+1}$ is independent. Moreover, thus $B\cap A_{i+1}\subseteq M_i$, and hence by induction $B\cap A_{i+1}=B\cap A_i=A_0$.
        \end{itemize}
        This completes the proof for continuity.
        \item For decomposability, suppose $M_0\leq M_1\leq N$. Fix $A_0$ a basis of $M_0$. and extend to $A_1$ a basis of $M_1$. Extending further to $B$ a basis for $N$, let $M_2=\cl_N(A_0\cup(B-A_1))$. Then $N$ is an $\scrA$-amalgam of $M_1, M_2$ over $M_0$ by inclusion, as required.
    \end{enumerate}
\end{proof}

\begin{lem}
    $\scrA$ as defined above is regular.
\end{lem}

\begin{proof}
    Recalling the definition of regularity (Definition \ref{propdef}), we shall prove the implications $2\Rightarrow 1\Rightarrow 3\Rightarrow 2$
    \begin{itemize}
        \item ($2\Rightarrow 1$) Suppose that
        \begin{equation*}
            \begin{tikzcd}
                    N_0 \incarrow{r} \phanArrow{dr} & N_1 \incarrow{r} \phanArrow{dr} & N_2\\
                    M_0 \incarrow{r} \incarrow{u} & M_1 \incarrow{r} \incarrow{u} & M_2 \incarrow{u}
            \end{tikzcd}
        \end{equation*}
        Fix independent sets $A_0^1, A_1^1, A_1^2, A_2^2, B_0^1, B_1^1, B_1^2$ such that:
        \begin{enumerate}
            \item $A_0^1$ is a basis for $M_0$, $B_0^1$ is a basis for $N_0$
            \item $A_1^1, A_1^2$ are bases for $M_1$, $B_1^1, B_1^2$ are bases for $N_1$
            \item $A_2^2$ is a basis for $M_2$
            \item $A_0^1=B_0^1\cap A_1^1$, $B_1^1=A_1^1\cup B_0^1$, and $A_1^2=B_1^2\cap A_2^2$
            \item $A_2^2\cup B_1^2$ is a basis for $N_2$
        \end{enumerate}
        By applying the exchange property, we can find $A_2^1$ which extends $A_1^1$ and is a basis for $M_2$. Since $\cl_{N_1}(B_1^1)=\cl_{N_1}(B_1^2)$ and $A_2^2\cup B_1^2$ is independent with $A_2^2\cap B_1^2\subseteq M_1$, hence $B_0^1\cup A_2^1$ is also independent. Hence $N_2$ is an $\scrA$-amalgam of $M_2, N_0$ over $M_0$ by inclusion.
        \item ($1\Rightarrow 3$) Suppose that
        \begin{equation*}
            \begin{tikzcd}
                    N_0 \incarrow{r} \phanArrow{dr} & N_2\\
                    M_0 \incarrow{r} \incarrow{u} & M_2 \incarrow{u}
            \end{tikzcd}
        \end{equation*}
        Further, let $M_1$ be such that $M_0\leq M_1\leq M_2$. Now, as $N_2$ is an $\scrA$-amalgam of $N_0, M_2$ over $M_0$, there is a basis $B$ of $N_2$ such that $B\cap N_0, B\cap M_2, B\cap M_0$ are all bases of the respective models. So extend $B\cap M_0$ to $B_1$, a basis of $M_1$, and note that $B_1\cup(B\cap N_0)$ is still independent as $M_1\leq M_2$. So taking $N_1=\cl_{N_2}(B_1\cup (B\cap N_0))$, we get
        \begin{equation*}
            \begin{tikzcd}
                    N_0 \incarrow{r} \phanArrow{dr} & N_1\\
                    M_0 \incarrow{r} \incarrow{u} & M_1 \incarrow{u}
            \end{tikzcd}
        \end{equation*}
        Furthermore, we can extend $B_1$ to a basis $B_2$ of $M_2$, and still maintain that $B_2\cup (B_1\cup (B\cap N_0))=B_2\cup (B\cap N_0)$ is independent. Hence $N_2$ is also an $\scrA$-amalgam of $M_2, N_1$ over $M_1$. Note that this is sufficient to show $1\Rightarrow 3$, since $\scrA$ being absolutely minimal implies that $N_1$ is the unique $\scrA$-amalgam of $M_1, N_0$ over $M_0$ inside $N_2$.
        \item ($3\Rightarrow 2$) This is trivial.
    \end{itemize}
\end{proof}

\begin{lem}
    For $\scrA$ as defined above, $\mu(K)=\aleph_0$
\end{lem}

\begin{proof}
    This is straightforward from the fact that each closure operator has finite character.
\end{proof}

Since we are interested in types which have $U$-rank 1, we require the class $K$ to admit some suitable notion of nonforking. For this, we use the notions of stable and simple independence given in \cite{GrMa}, which extends earlier work in \cite{BGKV16} and \cite{LRV19}. The reader is encouraged to consult \cite{GrMa} for the relevant definition.

\begin{fac}[\cite{GrMa}, Proposition 5.9]\label{factsepmodel}
    Suppose $N$ is a monster model in $K$, and $\dnf$ is a simple independence relation on $N$. If $A\dnf_M B$, then there is a model $M'\geq M$ such that $B\subseteq M'$ and $A\dnf_M M'$.
\end{fac}


\begin{lem}
    Suppose $N$ is a monster model in $K$, $\dnf$ is a supersimple (in particular, simple) independence relation on $N$ with the $(<\aleph_0)$-witness property for singletons, and $p\in S^1(M_0)$ is a Galois type with $U(p)=1$. Define the operator $\cl_p$ on $p(N)$ by:
    \begin{equation*}
        \cl_p(A):=\{x\in p(N):x\df_{M_0} A\}
    \end{equation*}
    Then $\cl_p$ is a closure operator on $p(N)$, and $(p(N),\cl_p)$ is a pregeometry.
\end{lem}

\begin{proof}
    We first need a claim:
    \begin{clm}
        If $M\geq M_0$ and $x\df_M A$, then for every model $M^*\geq M$ with $A\subseteq M^*$, $x\in M^*$ also.
    \end{clm}
    \begin{proof}
        Otherwise, if $M^*\geq M\geq M_0$ is such that $A\subseteq M^*$ but $x\notin M^*$, note then as $x\in p(N)$, $\gtp(x/M^*, N)$ must be the unique nonalgebraic extension of $p$ to $M^*$, and hence is the nonforking extension of $p$ to $M^*$. Thus by transitivity $x\dnf_M M^*$, contradicting $x\df_M A$.
    \end{proof}
    We can now show the properties required of $\cl_p$:
    \begin{itemize}
        \item $\cl_p$ is monotonic: for every $a\in A$, $a\df_{M_0} A$
        \item $\cl_p$ is idempotent: let $X,y$ be such that $y\df_{M_0} A\cup X$ and for every $x\in X, x\df_{M_0} A$. Suppose for a contradiction that $y\dnf_{M_0} A$, so by Fact \ref{factsepmodel} there is some $M\geq M_0$ such that $A\subseteq M$ and $y\dnf_{M_0} M$. Note that since each $x\df_{M_0} A$, by the above claim $X\subseteq M$, and hence in particular $y\dnf_{M_0} A\cup X$, a contradiction.
        \item $\cl_p$ has finite character: If $x\df_{M_0} A$, then by the $(<\aleph_0)$-witness property there must be a finite $A_0\subseteq A$ such that $x\df_{M_0} A_0$.
        \item $\cl_p$ satisfies the exchange property: Suppose that $x\in\cl_p(A\cup\{b\})-\cl_p(A)$. Hence $x\dnf_{M_0} A$, and $b\dnf_{M_0} A$. Now, let $M$ be a model such that $M_0\leq M$ and $A\subseteq M$: note that since $x\in\cl_p(A\cup\{b\})$, by the above claim, for every model $M'\geq M_0$, if $b\in M'$ then $x\in M'$. In particular, this holds for any $M'\geq M$. But then by the previous facts this implies that $x\df_M b$, and hence by symmetry $b\df_M x$ for any such arbitrary $M$. So assume for a contradiction that $b\notin\cl_p(A\cup\{x\})$, and hence there must be some model $M\geq M_0$ such that $A\cup\{x\}\subseteq M$ but $b\notin M$. This contradicts that $b\df_M x$.
    \end{itemize}
\end{proof}

\begin{rem}
    The assumption that $\dnf$ has the $(<\aleph_0)$-witness property may appear at first to be very strong, but it was shown in \cite{GrMa} (Theorem 7.12 and Corollary 7.16) that having bounded $U$-rank is equivalent to $\dnf$ being supersimple, which for classes with (arbitrary) intersection implies that $\dnf$ does have the $(<\aleph_0)$-witness property. Since the assumption of $U(p)=1$ is necessary for the construction in consideration here, assuming that $\dnf$ does have the $(<\aleph_0)$-witness property does not significantly increase the strength of our assumptions in totality.
\end{rem}

\begin{defn}
    Suppose $K$ is an AEC in a relational language $\tau$ and $\dnf$ is a supersimple independence relation with the $(<\aleph_0)$-witness property for singletons on a monster model $N$ of $K$. Let $p\in S^1(M_0)$ be a Galois type such that $U(p)=1$. We define the abstract class $(K_p, \leq_p)$, where:
    \begin{enumerate}
        \item $\tau(K_p)=\tau_{M_0}=\tau\sqcup\{c_a\}_{a\in M_0}$, where each $c_a$ is a new constant symbol.
        \item A $\tau_{M_0}$ structure $\mathscr{M}$ is a model in $K_p$ iff there is a $\tau_{M_0}$-embedding $f$ from $\mathscr{M}$ into a set $A\cup M_0\subseteq N$, such that:
        \begin{itemize}
            \item $A\subseteq p(N)$ and $A$ is closed with respect to $\dnf_{M_0}$ i.e. if $b\in p(N)$ and $b\df_{M_0} A$, then $b\in A$.
            \item $f(c_a^{\mathscr{M}})=a$
        \end{itemize}
        \item $\mathscr{M}_1\leq_p\mathscr{M}_2$ iff there is a $\tau_{M_0}$-isomorphism $f:\mathscr{M}\longrightarrow p(N)\cup M_0$ such that both $f$ and $f\upharpoonright \mathscr{M}_1$ satisfies the above conditions.
    \end{enumerate}
\end{defn}

\begin{rem}
    Of course, $K_p$ as defined above is not strictly an AEC since all of its models are of bounded cardinality. However, by the lemma below, given some monster model $N'>N$ with a corresponding notion of independence, we can use $N'$ to extend $K_p$, and so in particular $K_p$ as already defined contains all ''small" models.
\end{rem}

\begin{lem}
    $K_p$ is an AEC with a system of pregeometries inherited from $N$, $\LS(K_p)=|M_0|+\LS(K)$, and $M_0$ as a $\tau_{M_0}$ structure is prime and minimal in $K$.
\end{lem}

\begin{proof}
    Having fixed $N$ a monster model of $K$ and $\dnf$ a stable independence relation on $N$, let us first describe the system of pregeometries: for any $M\in K_p$, $M=(A,M_0)$ where there is a $\tau$-embedding $f$ such that $f[A]\subseteq p(N)$, $f\upharpoonright M_0=\text{id}_{M_0}$, and $f[A]$ is closed w.r.t. $\dnf$. We define $\cl_M$ by:
\begin{enumerate}
    \item $\cl_M(\varnothing)=\cl_M(M_0)=M_0$
    \item For any $B$, $\cl_M(B)=\cl_M(B\cup M_0)$
    \item For $B\subseteq A$, $\cl_M(B)=M_0\cup\{x\in A:f(x)\df_{M_0} f[B]\}$
\end{enumerate}
Note that as $f$ is a $\tau$-isomorphism from $A$ to $f[A]$, $\cl_M$ as defined above is independent of the choice of $f$ as $\dnf$ is invariant under $\tau$-automorphisms of $N$. The other conditions for the closure operators to be a system of pregeometries for $K_p$ follows straightforwardly. Moreover, since any $\tau_{M_0}$-embedding must be the identity on $M_0$, $M_0$ is indeed prime and minimal in $K_p$.
\end{proof}

\begin{defn}
    Given $(X, \cl)$ a pregeometry and closed sets $A_0, A_1, A_2\subseteq X$, we say that $A_1, A_2$ \textbf{are independent over} $A_0$ if there are independent sets $B_1, B_2$ such that:
    \begin{itemize}
        \item $\cl(B_1)=A_1, \cl(B_2)=A_2$
        \item $\cl(B_1\cap B_2)=A_0$
        \item $B_1\cup B_2$ is an independent set.
    \end{itemize}
    We say that the pair $(B_1, B_2)$ is a \textbf{witness} to $A_1, A_2$ being independent over $A_0$. Note that if $A_1, A_2$ are independent over $A_0$, then $A_1\cap A_2=A_0$.
\end{defn}

\begin{thm}
    Given $K_p$ as defined above, if $\scrA$ is defined using the system of pregeometries inherited from $N$, then it has uniqueness.
\end{thm}

\begin{proof}
    Since the system of pregeometries of $K_p$ are inherited from the pregeometry $(N, \cl_p)$ and $\scrA$ is defined by independence w.r.t the system of pregeometries, it suffices to prove that:
    \begin{clm}
        Suppose $A_1, A_2$ are closed subsets of $p(N)$ and independent over $A_0$. If $f, g$ are $\tau_{M_0}$-automorphisms of $N$ such that $f\upharpoonright A_0=g\upharpoonright A_0$ and $f[A_1], g[A_2]$ are independent over $f[A_0]$, then there is $h$ a $\tau_{M_0}$-automorphism of $N$ which is an isomorphism between $\cl(A_1\cup A_2)$ and $\cl(f[A_1]\cup g[A_2])$.
    \end{clm}
    So to prove the claim, fix $(B_1, B_2)$ which witnesses that $A_1, A_2$ are independent over $A_0$, and let $B_0:=B_1\cap B_2$. Letting $\lambda=|B_2-B_0|$, fix also an enumeration $B_2-B_0=\{b_i:i<\lambda\}$, and we will construct a sequence $(h_i:i<\lambda)$ such that:
    \begin{enumerate}
        \item Each $h_i$ is a restriction of a $\tau_{M_0}$-automorphism of $N$, and the sequence is an increasing continuous chain
        \item $h_0=f\upharpoonright A_1$
        \item For each $i<\lambda$, $\text{dom }h_i=\cl_p(A_1\cup\{b_j:j<\lambda\})=: A_1^i$
        \item For each $i<\lambda$, $h_i\upharpoonright\cl_p(b_j:j<i)=g\upharpoonright\cl_p(b_j:j<i)$
    \end{enumerate}
    This is sufficient, as letting $h=\bigcup_{i<\lambda} h_i$ gives the desired automorphism. So let us proceed inductively:
    \begin{itemize}
        \item For $i=0$, take $h_0=f\upharpoonright A_1$ as required.
        \item At limit stages, we take the union as required.
        \item If $h_i$ is constructed with $h_i=h^*\upharpoonright A_1^i$ and $A_1^i =\cl_p(A_1\cup\{b_j:j<i\})$ for some $h^*$ a $\tau_{M_0}$-automorphism of $N$, note that as $B_2$ is independent by assumption, $b_i$ is independent from $A_1^i$, and so is $h^*(b_i)$ from $h_i[A_1^i]$. Hence there is some model $M_1$ such that $h_i[A_1^i]\subseteq M_1$ but $h^*(b_i)\notin M_1$. Similarly, $g(b_i)$ is independent from $f[A_1]\cup g[\cl_p(b_j:j<i)]=h_i[A_1^i]$, and we can find a model $M_2$ similarly with $g(b_i)\notin M_2$. Now, let $y\in p(N)$ be such that $y\notin M_1, M_2$: in particular, $y$ is independent from $M_1$ over $M_0$, and as $U(p)=1$ thus $\gtp(y/M_1, N) =\gtp(h^*(b_i)/M_1, N)$. Similarly, $\gtp(y/M_2,N)=\gtp(y/M_2,N)$. Note that since $A_1^i\subseteq M_1\cap M_2$ by construction, this implies that there is some automorphism $h'$ of $N$ such that:
        \begin{itemize}
            \item $h'\upharpoonright A_1^i = h_i$: and
            \item $(h'\circ h^*)(b_i) = g(b_i)$
        \end{itemize}
        So we can take $h_{i+1} = h'\upharpoonright \cl_p(A_1^i\cup\{b_i\})$ (possibly by composing with a suitable automorphism of $N$ to ensure $h_{i+1}\upharpoonright\cl_p(b_j:j<i+1)=g\upharpoonright\cl_p(b_j:j<i+1)$
    \end{itemize}
    This completes the construction, and hence the proof.
\end{proof}

\begin{lem}
    For any $\mathscr{M}_1,\mathscr{M}_2\in K_p$ with $|\mathscr{M}_1|=|\mathscr{M}_2|=|M_0|$, $(\mathscr{M}_1, M_0)\backsim(\mathscr{M}_2, M_0)$
\end{lem}

\begin{proof}
    Note that if $|\mathscr{M}|=|M_0|$, then $\mathscr{M}^{\theta(K_p)}/M_0=\mathscr{M}^{|M_0|}/M_0=(A, M_0)$ where $A$ has dimension $|M_0|$ as a pregeometry. Since $U(p)=1$, if $b_1, b_2$ are both independent from $A$, then there is some $\tau_{M_0}$-automorphism of $N$ which fixes $A$ pointwise but sends $b_1$ to $b_2$. This provides the desired $\tau_{M_0}$-isomorphism between $\mathscr{M}_1^{\theta(K_p)}/M_0$ and $\mathscr{M}_2^{\theta(K_p)}/M_0$.
\end{proof}

\begin{thm}\label{final}
    Suppose $K$ has a monster model and a supersimple independence relation with the $(<\aleph_0)$-witness property for singletons. If $U(p)=1$, then $K_p$ is $\lambda$-categorical in all $\lambda>|\text{dom }p|+\+LS(K)$
\end{thm}

\begin{proof}
    We have shown that $\scrA$ is a notion of free amalgamation for $K_p$, and that $M_0$ is a prime and minimal model for $K_p$. Furthermore, the above lemma establishes that for there is a unique $\backsim$ class for models of cardinality $|M_0|=\LS(K_p)$, so the proof of Theorem \ref{catmain1} also applies here. Furthermore, as stated in Theorem \ref{catmain1}, we can improve the cardinality transfer bound to $\LS(K_p)+I(K_p, \LS(K_p))$; but the above lemma establishes that $I(K_p, \LS(K_p))=\LS(K_p)$, which gives the desired bound.
\end{proof}

\begin{cor}
    For any $M_1, M_2\in K$, if $M_0\leq M_1, M_2$ and $|p(M_1)|=|p(M_2)|>|M_0|$, then $p(M_1)\cong_{M_0} p(M_2)$ as $\tau$-structures.
\end{cor}

\section{Open questions}

There are some questions that arise immediately from our treatment of notions of amalgamation but which we have yet to answer. For example:

\begin{que}
    How do we define an independence relation $\dnf$ from $\scrA$ (as was done in section 4) but without the assumption of finite intersections?
\end{que}

Note that the categoricity result of section 5 does not assume that the class admits finite intersections, and the canonicity of forking established in \cite{BGKV16} implies that there should be a canonical notion of forking which is equivalent to $\dnf$ as defined in section 4 if the class does have FI. This suggest that there should be a ``correct" definition of $\dnf$ using only properties of $\scrA$. Alternatively, it may be the case that having FI follows from having free amalgamation, although we are unaware of any evidence that this should be the case.

\begin{que}
    Suppose that $\scrA$ is absolutely minimal, continuous, and regular. Is $\scrA$ having uniqueness equivalent to $K$ being Galois stable in some $\lambda$?
\end{que}

Both directions require some further work beyond what we have presented here. In the forward direction, note that despite having defined a well-behaving notion of forking in section 4, we assumed that $\scrA$ admits decomposition and in particular this is necessary to establish $\mu(K)$ as the cardinal for local character. A more satisfying method would be to show some form of local character without assuming that $\scrA$ admits decomposition, and in particular without reference to $\mu(K)$.

In the reverse direction, note that Theorem \ref{nonunimany} assumes weak 3-existence. On the other hand, there are simple first order theories for which 3-amalgamation of \underline{types} is not possible; it therefore seems plausible that (assuming tameness and shortness) stability implies that there are no non-uniqueness triples.

\begin{que}
    How can we weaken the property of admitting decomposition? Is admitting decomposition equivalent to being superstable?
\end{que}

A relevant result here is Theorem 4.26 by Mazari-Armida from \cite{Ma21}, which states (in particular) that for a class $K$ of $R$-modules closed under direct sums, $K$ (with the pure-submodule ordering) is superstable iff every module in $K$ is pure-injective. Defining $\scrA$ to be amalgamation by direct sums, we note that every module of $K$ being pure-injective implies that $\scrA$ admits decomposition. Furthermore, it is straightforward to check that for any $K$ and a notion $\scrA$ of free amalgamation on $K$, that $\scrA$ admits decomposition implies that $K$ is superstable (via uniqueness of limit models). On the other hand, the result of \cite{Ma21} depends heavily on a corresponding Galois types with syntactic types, and hence it seems likely that any development in this direction would require at least tameness and shortness.

\begin{que}
    If $K$ is eventually categorical, must $K$ admit a notion of free amalgamation? If $K$ has uniqueness of limit models, does the subclass of limit models (with the ``limit over" ordering) admit a notion of free amalgamation?
\end{que}

This is of course true when $K$ is the elementary class of a countable first order theory based on the Baldwin-Lachlan argument for Morley's categoricity theorem. Hyttinen and Kangas also showed in \cite{HK18} that universal classes which are eventually categorical are essentially either vector spaces or disintegrated, and in either case would admit a notion of free amalgamation. On the other hand, both arguments require two essential steps: finding some ``minimal" type on which a pregeometry can be defined, and showing that every model is prime (in $K$) over their realizations of the minimal type; in particular, for a general AEC $K$ it is not yet clear to us whether or not there exists such a correspondence between models $M\in K$ and the (pre)geometric portions of such models. A simpler question would be whether class of limit models admit free amalgamation under the assumption of uniqueness of limit models; this can be achieved if a universal resolution of $N$ over $M$ can be ``copied" to any $M'$ a limit model over $M$; essentially reducing back to Shelah's construction of $(\lambda,2)$-good sets in \cite{Sh83a}.

\appendix
\section{The class of free groups as a weak AEC}

For this appendix, let $K$ be the class of free groups with the ordering $G\leq_f H$ iff $G$ is a free factor of $H$. We will show in detail that $(K, \leq_f)$ is a weak AEC which admits finite intersection and has a notion of free amalgamation; this follows entirely from Perin's work in \cite{Pe11}, which builds off a series of work by Sela, in particular \cite{Se06a} and \cite{Se06b}.

\begin{nota}
    For any set $X$, we let $F(X)$ denote the free group with $X$ as the set of generators. For any ordinal $\alpha$, we let $F_\alpha$ denote the free group with $\alpha$ (as a set of ordinals) as the set of generators, so that if $\beta<\alpha$, then $F_\beta$ is a subgroup of $F_\alpha$.
    
    We use $\preccurlyeq$ to indicate the relation of being an elementary submodel.
\end{nota}

\begin{fac}[\cite{Pe11}, Theorem 1.3]\label{freefacelem}
    Let $H$ be a proper subgroup of $F_n$, the free group on $n$-generators. Then $H$ is an elementary submodel of $F_n$ iff $H$ is a free factor of $F_n$.
\end{fac}

In particular, if $X=\{x_0,\ldots,x_{n-1}\}$, $Y\subseteq X$, then $F(Y)\preccurlyeq F(X)$. Note that the result as stated only applies when $X$ is finite; however, it is straightforward to see that this implies the same result for free groups of infinite rank:

\begin{lem}
For any ordinal $\alpha$, $F_\alpha\preccurlyeq F_{\alpha+1}$
\end{lem}

\begin{proof}
    By induction on $\alpha$:
    \begin{itemize}
        \item When $\alpha$ is finite, this follows from Fact \ref{freefacelem}.
        \item Suppose the statement holds for $\alpha$, and for $\beta\leq\alpha$ let $G_\beta:=F(\beta\cup\{\alpha+1\})$. By induction, we have that each $F_\beta\preccurlyeq G_\beta$, and hence
        \begin{equation*}
            F_{\alpha+1}=\bigcup_{\beta\leq\alpha}F_\beta\preccurlyeq \bigcup_{\beta\leq\alpha}G_\beta=F(\alpha\cup\{\alpha+1\})=F_{\alpha+2}
        \end{equation*}
    \end{itemize}
\end{proof}

\begin{cor}
    For ordinals $\alpha<\beta$, $F_\alpha\preccurlyeq F_\beta$
\end{cor}

\begin{cor}
    For any sets $X\subseteq Y$, $F(X)\preccurlyeq F(Y)$
\end{cor}

\begin{fac}[Corollary to Kurosh's Subgroup Theorem]\label{kurosh}
    If $F, G$ are free factors of $H$, then $F\cap G$ is a free factor of both $F$ and $G$. In particular, if $F\subseteq G$, then $F$ is a free factor of $G$.
\end{fac}

\begin{cor}
    $K$ admits finite intersection.
\end{cor}

\begin{lem}
    The class $(K,\leq_f)$ is a weak AEC.
\end{lem}

\begin{proof}
    The only property which is not immediate is Coherence. So suppose that $F\leq_f H$, $G\leq_f H$, and $F\subseteq G$. Hence $F, G$ are both free factors of $H$, and so $F\leq_f G$ by Fact \ref{kurosh}.
\end{proof}

\begin{rem}
    It should be noted that $(K, \leq_f)$ is \underline{not} an AEC as it does not satisfy Smoothness, as exemplified by this example from \cite{Bu77}: Let $X=\{x_i:i<\omega\}$, and define $y_i:=x_ix_{i+1}^2$, $G_i:=\langle y_j:j<i\rangle$. Note then that each $G_i$ is a free factor of $F(X)$, but $\bigcup_{i<\omega} G_i=\langle x_ix_{i+1}^2\rangle$ is not a free factor of $F(X)$.
\end{rem}

In $(K, \leq_f)$, we define the notion of amalgamation $\scrA$ to be the group (nonabelian) free amalgamation: the commutative square
\begin{equation*}
    \begin{tikzcd}
        G_1 \incarrow{r} & H\\
        G_0 \incarrow{u} \incarrow{r} & G_2 \incarrow{u}
    \end{tikzcd}
\end{equation*}
is an $\scrA$-amalgam iff there is a set $Y$ with subsets $X_1, X_2\leq Y$ such that $H=F(Y)$, $G_1=F(X_1)$, $G_2=F(X_2)$, and $G_0=F(X_1\cap X_2)$. Equivalently, there exists $G'_1, G'_2$ such that $G_1=G_0\ast G'_1$, $G_2=G_0\ast G'_2$, and $H=G_0\ast G'_1\ast G'_2$.

\begin{lem}
    $\scrA$ is absolutely minimal.
\end{lem}

\begin{proof}
    If $H$ is a $\scrA$-amalgam of $G_1, G_2$ by inclusion over $G_0$, then $H=\langle G_1\cup G_2\rangle$, which is the minimal subgroup containing $G_1, G_2$ in $H$ (and every extension of $H$).
\end{proof}

\begin{lem}
    If $H_1$ is an $\scrA$-amalgam of $G_1, H_0$ over $G_0$ by inclusion, and $X_0, X_1$ are free bases of $G_0, G_1$ respectively such that $X_0\leq X_1$, then there is a set $Y_1\supseteq X_0$ such that $Y_1, Y_1\cup X_1$ are free bases of $H_0, H_1$ respectively.
\end{lem}

\begin{proof}
    Translating to free products of groups, the assumption implies that there are groups $G', H'$ such that:
    \begin{itemize}
        \item $G_1=G_0\ast G'$
        \item $H_0=G_0\ast H'$
        \item $H_1=G_0\ast G'\ast H'$
    \end{itemize}
    Hence, if $Y'$ is any free basis of $H'$, then letting $Y_1=X_0\cup Y'$ gives the desired result.
\end{proof}

\begin{lem}
    $\scrA$ is continuous.
\end{lem}

\begin{proof}
    Given the $\scrA$-amalgams
    \begin{equation*}
        \begin{tikzcd}
            H_0 \incarrow{r} \phanArrow{dr} & H_1 \incarrow{r} & \cdots\\
            G_0 \incarrow{u} \incarrow{r} & G_1 \incarrow{u} \incarrow{r} & \cdots\\
        \end{tikzcd}
    \end{equation*}
    Fix a free basis $X_0$ of $G_0$, and let $Y$ be such that $X_0\cup Y$ is a basis for $H_0$. By the above lemma, we can find $X_1\supseteq X_0$ such that $X_1, X_1\cup Y$ are bases for $G_1, H_1$ respectively. Proceeding by induction, we get that $Y\cup\bigcup_{i}X_i$ is a basis for $\bigcup_{i} H_i$, and hence this is an $\scrA$-amalgam of $H_0, \bigcup_{i} G_i$ over $G_0$ by inclusion.
\end{proof}

\begin{lem}
    $\scrA$ is regular.
\end{lem}

\begin{proof}
    Recall the definition of regularity in Definition \ref{propdef}; We will prove that the three statements are equivalent for $\scrA$.
    \begin{itemize}
        \item $1\Rightarrow 3$: If $H$ is an $\scrA$-amalgam of $G_1, G_2$ over $G_0$ by inclusion, then $H=G_0\ast G'_1\ast G'_2$. Now, if $G_*$ is such that $G_0\leq_f G^*\leq_f G_1$, then there is some $G'_*\leq_f G'_1$ such that $G_*=G_0\ast G'_*$. Thus we have that
        \begin{equation*}
            \begin{tikzcd}
                G_2 \incarrow{r} \phanArrow{dr} & G_0\ast G'_2\ast G'_* \incarrow{r} \phanArrow{dr} & H\\
                G_0 \incarrow{u} \incarrow{r} & G_* \incarrow{u} \incarrow{r} & G_1 \incarrow{u}
            \end{tikzcd}
        \end{equation*}
        \item $2\Rightarrow 1$: Assume that
        \begin{equation*}
            \begin{tikzcd}
                H_0 \incarrow{r} \phanArrow{dr} & H_1 \incarrow{r} \phanArrow{dr} & H_2\\
                G_0 \incarrow{u} \incarrow{r} & G_1 \incarrow{u} \incarrow{r} & G_2 \incarrow{u}
            \end{tikzcd}
        \end{equation*}
        Hence we have that $H_1=G_0\ast G'_1\ast H'$ and $H_2=G_1\ast G'_2\ast H'=G_0\ast G'_1\ast G'_2\ast H$. So $H_2$ is indeed an $\scrA$-amalgam of $G_2, H_0$ over $G_0$ by inclusion.
        \item $2\Rightarrow 3$: This is straightforward.
    \end{itemize}
\end{proof}

\begin{lem}
    $\scrA$ admits decomposition, and $\mu(K)=\aleph_0$
\end{lem}

\begin{proof}
    If $G_0\leq_f G_1\leq_f G_2$, then there is some $G'$ such that $G_2=G_1\ast G'$, and hence $G_2$ is the $\scrA$-amalgam of $G_1, G_0\ast G'$ over $G_0$ by inclusion. That $\mu(K)=\aleph_0$ is equivalent to the fact that all words in a free group are of finite length.
\end{proof}

\begin{lem}
    $\scrA$ has uniqueness.
\end{lem}

\begin{proof}
    This is straightforward from the fact that free amalgamation is a pushout in the category of groups.
\end{proof}

\begin{cor}
    $\scrA$ is a notion of free amalgamation on $(K,\leq_f)$
\end{cor}

\printbibliography

\end{document}